\documentclass[11pt,reqno]{amsart}

\usepackage{mathrsfs}
\usepackage{amsfonts}
\usepackage{graphicx,epsfig,amsmath,amsfonts}
\usepackage{color}
\usepackage{latexsym,amssymb}
\usepackage{lineno,version}
\usepackage{multirow}
\usepackage{bm}
\usepackage{float}
\usepackage{subfigure}

\usepackage [latin1] {inputenc}

\usepackage{bm}

\makeatletter
\renewcommand{\@thesubfigure}{\hskip\subfiglabelskip}
\makeatother
\usepackage{subeqnarray}
\usepackage{cases}
\usepackage{color}
\usepackage{caption}
\newtheorem{theorem}{\hskip\parindent\bf Theorem}[section]

\newtheorem{remark}{\hskip\parindent\bf Remark}[section]
\newtheorem{example}{\hskip\parindent\bf Example}[section]
\newtheorem{algorithm}{\hskip\parindent\bf Algorithm}[section]

\def\bc{\begin{center}}
\def\ec{\end{center}}
\numberwithin{equation}{section}{}

\newcommand{\ba}{\begin{array}}\newcommand{\ea}{\end{array}}
\newcommand{\be}{\begin{eqnarray}}\newcommand{\ee}{\end{eqnarray}}
\newcommand{\bex}{\begin{eqnarray*}}
\newcommand{\eex}{\end{eqnarray*}}

\begin{document}
\graphicspath{{figure/},}

\title[Numerical schemes for phase field dendritic crystal growth models]
{An efficient numerical method for the anisotropic phase field dendritic crystal growth model$^*$}
\author{Yayu Guo$^{1}$
\quad
Mejdi  Aza\"\i ez$^{1,2}$
\quad
Chuanju Xu$^{1,3}$}

\thanks{\hskip -12pt
${}^*$This research is partially supported by
NSFC grant 11971408.\\
${}^{1}$School of Mathematical Sciences and
Fujian Provincial Key Laboratory of Mathematical Modeling and High Performance
Scientific Computing, Xiamen
University, 361005 Xiamen, China. \\
${}^{2}$Universit\'e de Bordeaux, I2M, (UMR CNRS 5295), 33400 Talence, France.\\
${}^{3}$Corresponding author.\\
Emails:
yayuguo@stu.xmu.edu.cn (Yayu Guo);
azaiez@bordeaux-inp.fr (Mejdi Aza\"\i ez); cjxu@xmu.edu.cn (Chuanju Xu)}


\date {\today}
\maketitle

\begin{abstract}
In this paper,
we propose and analyze an efficient numerical method
for the anisotropic phase field dendritic crystal
growth model, which is challenging because we are facing
the nonlinear coupling and anisotropic coefficient in the model.
The proposed method is
a two-step scheme. In the first step, an intermediate solution is computed by
using BDF schemes of order up to three for both the phase-field and heat equations.
In the second step the intermediate solution is stabilized by multiplying
an auxiliary variable. The key of the second step is to stabilize the overall scheme
while maintaining the convergence order of the stabilized solution.
In order to overcome the difficulty caused by the
gradient-dependent anisotropic coefficient and the nonlinear terms,
some stabilization terms are added to the BDF schemes in the first step.
The second step makes use of a generalized auxiliary variable approach
with relaxation.
The Fourier spectral method is applied for the spatial discretization.
Our analysis shows that the proposed scheme is unconditionally stable and
has accuracy in time up to third order.
We also provide a sophisticated implementation
showing that the computational complexity of our schemes is equivalent to
solving two linear equations and some algebraic equations.
To the best of our knowledge, this is the cheapest unconditionally stable schemes
reported in the literature.
Some numerical examples are given to verify the efficiency of the proposed
method.
\end{abstract}

\section{Introduction}

Dendritic crystal growth is a process that occurs when a supercooled liquid solidifies into a crystalline solid, and the solidification front develops into a branched, treelike structure known as a dendrite. The dendritic crystal growth model is a theoretical framework that describes the growth and morphology of dendritic crystals. The model is based on the assumption that the dendritic crystal grows through the attachment of individual atoms or molecules to the solid-liquid interface. The growth rate and branching of the dendrite depend on the local conditions of temperature, concentration, and fluid flow. The dendritic crystal growth model takes into account the effects of diffusion, convection, and thermal gradients on the growth process.
One of most classical dendritic crystal growth models is
the Mullins-Sekerka model, which was proposed in 1963 by John Mullins and Robert Sekerka. The model describes the growth of a single dendrite in a homogeneous, isotropic medium under conditions of zero gravity and no fluid flow.
On the other hand, the phase field model is a more complex and versatile model that can describe the dynamics of multiple dendrites in a heterogeneous and anisotropic medium under the influence of fluid flow and gravity. The phase field model uses a mathematical field to represent the solid and liquid phases and their interface, and the evolution of this field is governed by a set of partial differential equations.
The model takes into account various physical phenomena that affect crystal growth, such as diffusion, advection, surface energy, and elastic deformation.
Compared to classical models, the phase field model is more computationally intensive and requires more complex numerical techniques to solve the governing equations. However, the phase field model is more versatile and can be used to simulate a wider range of dendritic growth phenomena
\cite{chalmers1970principles,glicksman1989fundamentals,fix1982phase,caginalp1986analysis,kim1999universal}.

Numerical simulations play a crucial role in the study of dendritic crystal growth, which allow us to understand the complex dynamics in a quantitative and detailed way. Dendritic crystal growth is a highly nonlinear process that depends on many variables, including temperature gradients, solute concentration, and surface energy. Numerical simulations provide a way to systematically explore the effect of these variables on the growth behavior of dendritic crystals; see,
e.g., \cite{kobayashi1993modeling,warren1995prediction,karma1999phase,ramirez2004phase,ferreira2011numerical,shah2014numerical,zhao2019linear}.

Energy dissipation is an important physical process that governs the dynamics of dendritic crystal growth models.
The difficulties in numerically solving dendritic crystal growth models stem from two
facts. First, it is desirable to construct numerical schemes that preserve the energy dissipation of the models. Secondly, we want numerical schemes to be linear,
decoupled, and highly stable so that they can be efficiently implemented.
The nature of the dendritic crystal growth models, i.e.,
strong nonlinearity, anisotropic coefficient, and coupling of different variables
make the task challenging. There exist several numerical schemes developed for dendritic crystal growth models. Earlier work includes the operator splitting method proposed in \cite{li2011fast,li2012phase}
without analysis of discrete energy stability.
\cite{zhao2017numerical} considered the dendrite growth model in the isothermal case,
and proposed a linear, decoupled numerical scheme based on the invariant energy quadratization approach (IEQ). Later on, a scheme combining IEQ and stabilization was proposed in \cite{yang2019efficient}, but the phase field and the temperature is coupled in the scheme.
The first fully decoupled method proposed in \cite{zhang2019novel}, but only first order accurate.
There also exist numerical schemes based on the so-called scalar auxiliary variable (SAV) approach \cite{shen2018scalar}. \cite{yang2020efficient} proposed the decoupled scheme by using the SAV technique, but the energy stability of the second-order scheme was not established.
Recently, Yang \cite{yang2021fully,yang2021novel} proposed a SAV-based second order scheme
which is decoupled and unconditionally energy stable.
\cite{li2022new} proposed an improvement to reduce the number (from five to four) of the equations to be solved at each time step.
A parallel algorithm using direction splitting method and the stabilization technique was proposed in \cite{wang2022accurate}, which can be fourth order accurate.
But the law of energy dissipation was not given.
Notice that
a second-order unconditionally stable method was constructed in \cite{li2023second} for the anisotropic dendritic crystal growth model with an orientation-field.

The goal of the our paper is to propose and analyze a more efficient scheme for a phase field dendritic crystal growth model given in \cite{karma1998quantitative}.
The model is composed of the Allen-Cahn type equation
with gradient dependence anisotropy coefficient and heat transfer equation.
The proposed scheme makes use of a generalized auxiliary variable approach with relaxation
\cite{zhang2022generalized}, and the $k$-order backward difference method for the
temporal discretization. For the spatial discretization, we consider a Fourier spectral method.

The main contributions of the paper are as follows:

\begin{itemize}
  \item We extend the auxiliary variable approach, proposed in \cite{zhang2022generalized} for single dissipative equations,
  to our dendritic crystal growth model, which is the coupling of a nonlinear
  phase field equation and a heat equation. As it is emphasized above, the strong nonlinearity, anisotropic coefficient, and strong interaction between the phase field and the temperature
make the extension non-trivial;

  \item Compared with the existing schemes for the dendritic crystal growth model,
the method proposed in the current paper has higher order convergence, fully decoupled, and
unconditionally energy stable.
In particular, compared with the most recent method in \cite{li2022new} which requires solving
four linear elliptic equations, our method only needs to solve two linear equations at each time step;

  \item The use of linear stabilizers in the schemes
helps in balancing the anisotropic coefficient and the nonlinear term.
As we will see in the numerical experiments, this allows using relatively larger time step sizes
to accurately capture the nonlinear dendritic crystal growth process.
\end{itemize}

The rest of the paper is organized as follows. In Section 2, we briefly describe the phase field dendrite crystal model and its energy dissipation properties.
In Section 3, we construct our schemes and prove the energy stability of the proposed schemes.
In Section 4, a series of numerical examples are provided to verify the accuracy and demonstrate
the effectiveness of the numerical method.
Some concluding remarks are given in the final section.

\section{The phase field dendrite crystal growth model}

We consider the phase field dendrite crystal model proposed in \cite{karma1998quantitative}. We add a constant $1$ to the original energy functional.
Let $\Omega$ be a smooth, open, bounded and connected domain in $\mathbb{R}^d$ with $d=2,3$.  Consider the energy functional
\begin{equation}\label{energy}
E(\phi, u)=\int_{\Omega}\left(\frac{1}{2}|m(\nabla \phi) \nabla \phi|^2+\frac{\lambda}{2 \varepsilon K} u^2+ F(\phi)\right) d \mathbf{x}+1,
\end{equation}
where $\phi$ stands for the phase field, with $\phi=1$ for the solid and $\phi=-1$ for the fluid, $u$ denotes the temperature field, $m(\nabla \phi)$ is the anisotropic coefficient, $\varepsilon$, $\lambda$, and $K$ are positive constants. $F(\phi)=\frac{1}{4 \varepsilon^2}(\phi^2-1)^2$, $f(\phi)=F'(\phi)$.
The
anisotropic nature is described by the nonlinear coefficient $m(\nabla \phi)$, which is a function
that depends on the direction of the outer normal vector ${\bf{n}}$, i.e.,
${\bf{n}}=\frac{\nabla \phi}{|\nabla \phi|}$.
In 2D, the anisotropic coefficient is determined by
\bex
m(\nabla\phi)=1+\sigma\cos(\beta\theta),
\eex
where $\beta$ is the number of folds of anisotropy,
$\sigma\in [0,1)$ is the anisotropy strength, $\theta=\arctan(\frac{\phi_y}{\phi_x})$.
In the case $\sigma=0$, $m(\nabla \phi)$ is a constant, and
the surface energy becomes isotropic. Taking the example of 4-folds anisotropy, i.e.,
$\beta=4$, then we have (see, e.g., \cite{karma1998quantitative,karma1999phase}):
\begin{equation}\label{1.2a}
m(\nabla\phi)=(1-3\sigma)\Big(1+\frac{4\sigma}{1-3\sigma}\frac{\phi_x^4+\phi_y^4}{|\nabla\phi|^4}\Big).
\end{equation}
In 3D case, $m(\nabla\phi)$ is defined \cite{li2011fast} as
\begin{equation}\label{1.2b}
m(\nabla\phi)=(1-3\sigma)\Big(1+\frac{4\sigma}{1-3\sigma}\frac{\phi_x^4+\phi_y^4+\phi_z^4}{|\nabla\phi|^4}\Big).
\end{equation}

\begin{remark} \label{rm0}
 It is notable that an artificial constant 1 term is added in the definition of the energy functional \eqref{energy}, compared with the original definition in the literature.
This constant 1 term does not change the corresponding gradient flow model of course,
but makes the energy bigger than 1, which
allows simplifying the construction and analysis of the schemes as we will see in the following.
\end{remark}

By adopting a relaxation dynamics,
the governing equations of the phase field dendrite crystal growth reads:
\begin{subequations}\label{2.1}
\begin{align}
\tau\phi_t
=& - \frac{\delta E}{\delta\phi} -4\lambda\varepsilon F(\phi)u, \label{2.1a}\\
u_t=&D\Delta u+4\varepsilon^2KF(\phi)\phi_t, \label{2.1b}
\end{align}
\end{subequations}
where $D$ is the diffusion rate of the temperature,
$\frac{\delta E}{\delta\phi}$ is the gradient of the energy functional $E(\phi, u)$
with respect to $\phi$ in $L^2$ space:
\bex
\frac{\delta E}{\delta\phi}=-\nabla \cdot(m^2(\nabla \phi) \nabla \phi+|\nabla \phi|^2 m(\nabla \phi) \boldsymbol{H}(\phi)) + f(\phi),
\eex
where $\boldsymbol{H}(\phi)$ is the variational derivative of $m(\nabla\phi)$.
For $m=4$, according to \eqref{1.2a} and \eqref{1.2b}, $\boldsymbol{H}(\phi)$ reads
\bex
 \boldsymbol{H}(\phi):=\left\{
             \begin{array}{ll}
4\sigma\frac{4}{|\nabla\phi|^6}(\phi_x(\phi_x^2\phi_y^2-\phi_y^4),\phi_y(\phi_x^2\phi_y^2-\phi_x^4)),   & \text{in 2D},\\
4\sigma\frac{4}{|\nabla\phi|^6}(\phi_x(\phi_x^2\phi_y^2+\phi_x^2\phi_z^2-\phi_y^4-\phi_z^4),\phi_y(\phi_y^2\phi_z^2+\phi_x^2\phi_y^2-\phi_x^4-\phi_z^4),
 \\ \phi_z(\phi_x^2\phi_z^2+\phi_y^2\phi_z^2-\phi_x^4-\phi_y^4)),    & \text{in 3D}.
             \end{array}
           \right.
\eex
To avoid the complexity of integrating on the boundary,
it is common to consider the periodic boundary condition or the homogeneous Neumann condition.
Under these conditions,
taking the $L^2$ inner product of the system \eqref{2.1a} and \eqref{2.1b} with $\phi_t$ and $\frac{\lambda}{\varepsilon K}u$ respectively, we obtain the energy dissipation laws as follows:
\begin{equation}\label{1.5}
\frac{dE(\phi,u)}{dt}=-\|\sqrt{\tau}\phi_t\|_0^2-\frac{\lambda D}{\varepsilon K}\|\nabla u\|_0^2\leq0,
\end{equation}
where $\|\cdot\|_0$ denotes the usual $L^2$ norm.
From \eqref{1.5}, we see that the free energy $E$ decays in time.
The goal of the next section is to construct highly efficient schemes for \eqref{2.1}, which satisfy
a discrete counterpart of the energy dissipation law \eqref{1.5}.
The design of the schemes follows an auxiliary variable approach, and make use of linear stabilizers
to balance the explicit treatment of the anisotropic coefficient and the nonlinear terms.
The first step is to introduce a modified energy as follows
\begin{equation}\label{reforenergy}
E_1(\phi, u)=\int_{\Omega}\left(\frac{1}{2}\big(m^2(\nabla \phi)-S_1\big) |\nabla \phi|^2+\frac{\lambda}{2 \varepsilon K} u^2+ F(\phi)-\frac{S_2}{2\varepsilon^2}\phi^2 \right) d \mathbf{x},
\end{equation}
where $S_1, S_2$ are two positive constants. The energy is then split into two parts:
\bex
E(\phi, u)=E_1(\phi, u)+E_2(\phi, u),
\eex
where
\bex
E_2(\phi, u)=\int_{\Omega}\left(\frac{S_1}{2} |\nabla \phi|^2+ \frac{S_2}{2\varepsilon^2}\phi^2\right) d \mathbf{x}+1.
\eex
This energy splitting leads to the reformulation of the phase field equation \eqref{2.1a}:
\bex
\tau\phi_t+\frac{\delta E_2}{\delta\phi}=-\frac{\delta E_1}{\delta\phi}-4\lambda\varepsilon F(\phi)u,
\eex
which
\bex
&&\frac{\delta E_1}{\delta\phi}=-\nabla \cdot\big(m^2(\nabla \phi) \nabla \phi+|\nabla \phi|^2 m(\nabla \phi) \boldsymbol{H}(\phi)\big)+ S_1\Delta\phi+f(\phi)-\frac{S_2}{\varepsilon^2}\phi,\\
&&\frac{\delta E_2}{\delta\phi}= - S_1\Delta\phi+\frac{S_2}{\varepsilon^2}\phi.
\eex
The benefit of this reformulation is the presence of the linear terms $S_1\Delta\phi$ and $\frac{S_2}{\varepsilon^2}\phi$ in the left side,
which can be treated implicitly in the scheme construction.

Then we introduce the auxiliary variable
\be\label{av}
q(t)=E(\phi, u).
\ee
Differentiating $q(t)$ gives
\be\label{qeq}
\frac{d q}{d t}=-\frac{q}{E(\phi, u)} \mathcal{H}(\phi, u),
\ee
where, according to \eqref{1.5} and \eqref{2.1a},
\be\label{H}
\mathcal{H}(\phi, u)
= - \frac{dE(\phi,u)}{dt}
= \frac{1}{\tau}\Big\| \frac{\delta E}{\delta \phi}+4\lambda\varepsilon F(\phi)u\Big\|_0^2+\frac{\lambda D}{\varepsilon K}\|\nabla u\|_0^2.
\ee
Notice that the quantity $\frac{q}{E(\phi, u)}$, which is equal to 1, is technically added in the front
of $\mathcal{H}(\phi, u)$. This is to allow more flexibility in designing the schemes.
To summarize, we arrive at the following reformulation of the phase field dendrite crystal growth model, which is
strictly equivalent to \eqref{2.1}:
\begin{subequations}\label{reform}
\begin{align}
\tau\phi_t- S_1\Delta\phi+\frac{S_2}{\varepsilon^2}\phi
&=-\frac{\delta E_1}{\delta\phi}-4\lambda\varepsilon F(\phi)u, \label{reform_1}\\
u_t
&=D\Delta u+4\varepsilon^2KF(\phi)\phi_t, \label{reform_2}\\
\frac{d q}{d t}
&= -\frac{q}{E(\phi, u)} \mathcal{H}(\phi, u). \label{reform_3}
\end{align}
\end{subequations}

\section{Construction of the schemes}

We are now in a position to construct our schemes based on the reformulation
\eqref{reform}. We start with the time stepping scheme.

\subsection{The time stepping scheme}

Let $M$ be a positive integer and $0=t^0<t^1<\cdots<t^M=T$ be a uniform partition of $[0,T]$, where $t^n=n\delta t, n=0,1,\cdots, M$, and $\delta t=T/M$ is the time step size.
Let $(\cdot)^{n+1}$ denotes the numerical approximation to $(\cdot)$ at $t^{n+1}$.
Essentially, the proposed schemes is a kind of BDF-$k$ approximation
to the time derivative terms in the reformulation. The basic idea is to treat the
linear terms implicitly, nonlinear terms explicitly, and interaction terms explicitly or semi-implicitly.
Precisely, we propose the following scheme: Given $\phi^{n}, u^{n}$, and ${q}^{n}$.

{\bf Step 1}\ Compute the intermediate solution $\bar\phi^{n+1}, \bar u^{n+1}, \bar{q}^{n+1}$
by:
\begin{subequations}\label{2.4}
\begin{align}
& \tau\frac{\alpha_k \bar{\phi}^{n+1}-A_k(\phi^n)}{\delta t}- S_1\Delta\bar{\phi}^{n+1}+\frac{S_2}{\varepsilon^2}\bar{\phi}^{n+1}
=-\frac{\delta E_1}{\delta\phi}(B_k(\phi^n))-4\lambda\varepsilon F(B_k(\phi^n))B_k(u^n), \label{2.4a}\\
& \frac{\alpha_k \bar{u}^{n+1}-A_k(u^n)}{\delta t}-D\Delta\bar{u}^{n+1}=4\varepsilon^2KF(B_k(\phi^n))\frac{\alpha_k \bar{\phi}^{n+1}-A_k(\phi^n)}{\delta t}, \label{2.4b}\\
& \frac{1}{\delta t}(\bar{q}^{n+1}-q^n)=-\frac{\bar{q}^{n+1}}{E(\bar{\phi}^{n+1},\bar{u}^{n+1})} \mathcal{H}(\bar{\phi}^{n+1},\bar{u}^{n+1}), \label{2.4c}
\end{align}
\end{subequations}
where the coefficients $\alpha_k$ and the operators $A_k(\cdot)$, $B_k(\cdot)$ are defined
according to the BDF-$k$ approximation. Specifically, we have for $k=1,2,3$

- BDF-1:
$
\alpha_1=1, \quad A_1(\varphi^n)=\varphi^n, \quad B_1(\varphi^n)=\varphi^n ;
$

- BDF-2:
$
\alpha_2=\frac{3}{2}, \quad A_2(\varphi^n)=2 \varphi^n-\frac{1}{2} \varphi^{n-1}, \quad B_2(\varphi^n)=2 \varphi^n-\varphi^{n-1};
$

- BDF-3:
$
\alpha_3=\frac{11}{6}, \quad A_3(\varphi^n)=3 \varphi^n-\frac{3}{2} \varphi^{n-1}+\frac{1}{3} \varphi^{n-2}, \quad B_3(\varphi^n)=3 \varphi^n-3 \varphi^{n-1}+\varphi^{n-2}.
$

{\bf Step 2}\  Compute first the coefficients
\be\label{2.4d}
& \xi_{k}^{n+1}=\frac{\bar{q}^{n+1}}{E(\bar{\phi}^{n+1},\bar{u}^{n+1})}, \quad  \eta_k^{n+1}=1-(1-\xi_{k}^{n+1})^{k+1},
\ee
then the solution at the step $n+1$, i.e., $\phi^{n+1}, u^{n+1}, q^{n+1}$, by
\be\label{2.4e}
& \phi^{n+1}=\eta_k^{n+1} \bar{\phi}^{n+1}, \quad  u^{n+1}=\eta_k^{n+1} \bar{u}^{n+1},
\ee
and
\be\label{2.5}
q^{n+1}=\zeta^{n+1} \bar{q}^{n+1}+(1-\zeta^{n+1}) E(\phi^{n+1},u^{n+1}),
\ee
where $\zeta^{n+1}\in [0,1]$ is chosen such that
\be\label{2.6}
\frac{q^{n+1}-\bar{q}^{n+1}}{\delta t}\leq
\frac{\bar{q}^{n+1}}{E(\bar{\phi}^{n+1},\bar{u}^{n+1})} \mathcal{H}(\bar{\phi}^{n+1},\bar{u}^{n+1}).
\ee

Before proving the existence of $\zeta^{n+1}$, there are several
key points in the scheme worthy of explanation.

\begin{remark}\label{rm1}
$\ $

1) First formally, it is readily seen from \eqref{2.4a} and \eqref{2.4b} that $\bar{\phi}^{n+1}$ and
$\bar{u}^{n+1}$ are respectively $k$-order approximation to ${\phi}(t^{n+1})$ and $u(t^{n+1})$.
Therefore, directly setting ${\phi}^{n+1}=\bar{\phi}^{n+1}$ and $u^{n+1}=\bar{u}^{n+1}$ would result
in a $k$-order scheme. However this scheme would be unstable.

2) $\bar{q}^{n+1}$ computed by \eqref{2.4c} is a first order approximation to
${q}(t^{n+1})$, which, by \eqref{av}, is the original energy $E$ at $t^{n+1}$.
Thus $\xi_{k}^{n+1}$ is a first order approximation to $1$, i.e., $\xi_{k}^{n+1}=1+O({\delta t})$,
and $\eta_k^{n+1}$ is a $k+1$-order approximation to 1. Therefore, updating
$\phi^{n+1}$ and $u^{n+1}$ by \eqref{2.4e} remains $k$-order accurate.

3) It follows from \eqref{2.4c} that
\begin{equation}\label{2.12}
\bar{q}^{n+1}=\frac{q^n}{1+\delta t \frac{\mathcal{H}(\bar{\phi}^{n+1},\bar{u}^{n+1})}{E(\bar{\phi}^{n+1},\bar{u}^{n+1})}}.
\end{equation}
According to \eqref{energy} and \eqref{H}, both $\mathcal{H}$ and $E$ are non negative.
Therefore, $\bar{q}^{n+1}$ remains positive if ${q^n}$ is positive.

4) Updating $\phi^{n+1}$ and $u^{n+1}$ by \eqref{2.4e} follows the popular idea of
the SAV approach, which is a key step toward unconditionally stable schemes.

5) The updating step of ${q}^{n+1}$, i.e., \eqref{2.5}, is usually called relaxation step.
This relaxation process was initially introduced by Jiang et al. in \cite{Jiang2022improving}
to improve the accuracy of the auxiliary variable method, then extended by Zhang et al.
\cite{zhang2022generalized} for the generalized auxiliary variable method.
The aim of this step is to adjust the auxiliary variable so that it better approximates
the original energy $E$, therefore makes the computed result more reliable.
The constraint \eqref{2.6} imposed by the updating step \eqref{2.5} has the purpose
to keep the scheme stable, as shown in Theorem \ref{thm1}.

\end{remark}

Now we turn to proving the existence of $\zeta^{n+1}\in [0,1]$ in \eqref{2.6}.
It is readily seen that this is equivalent to the problem of finding $\zeta^{n+1}\in [0,1]$ such that
\begin{equation}\label{opt3}
\begin{aligned}
\left(\bar{q}^{n+1}-E(\phi^{n+1},u^{n+1})\right) \zeta^{n+1}\leq
&\bar{q}^{n+1}-E(\phi^{n+1},u^{n+1})
+\delta t \frac{\bar{q}^{n+1}}{E(\bar{\phi}^{n+1},\bar{u}^{n+1})} \mathcal{H}(\bar{\phi}^{n+1},\bar{u}^{n+1}).
\end{aligned}
\end{equation}
In fact, $\zeta^{n+1}$ is not unique.
We will propose an algorithm
to choose $\zeta^{n+1}\in [0,1]$ as close as possible to 0 (so, according to \eqref{2.5},
$q^{n+1}$ as close as possible to $E(\phi^{n+1},u^{n+1})$).

\begin{algorithm}\label{alg1}
(Determination of the parameter $\zeta^{n+1}\in [0,1]$)
\begin{enumerate}


\item If $\bar{q}^{n+1}\geq E(\phi^{n+1},u^{n+1})$, we set $\zeta^{n+1}=0$.

\item If $\bar{q}^{n+1}<E(\phi^{n+1},u^{n+1})$ and $\bar{q}^{n+1}-E(\phi^{n+1},u^{n+1})+\delta t \frac{\bar{q}^{n+1}}{E(\bar{\phi}^{n+1},\bar{u}^{n+1})} \mathcal{H}(\bar{\phi}^{n+1},\bar{u}^{n+1}) \geq 0$, we set $\zeta^{n+1}=0$.
\item If $\bar{q}^{n+1}<E(\phi^{n+1},u^{n+1})$ and $\bar{q}^{n+1}-E(\phi^{n+1},u^{n+1})+\delta t \frac{\bar{q}^{n+1}}{E(\bar{\phi}^{n+1},\bar{u}^{n+1})} \mathcal{H}(\bar{\phi}^{n+1},\bar{u}^{n+1})<0$, we set
\bex
\zeta^{n+1}=1-\frac{\delta t \bar{q}^{n+1} \mathcal{H}(\bar{\phi}^{n+1},\bar{u}^{n+1})}
       {E(\bar{\phi}^{n+1},\bar{u}^{n+1})(E(\phi^{n+1},u^{n+1})-\bar{q}^{n+1})}.
\eex
Then obviously $\zeta^{n+1}\in [0,1]$ (remember $\bar{q}^{n+1}\ge 0$ according to Remark \ref{rm1} 3).
 \end{enumerate}
\end{algorithm}

\begin{theorem}\label{thm1}
The scheme \eqref{2.4}-\eqref{2.5} satisfies the energy dissipation law as follows:
\begin{equation}\label{2.9}
q^{n+1}\leq q^n, \ \ \forall n\ge 0.
\end{equation}
In addition, the quantity
$\|\nabla\phi^{n}\|_0$ and $\|u^n\|_0$ remain bounded for all $n$.
\end{theorem}

\begin{proof}
Summing \eqref{2.4c} and \eqref{2.6} yields:
\bex
q^{n+1}-q^n\leq 0.
\eex
This gives \eqref{2.9}.\\
Now we prove the boundedness of $\displaystyle \|\nabla\phi^{n}\|_0$ and $\|u^n\|_0$.
First, it follows from \eqref{2.12} and \eqref{2.9} that for all $n\ge 0$,
\bex
\bar{q}^{n+1} \leq q^n \leq q^0:=E(\phi^{0}, u^{0}).
\eex
Furthermore, according to \eqref{2.4d}, we have
\begin{equation}\label{2.13}
|\xi_k^{n+1}|=\frac{\bar{q}^{n+1}}{E(\bar{\phi}^{n+1},\bar{u}^{n+1})}
\leq \frac{q^0}{E(\bar{\phi}^{n+1},\bar{u}^{n+1})}.
\end{equation}
According to \eqref{2.4d} again, we obtain
\be\label{etak}
\eta_k^{n+1}=1-(1-\xi_k^{n+1})^{k+1}=\xi_k^{n+1} P_k(\xi_k^{n+1}),
\ee
where $P_k$ is a polynomial of degree $k$.
Noticing $E(\bar{\phi}^{n+1},\bar{u}^{n+1})\ge 1$ (see Remark \ref{rm0}), we deduce from \eqref{2.13} that $|\xi_k^{n+1}|$ is bounded by $q^0$. Therefore, $P_k(\xi_k^{n+1})$ is bounded by a constant, which depends only on $q^0$ and $k$.
It then follows from \eqref{etak}, \eqref{2.13}, and the definition of $E$:
\begin{equation}\label{2.14}
|\eta_k^{n+1}|
\leq
\frac{c}{\int_{\Omega}\left(\frac{1}{2}|m(\nabla \bar{\phi}^{n+1}) \nabla \bar{\phi}^{n+1}|^2+\frac{\lambda}{2 \varepsilon K} (\bar{u}^{n+1})^2\right)d \mathbf{x}+1}.
\end{equation}
Finally, we deduce from the update step \eqref{2.4e}, i.e. $\phi^{n+1}=\eta_k^{n+1} \bar{\phi}^{n+1}$, $u^{n+1}=\eta_k^{n+1} \bar{u}^{n+1}$, and the fact that
$m^2(\nabla\bar{\phi}^{n+1})\geq(1-\sigma)^2$:
\bex
&&\int_{\Omega} \left(\frac{1}{2}(1-\sigma)^2|\nabla\phi^{n+1}|^2+\frac{\lambda}{2 \varepsilon K} (u^{n+1})^2\right)d \mathbf{x}\\
&&\leq (\eta_k^{n+1})^2\int_{\Omega} \left(\frac{1}{2}m^2(\nabla\bar{\phi}^{n+1})|\nabla\bar{\phi}^{n+1}|^2+\frac{\lambda}{2 \varepsilon K} (\bar{u}^{n+1})^2 \right)d \mathbf{x} \\
&&  \leq
\frac{\displaystyle c\int_{\Omega}\left(\frac{1}{2}|m(\nabla \bar{\phi}^{n+1}) \nabla \bar{\phi}^{n+1}|^2+\frac{\lambda}{2 \varepsilon K} (\bar{u}^{n+1})^2\right)d \mathbf{x}}
{\displaystyle\left(\int_{\Omega}\left(\frac{1}{2}|m(\nabla \bar{\phi}^{n+1}) \nabla \bar{\phi}^{n+1}|^2+\frac{\lambda}{2 \varepsilon K} (\bar{u}^{n+1})^2\right)d \mathbf{x}+1\right)^2}
\leq c.
\eex
This completes the proof.
\end{proof}

\begin{remark}
$\ $

1) It is observed from the proof of Theorem \ref{thm1} that the auxiliary variable
plays a key role in establishing the boundedness of
$\|\nabla\phi^{n+1}\|_0$ and $\|u^{n+1}\|_0$ (therefore the stability of the scheme).
As mentioned in Remark \ref{rm1} 1),
setting ${\phi}^{n+1}=\bar{\phi}^{n+1}$ and $u^{n+1}=\bar{u}^{n+1}$ without
the auxiliary variable $\xi^{n+1}_k$ and $\eta^{n+1}_k$ is not stable.
Contrarily, adjusting dynamically
$\bar{\phi}^{n+1}$ and $\bar{u}^{n+1}$ by multiplying the
factor $\eta^{n+1}_k$ (which is close to 1) makes
$\eta^{n+1}_k\bar{\phi}^{n+1}$ and $\eta^{n+1}_k\bar{u}^{n+1}$ bounded,
thus makes the scheme stable.

2) It is desirable that the numerical scheme maintains the energy dissipation, i.e.,
$E(\phi^{n+1},u^{n+1}) \leq E(\phi^{n},u^{n})$. However this can not be proved
theoretically, although our numerical experiment implies this is true.
In Theorem \ref{thm1}, it is proved $q^{n+1}\le q^n$.
However, in general we don't have $q^{n+1}=E(\phi^{n+1},u^{n+1})$.
As explained in Remark \ref{rm1} 5), the updating step \eqref{2.5}
can guarantee that $q^{n+1}$ is as close as possible to $E(\phi^{n+1},u^{n+1})$.
Furthermore, by virtue of the parameter selection algorithm, i.e., Algorithm \ref{alg1},
we have $\zeta^{n+1}=0$ in the first two cases. Therefore we have
$q^{n+1}=E(\phi^{n+1},u^{n+1})$ in these two cases.
The only exception is the third case, in which $\zeta^{n+1}>0$.
However (and surprisingly), this case has never happened in our numerical
tests for unknown reason; see Section \ref{sect-num}.

3) It is very interesting to note that the stability of the scheme is
independent on how $\bar\phi^{n+1}$ and $\bar u^{n+1}$ are computed.
In our schemes, $\bar\phi^{n+1}$ and $\bar u^{n+1}$ are calculated in \eqref{2.4a}
and \eqref{2.4b} following a BDF-$k$ approach in which the linear terms are treated
implicitly while nonlinear terms are treated explicitly. In particular,
two stabilization parameters $S_1$ and $S_2$ are added in \eqref{2.4a} in
the calculation of the phase function $\bar\phi^{n+1}$. Although the presence
of $S_1$ and $S_2$ in \eqref{2.4a} plays no role in the stability proof in Theorem
\ref{thm1}, our numerical examples show that they are helpful in capturing
meaningful crystal growth phenomena
when relatively larger time step sizes are used in the simulation.
\end{remark}

\subsection{The spatial discretization}
In this section, we describe spatial discretization of the semi-discrete problems to be solved
at each time step, and give some implementation details.
Since the Fourier spectral method is particularly well-suited for handling periodic problems, it will be utilized for the spatial discretization.
For simplicity, we consider a two-dimensional square domain $\Omega$.
It is observed in the scheme \eqref{2.4}-\eqref{2.5} that only the equations \eqref{2.4a} and \eqref{2.4b}
need to be discretized in space.
The goal of Fourier method is to seek the approximate solutions
$\bar\phi_N^{n+1}$ and $\bar u_N^{n+1}$
in the form of truncated Fourier expansions:
\be\label{Fexp}
\bar\phi_N^{n+1}(\mathbf{x})=\sum_{k_1, k_2=-N}^N \hat{\phi}_{\mathbf{k}}^{n+1} \exp (-\mathrm{i} {\mathbf{k} \mathbf{x}}),
\ \ \
\bar u_N^{n+1}(\mathbf{x})=\sum_{k_1, k_2=-N}^N \hat{u}_{\mathbf{k}}^{n+1} \exp (-\mathrm{i} {\mathbf{k} \mathbf{x}}), \ \ \mathbf{x}\in \Omega,
\ee
where $\mathrm{i}=\sqrt{-1}$, $\mathbf{k} = (k_1, k_2)$ and $N$ is a positive integer.
The spectral coefficients $\hat{\phi}_{\mathbf{k}}^{n+1}$ and $\hat{u}_{\mathbf{k}}^{n+1}$ are obtained by
solving respectively the equation
\be\label{space_1}
(\tau\frac{\alpha_k}{\delta t}+\frac{S_2}{\varepsilon^2}+ S_1|\mathbf{k}|^2)\hat{\phi}_{\mathbf{k}}^{n+1}
=\Big\{\frac{\tau}{\delta t}A_k(\phi_N^n)-\frac{\delta E_1}{\delta\phi}(B_k(\phi_N^n))-4\lambda\varepsilon F(B_k(\phi_N^n))B_k(u_N^n)\Big\}_{\mathbf{k}},
\ee
and
\be\label{space_2}
(\frac{\alpha_k}{\delta t}+D|\mathbf{k}|^2)\hat{u}_{\mathbf{k}}^{n+1}
=\Big\{\frac{1}{\delta t}A_k(u_N^n)+4\varepsilon^2KF(B_k(\phi_N^n))\frac{\alpha_k \bar{\phi}_N^{n+1}-A_k(\phi_N^n)}{\delta t}\Big\}_{\mathbf{k}},
\ee
which are derived by
applying the Fourier transform to \eqref{2.4a} and \eqref{2.4b}.
In the right hand sides of \eqref{space_1} and \eqref{space_2},
the $\{\cdot\}_{\mathbf{k}}$ has been used to represent the $\mathbf{k}$-th Fourier mode of the source terms.
We see that since the equations \eqref{2.4a} and \eqref{2.4b} are linear on
$\bar\phi_N^{n+1}$ and $\bar u_N^{n+1}$,
the spectral coefficients $\hat{\phi}_{\mathbf{k}}^{n+1}$ and $\hat{u}_{\mathbf{k}}^{n+1}$ can be computed one-by-one.
Once $\hat{\phi}_{\mathbf{k}}^{n+1}$ and $\hat{u}_{\mathbf{k}}^{n+1}$ are obtained from \eqref{space_1} and \eqref{space_2},
we use \eqref{Fexp} to get the approximate solutions $\bar\phi_N^{n+1}$ and $\bar u_N^{n+1}$.

The Fourier approximate solution $\bar{q}_N^{n+1}$ to $\bar{q}^{n+1}$ is explicit, given by
\be
\frac{1}{\delta t}(\bar{q}_N^{n+1}-q_N^n)=-\frac{\bar{q}_N^{n+1}}{E(\bar{\phi}_N^{n+1},\bar{u}_N^{n+1})} \mathcal{H}(\bar{\phi}_N^{n+1},\bar{u}_N^{n+1}).
\label{2.4c-F}
\ee

The full discrete version of the {\bf Step 2} is direct:  compute first the parameters
\be\label{2.4d-F}
& \xi_{k}^{n+1}=\frac{\bar{q}_N^{n+1}}{E(\bar{\phi}_N^{n+1},\bar{u}_N^{n+1})}, \quad
\eta_k^{n+1}=1-(1-\xi_{k}^{n+1})^{k+1},
\ee
then update $\phi_N^{n+1}, u_N^{n+1}$, and $q_N^{n+1}$ by
\be\label{2.4e-F}
\phi_N^{n+1}=\eta_k^{n+1} \bar{\phi}_N^{n+1}, \quad  u_N^{n+1}=\eta_k^{n+1} \bar{u}_N^{n+1},
\ee
and
\be\label{2.5-F}
q_N^{n+1}=\zeta^{n+1} \bar{q}_N^{n+1}+(1-\zeta^{n+1}) E(\phi_N^{n+1},u_N^{n+1}),
\ee
where $\zeta^{n+1}\in [0,1]$ is chosen such that
\bex
\frac{q_N^{n+1}-\bar{q}_N^{n+1}}{\delta t}\leq
\frac{\bar{q}_N^{n+1}}{E(\bar{\phi}_N^{n+1},\bar{u}_N^{n+1})} \mathcal{H}(\bar{\phi}_N^{n+1},\bar{u}_N^{n+1}).
\eex

To summarize, the full discrete problem can be implemented as follows.
 \begin{enumerate}

\item Compute the spectral coefficients of $\bar{\phi}_{N}^{n+1}$ from
\eqref{space_1}, then $\bar{\phi}_{N}^{n+1}$ itself from \eqref{Fexp}.
\item Compute the spectral coefficients of $\bar{u}_{N}^{n+1}$ from \eqref{space_2},
then $\bar{u}_{N}^{n+1}$ itself from \eqref{Fexp}.
\item Compute $\bar{q}_N^{n+1}$ by \eqref{2.4c-F}.
\item Compute $\xi_{k}^{n+1}$ and $\eta_{k}^{n+1}$ by \eqref{2.4d-F}.
\item Update $\phi_{N}^{n+1}$, $u_{N}^{n+1}$ by \eqref{2.4e-F}.
\item Update $q_N^{n+1}$ by \eqref{2.5-F}.

 \end{enumerate}

We see from the above algorithm that our proposed method is extremely efficient.
Since the time stepping scheme only requires to solve two linear elliptic equations
at each time step. The use of the Fourier spectral method for their spatial discretization
makes the computation of the solution (in the spectral space) in a totally separate way.
Note that the scheme proposed in \cite{li2022new} requires to solve four linear elliptic equations at each time step.

By following exactly the same lines as Theorem \ref{thm1}, we have the stability of the full discrete problem as follows.

\begin{theorem}\label{thm2}
 The fully discrete scheme \eqref{space_1}-\eqref{2.5-F} satisfies the energy dissipation law:
\bex
q_N^{n+1}\leq q_N^n, \ \ \forall n\ge 0.
\eex
In addition, the quantity
$\|\nabla\phi_{N}^{n}\|_0$ and $\|u_{N}^n\|_0$ remain bounded for all $n$.
\end{theorem}

\section{Numerical experiments}\label{sect-num}

In this section, a series of numerical examples are provided to verify the efficiency of the proposed schemes.
We start with verifying the convergence of the schemes in the case of isotropy and anisotropy by constructing appropriate exact solutions, then testing the stability of the  schemes. Finally we will demonstrate the reliability of the new method by simulating the dendrite growth with fourfold anisotropy and sixfold anisotropy under different values of the latent heat parameter $K$.
In all examples, the domain is chosen as $(0,2\pi)^d$, where $d=2,3$.

\subsection {Convergence order and stability}

We focus on the convergence order of the time stepping schemes.
To this end we use $128$ Fourier modes in the spatial discretization,
which is large enough so that the spacial discretization error is negligible compared to the temporal discretization. First of all, we consider the isotropic case.

\begin{example}\label{exp.1}
\rm{(Accuracy test in isotropic case) Set $\sigma=0$. Consider the equation}
\begin{equation*}
\left\{
             \begin{array}{ll}
\tau\phi_t=\Delta\phi-f(\phi)-4\lambda\varepsilon F(\phi)u+f_1,\\
u_t=D\Delta u+4\varepsilon^2KF(\phi)\phi_t+f_2,
             \end{array}
           \right.
\end{equation*}
with the fabricated exact solution:
\begin{equation*}
\left\{
             \begin{array}{ll}
\phi(x,y,t)=\sin(x)\cos(y)\cos(t),\\
u(x,y,t)=\cos(x)\sin(y)\cos(t),
             \end{array}
           \right.
\end{equation*}
where $f_1$ and $f_2$ represent the corresponding external force terms. Set the parameters as follows
$$
\begin{aligned}
\tau=10, \varepsilon=\lambda=D=K=1, S_1=4, S_2=4.
 \end{aligned}
$$
\end{example}

In Figure \ref{time_order_iso}, we present the $L^2$ errors in log-log scales of phase field $\phi$ and the temperature field $u$ for three values of the order parameter $k$ in BDF$k$. It is observed that the scheme using BDF$k$ with $k$=1,2,3 result in the convergence order $1$, $2$, and $3$ respectively, which is in a good agreement with
the expected convergence rates of the schemes.

\begin{example}\label{exp.2}
\rm{(Accuracy test in anisotropic case)
Consider the crystal growth model with fourfold anisotropy in the domain $(0,2\pi)^2$ as follows:}
\begin{equation}\label{e2}
\left\{
             \begin{array}{ll}
\tau\phi_t=\nabla \cdot(m^2(\nabla \phi) \nabla \phi+|\nabla \phi|^2 m(\nabla \phi) \boldsymbol{H}(\phi))-f(\phi)-4\lambda\varepsilon F(\phi)u+f_3,\\
u_t=D\Delta u+4\varepsilon^2KF(\phi)\phi_t+f_4,
             \end{array}
           \right.
\end{equation}
where $f_3$ and $f_4$ are force terms fabricated such that
the problem admits the same exact solution as Example \ref{exp.1}.
The parameters are set as follows:
$$
\begin{aligned}
\tau=4e3, \varepsilon=1, \lambda=1, D=1, K=0.01, \sigma=0.05,\beta=4, S_1=4, S_2=4.
 \end{aligned}
$$
\end{example}
The $L^2$-errors in log-log scales of phase field $\phi$ and the temperature field $u$ for $k=1,2,3$ are shown in Figure \ref{time_order_aniso}, from which we observe
the expected convergence rates, although the accuracy is slightly lower than the isotropic case.
\begin{figure}[h]
\begin{minipage}{0.3 \textwidth}
\centering
\includegraphics[width=\textwidth]{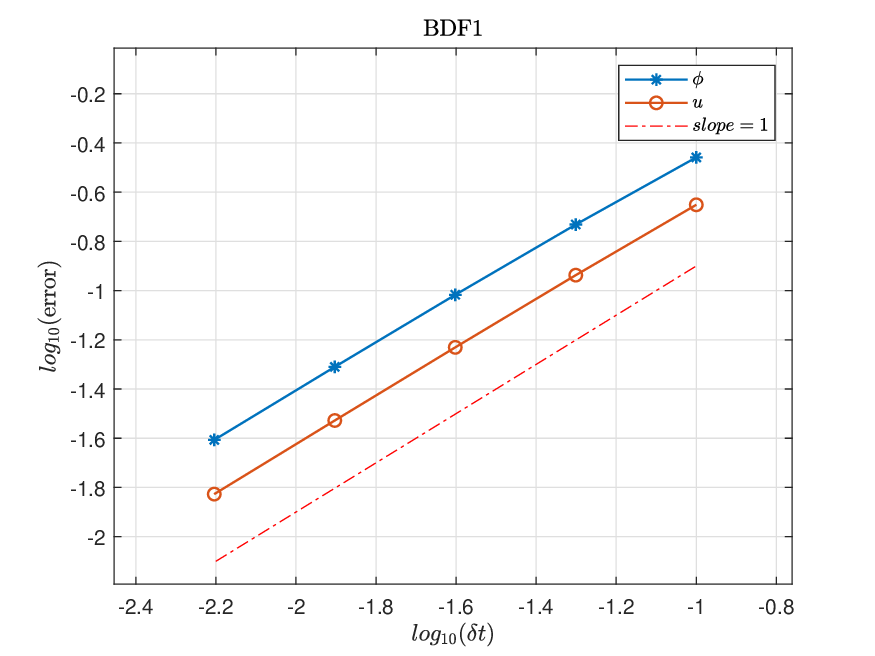}
\end{minipage}
\begin{minipage}{0.3 \textwidth}
\centering
\includegraphics[width=\textwidth]{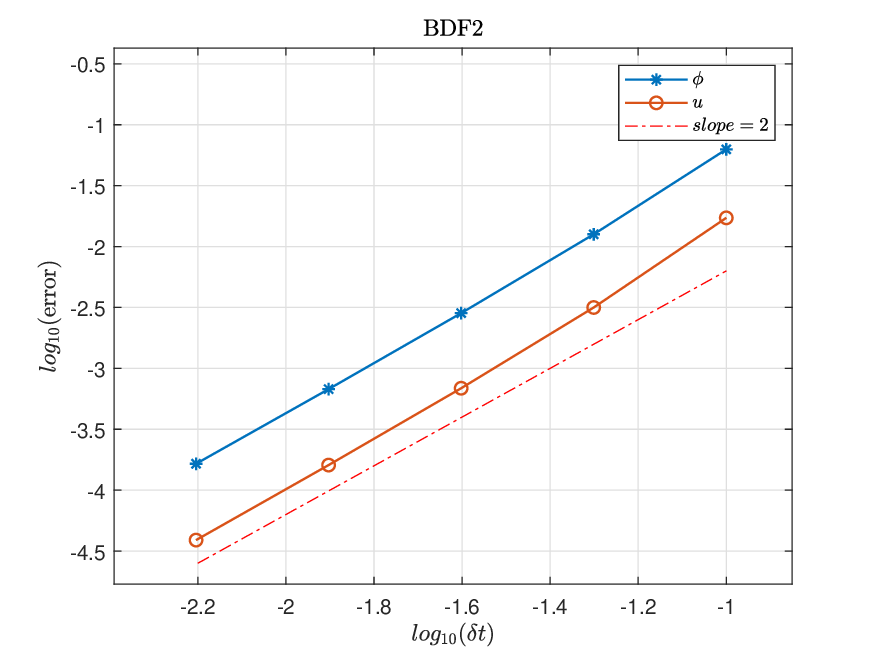}
\end{minipage}
\begin{minipage}{0.3 \textwidth}
\centering
\includegraphics[width=\textwidth]{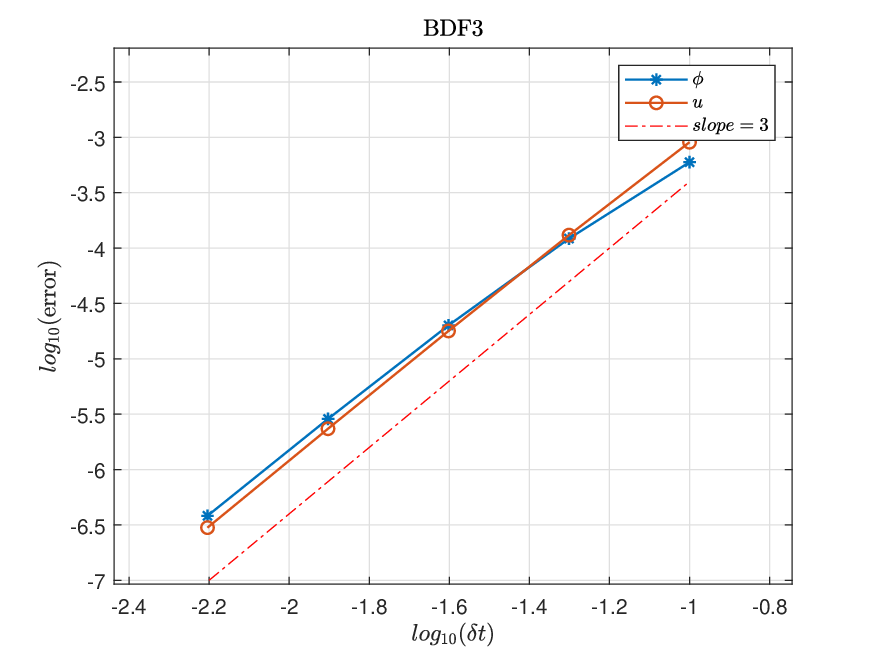}
\end{minipage}
\caption{(Example \ref{exp.1}) Accuracy in time of $\phi$ and $u$ for three order parameters $k=1,2,3$ in BDF$k$ in the isotropic case.}
\label{time_order_iso}
\end{figure}

\begin{figure}[h]
\begin{minipage}{0.3 \textwidth}
\centering
\includegraphics[width=\textwidth]{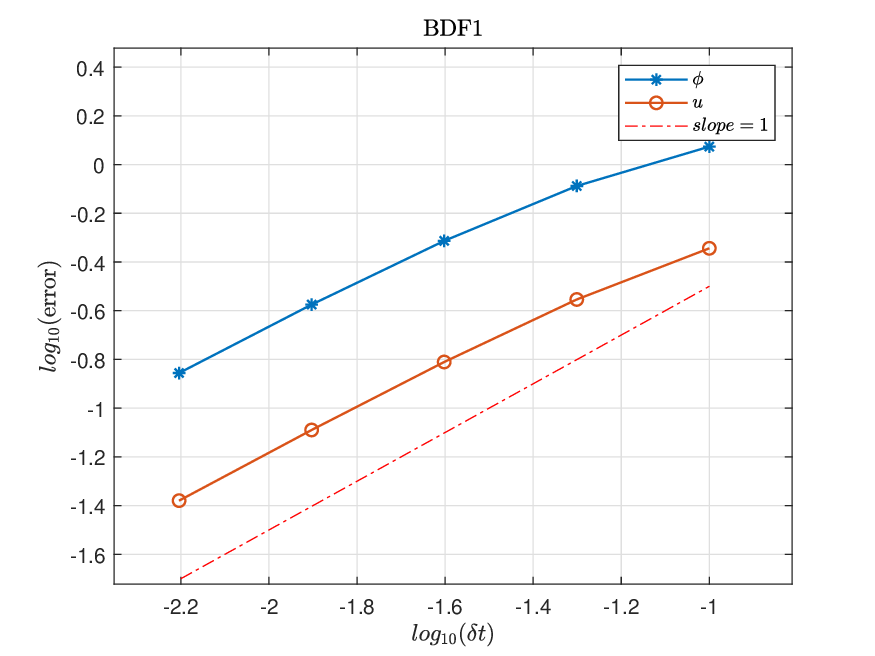}
\end{minipage}
\begin{minipage}{0.3 \textwidth}
\centering
\includegraphics[width=\textwidth]{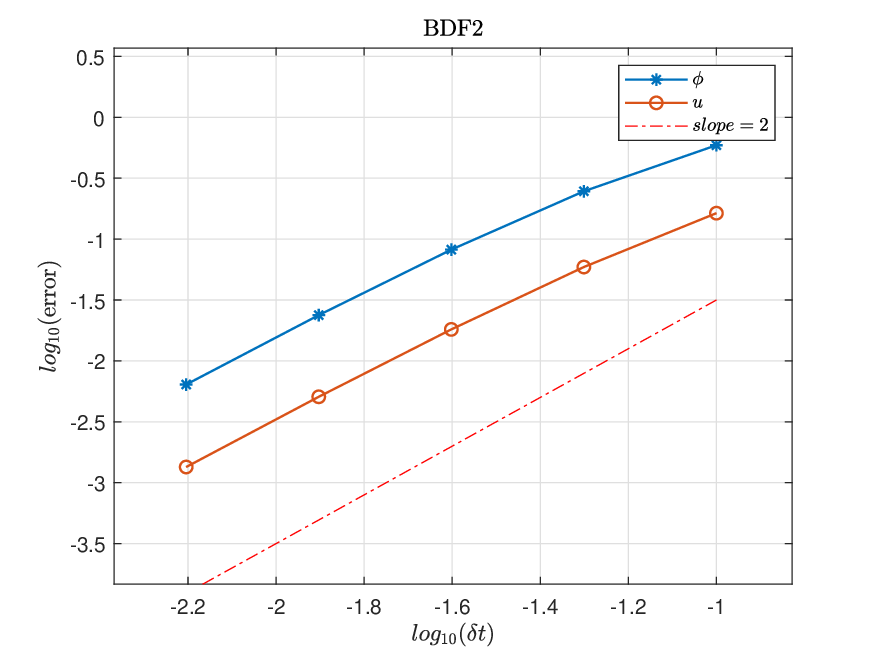}
\end{minipage}
\begin{minipage}{0.3 \textwidth}
\centering
\includegraphics[width=\textwidth]{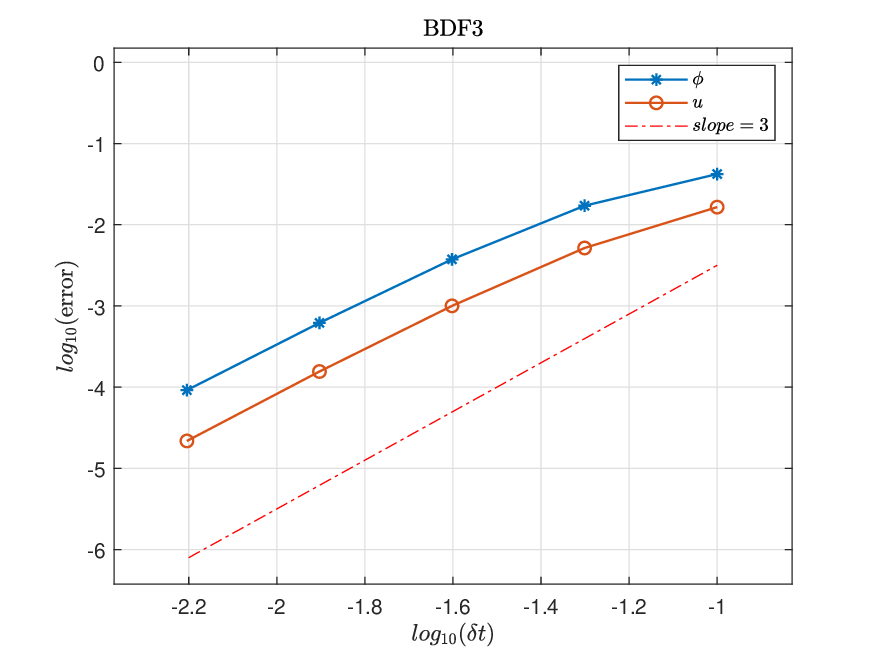}
\end{minipage}
\caption{(Example \ref{exp.2}) Accuracy in time of $\phi$ and $u$ for three order parameters $k=1,2,3$ in BDF$k$ at the anisotropic situation. }
\label{time_order_aniso}
\end{figure}

Now we turn to investigate the stability of the proposed schemes
by varying the time step size.
 \begin{example}\label{exp.3}
\rm{(Stability test) Consider the crystal growth model \eqref{2.1} with the initial condition}
\begin{equation}\label{e3}
\left\{
             \begin{array}{ll}
\phi(x,y,0)=\tanh\Big(\frac{1.5-\sqrt{(x-\pi)^2+(y-\pi)^2}}{0.1}\Big),\\
u(x,y,0)=-0.55.
             \end{array}
           \right.
\end{equation}
\end{example}
The parameters are taken as
$$
\begin{aligned}
\tau=1e2, \varepsilon=0.1, \lambda=1, D=2.25e-1,  K=1, \sigma=0.05,\beta=4.
\end{aligned}
$$
In this example, we use the scheme with $k=1$ for the time discretization, and the Fourier modes $N=128$ for the space discretization. The calculation is run up to $T=10$. In Figure \ref{energy_stable},
we plot the energy as a function of time by using
different time step sizes. A comparison is made between the stabilization
($S_1=S_2=4$) and without stabilization ($S_1=S_2=0$).
Comparing the figures (a) and (b), we see that the energy keeps decaying with the stabilization $S_1=S_2=4$, while it loses dissipation without stabilization for larger time
step sizes.
This clearly demonstrates the positive role of the stabilization terms in stabilizing the scheme.

\begin{figure}[h]
\subfigure[\scriptsize (a) $S_1=4, S_2=4$.]{
\begin{minipage}{0.45 \textwidth}
\centering
\includegraphics[width=\textwidth]{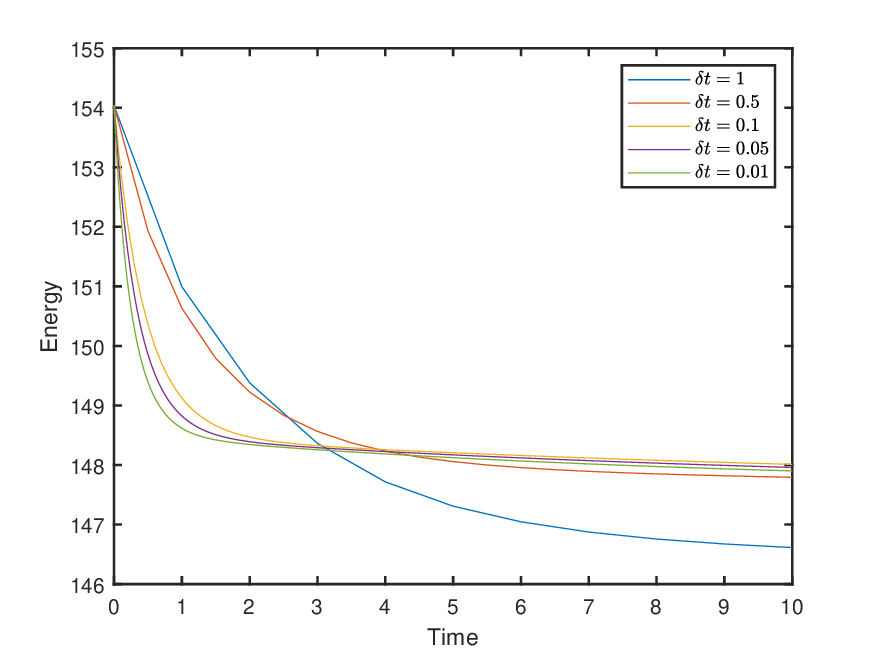}
\end{minipage}
}
\subfigure[\scriptsize (b) $S_1=0, S_2=0$.]{
\begin{minipage}{0.45 \textwidth}
\centering
\includegraphics[width=\textwidth]{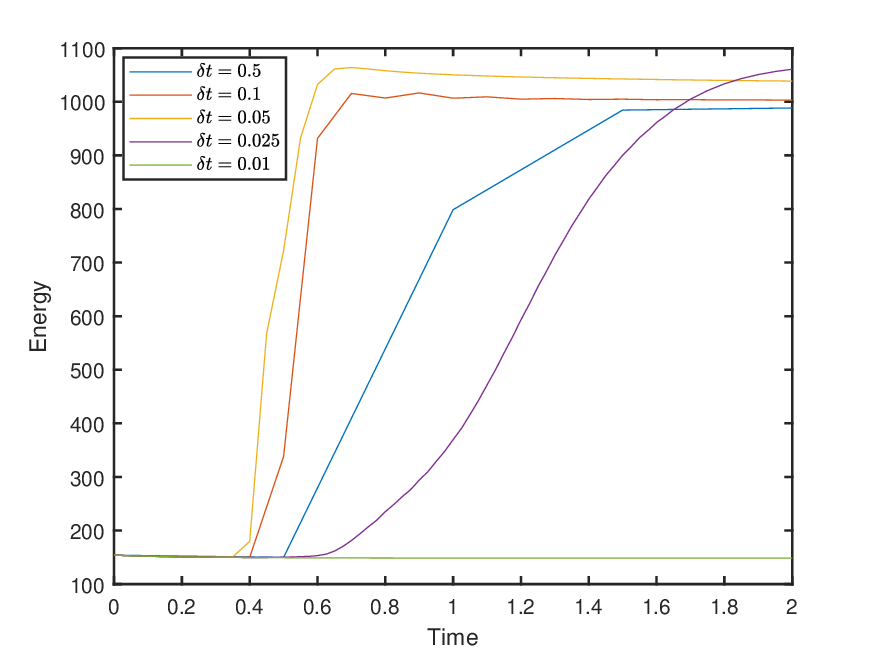}
\end{minipage}
}
\caption{(Example \ref{exp.3}) (a) Evolution in time of the energy with $S_1=S_2=4$;
(b) evolution of the energy with $S_1=S_2=0$. }
\label{energy_stable}
\end{figure}

\subsection{Fourfold dendrite crystal growth in 2D}

This subsection is devoted to simulate fourfold dendrite crystal growth and
investigate the effect of the latent heat parameter $K$ on crystal shape.
We use the proposed scheme with $k=3$ for the temporal discretization with $\delta t=0.01$ and the Fourier spectral method for the spatial discretization with $N=512$.

 \begin{example}\label{exp.4}
\rm{(Fourfold anisotropy crystal growth) Consider the anisotropic phase field model
with the following initial conditions:}
\begin{equation*}
\left\{
             \begin{array}{ll}
\phi(x,y,0)=\tanh\Big(\frac{0.02-\sqrt{(x-\pi)^2+(y-\pi)^2}}{0.072}\Big),\\
u(x,y,0)=\left\{
             \begin{array}{ll}
             0,  &\phi>0,\\
             -0.55, &otherwise.
                          \end{array}
           \right.
             \end{array}
           \right.
\end{equation*}
Set
$$
\begin{aligned}
\tau=4.4e3, \varepsilon=1.12e-2, \lambda=380, D=2.25e-4, \sigma=0.05,\beta=4.
\end{aligned}
$$
\end{example}
We plot in Figures \ref{fig:fourfoldK06}-\ref{fig:fourfoldK08} some snapshots of the phase field $\phi$
and the temperature field $u$ computed for the values 0.6, 0.7 and 0.8 of the latent heat parameter $K$.
We observe from these figures that the crystal surface, represented by the isocontours of $\{\phi = 0\}$, grows in time.  A star shape with four branches are always formed in all cases starting with a tiny nucleus.
However, the latent heat parameter has significant impact on the width of the branch of the growing crystals:
larger is $K$, thinner is the width of the branches, and sharper are the tips. This is consistent with the results reported in \cite{kobayashi1993modeling,yang2019efficient,zhang2020fully,yang2021novel,li2022new}.
The isocontours of the temperature of each simulation are also shown in Figures \ref{fig:fourfoldK06}-\ref{fig:fourfoldK08}.
We see that the contours of $u$ take same dendrite crystal shape as the phase field. The physical explanation
of this observation is that the heat is propagating only at the interface.

The dissipation behavior of the original energy in time is shown in Figure \ref{fourfold_energy}(a).
The decay feature of the original energy for all tested $K$ reflects good stability property
of the scheme employed in the calculation.
Figure \ref{fourfold_energy}(b) displays the crystal area,
defined by the quantity $\int_{\Omega}\frac{1+\phi}{2}dx$, indicating
that the area of the crystal increases in time and with decreasing $K$.
This is also in a good agreement with the existing results; see, e.g., \cite{kobayashi1993modeling,karma1998quantitative,yang2020efficient,li2022new}.

\begin{figure}[htbp]
	\vspace{-0.35cm}
	\subfigtopskip=2pt
	\subfigbottomskip=2pt
	\subfigcapskip=-5pt
	\centering	
    \subfigure[\scriptsize (a) Snapshots of the phase field $\phi$.]{
		\begin{minipage}[t]{0.185\linewidth}
			\centering
			\includegraphics[width=1\linewidth]{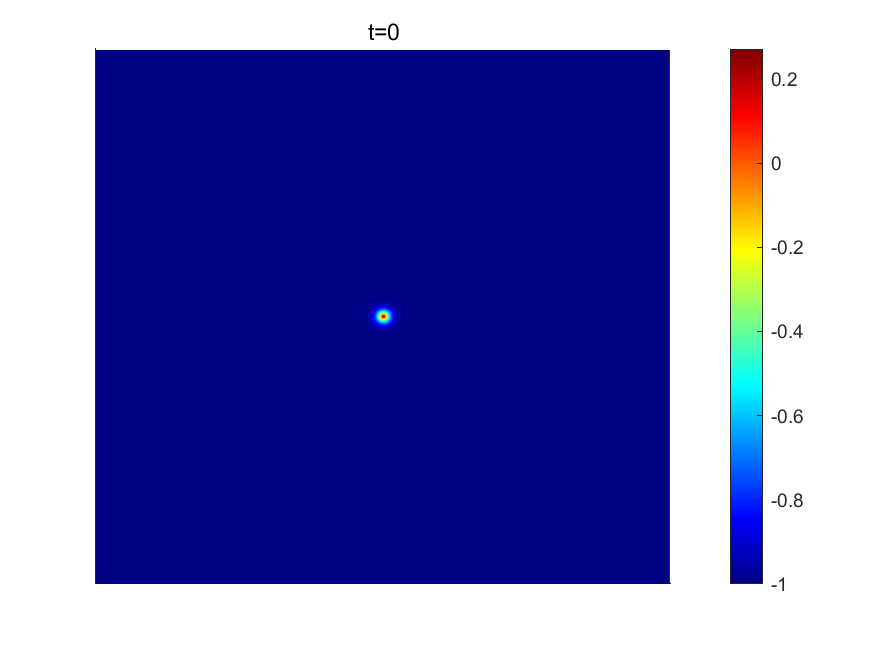}
		\end{minipage}
		\begin{minipage}[t]{0.185\linewidth}
			\centering
			\includegraphics[width=1\linewidth]{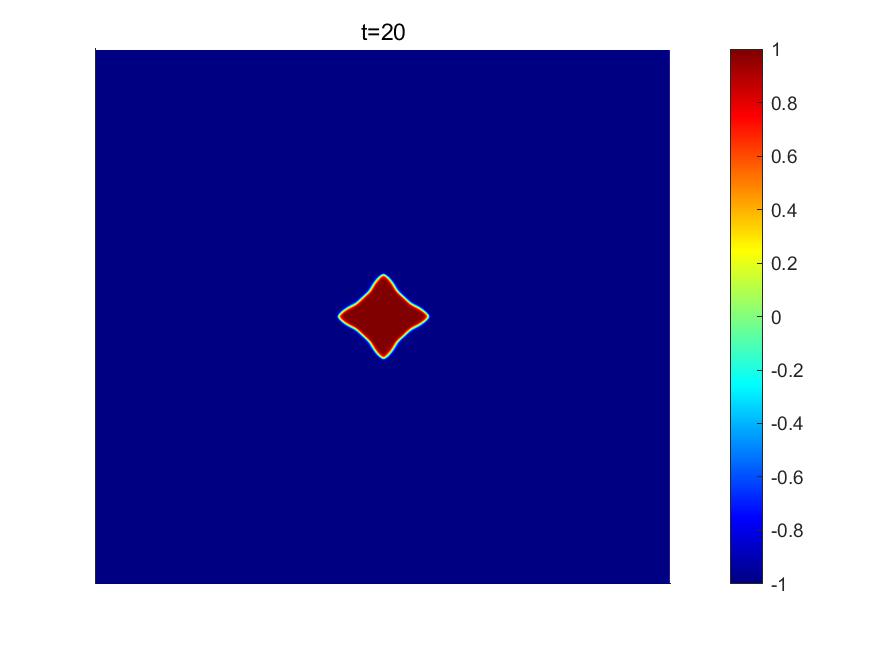}
		\end{minipage}
		\begin{minipage}[t]{0.185\linewidth}
			\centering
			\includegraphics[width=1\linewidth]{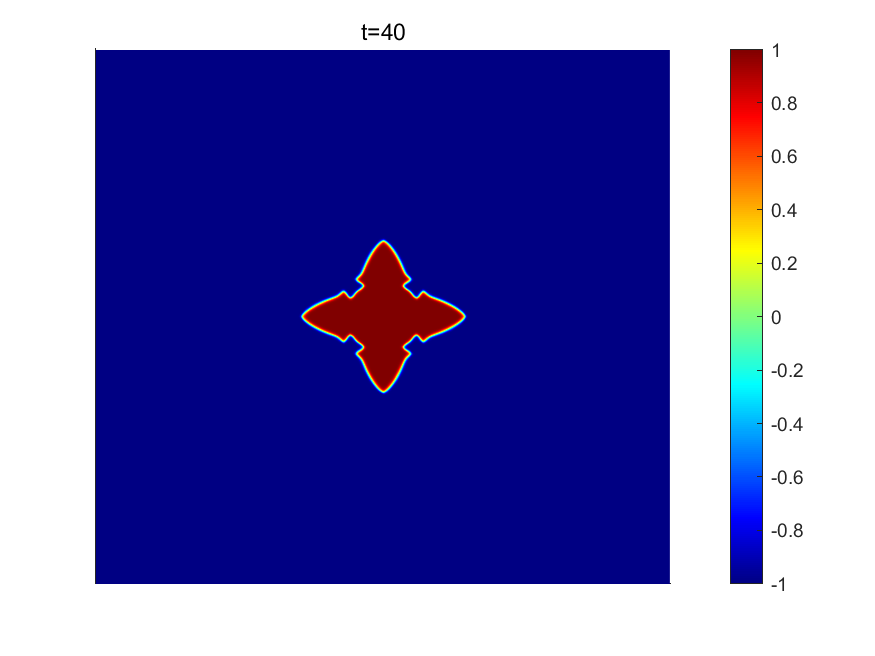}
		\end{minipage}
		\begin{minipage}[t]{0.185\linewidth}
			\centering
			\includegraphics[width=1\linewidth]{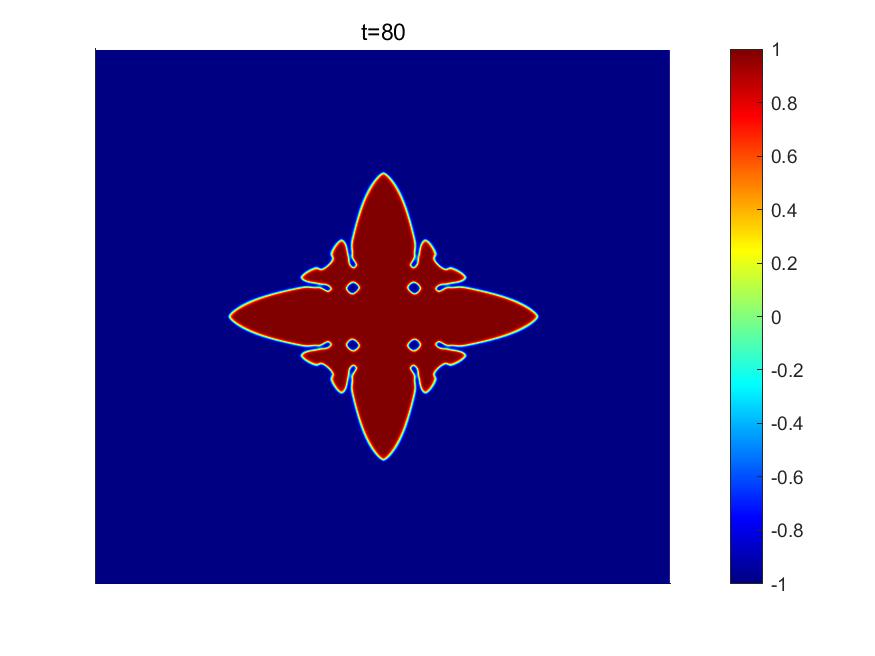}
		\end{minipage}
		\begin{minipage}[t]{0.185\linewidth}
			\centering
			\includegraphics[width=1\linewidth]{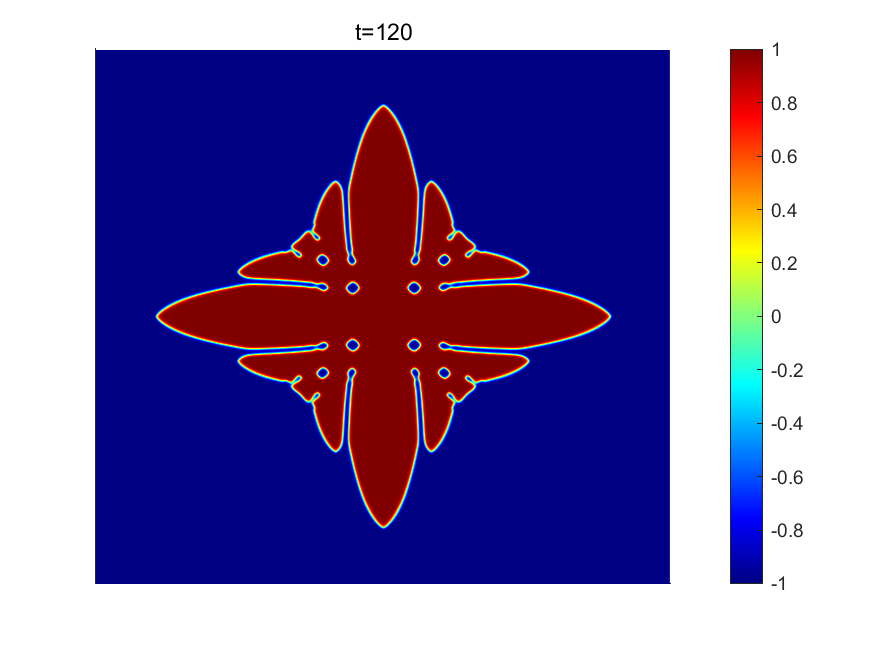}
		\end{minipage}
	}
    \subfigure[\scriptsize (b) Snapshots of the temperature $u$.]{
		\begin{minipage}[t]{0.185\linewidth}
			\centering
			\includegraphics[width=1\linewidth]{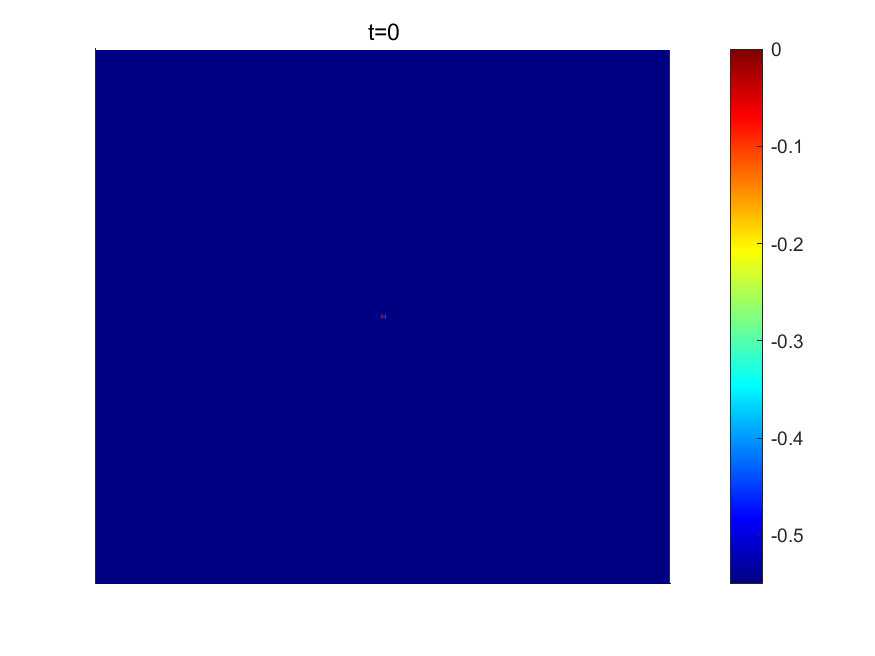}
		\end{minipage}
		\begin{minipage}[t]{0.185\linewidth}
			\centering
			\includegraphics[width=1\linewidth]{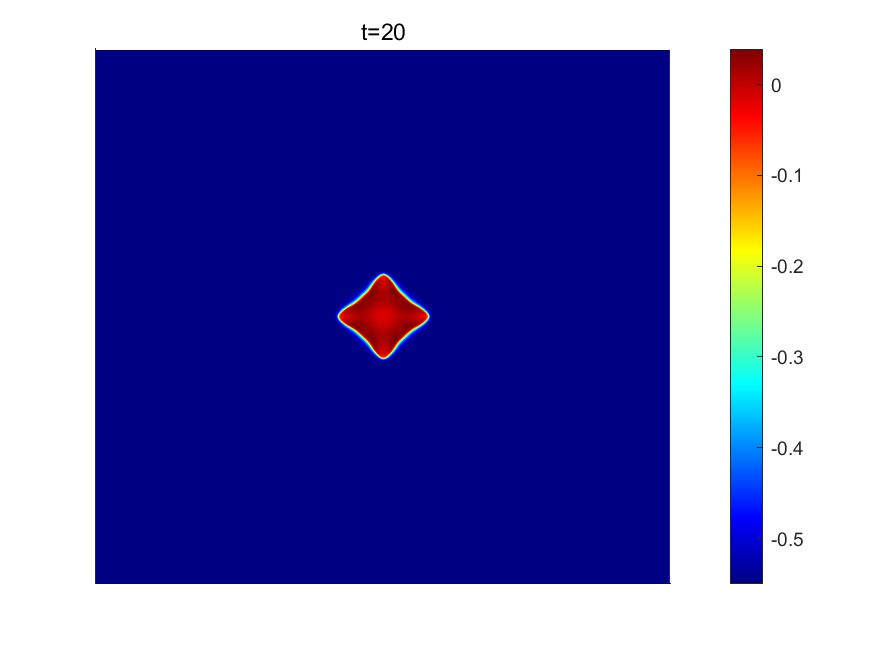}
		\end{minipage}
		\begin{minipage}[t]{0.185\linewidth}
			\centering
			\includegraphics[width=1\linewidth]{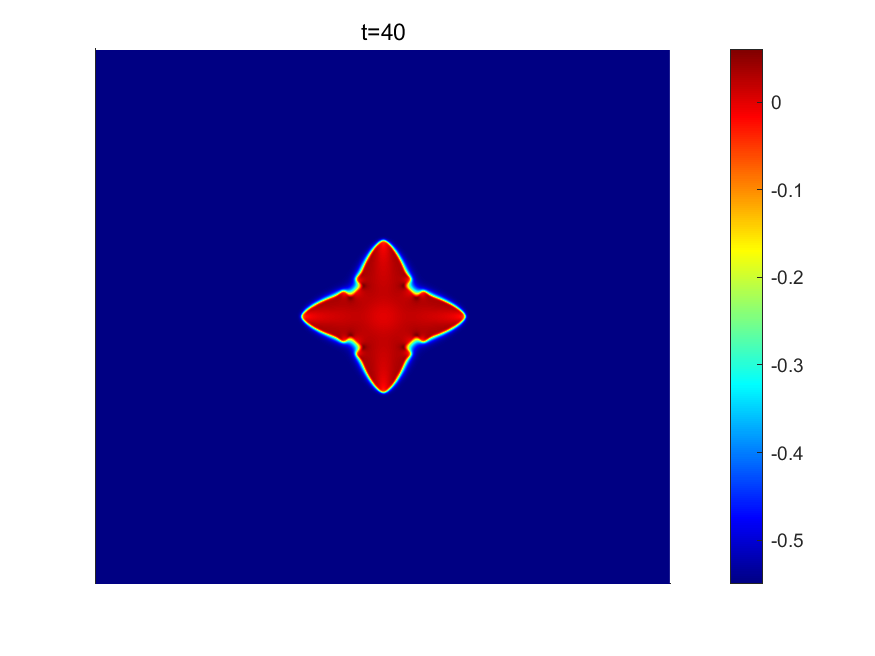}
		\end{minipage}
		\begin{minipage}[t]{0.185\linewidth}
			\centering
			\includegraphics[width=1\linewidth]{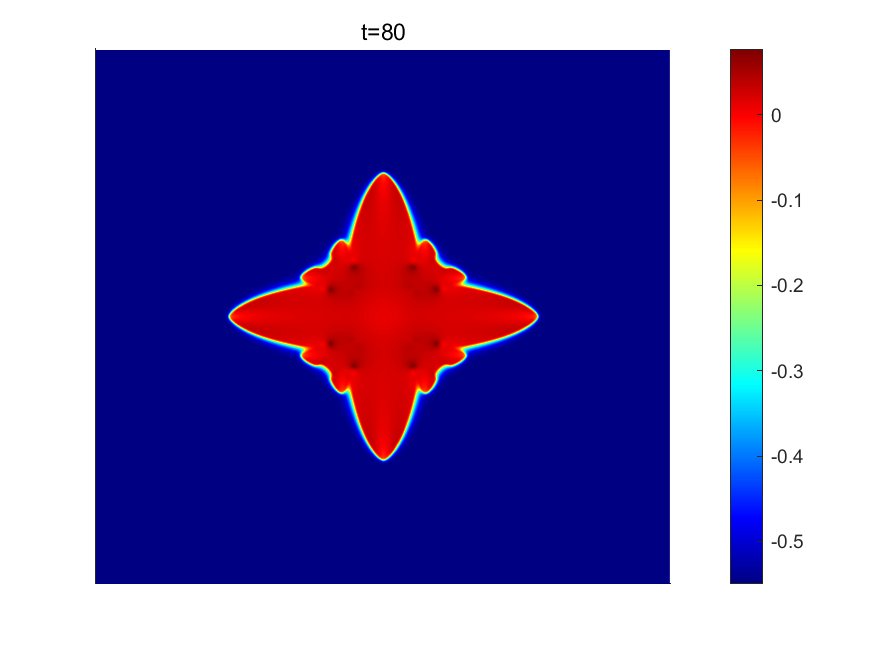}
		\end{minipage}
		\begin{minipage}[t]{0.185\linewidth}
			\centering
			\includegraphics[width=1\linewidth]{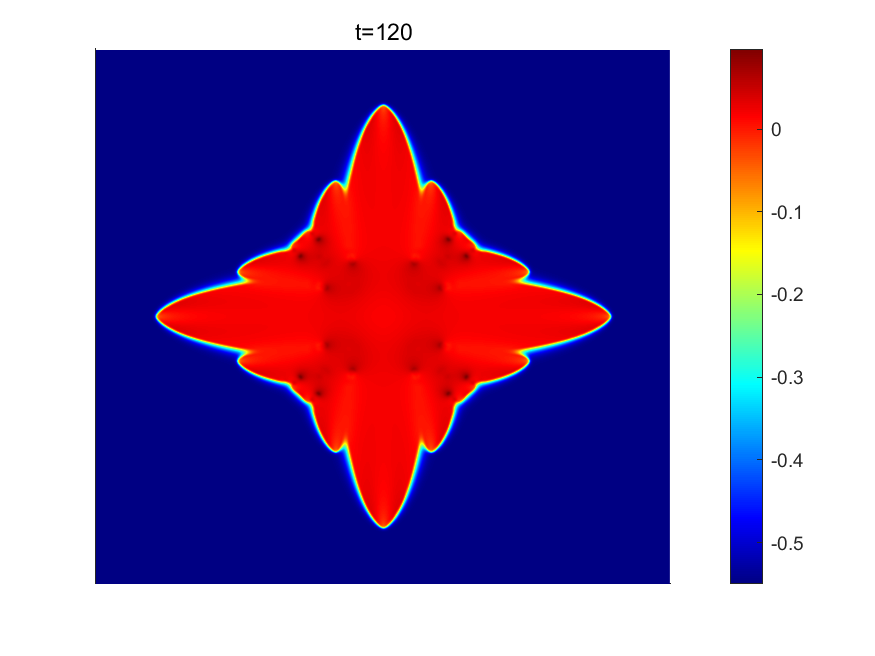}
		\end{minipage}
	}
	\caption{(Example \ref{exp.4}) Dendritic crystal growth in time with fourfold anisotropy $\beta=4$ and latent heat parameter $K=0.6$.}
	\label{fig:fourfoldK06}
\end{figure}

\begin{figure}[htbp]
	\vspace{-0.35cm}
	\subfigtopskip=2pt
	\subfigbottomskip=2pt
	\subfigcapskip=-5pt
	\centering	
    \subfigure[\scriptsize (a) Snapshots of the phase field $\phi$.]{
		\begin{minipage}[t]{0.185\linewidth}
			\centering
			\includegraphics[width=1\linewidth]{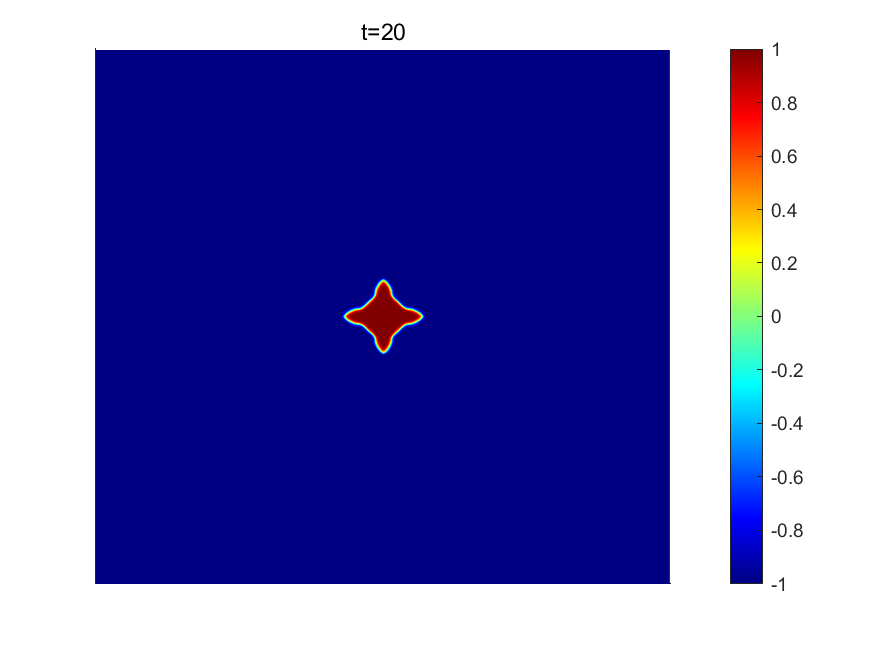}
		\end{minipage}
		\begin{minipage}[t]{0.185\linewidth}
			\centering
			\includegraphics[width=1\linewidth]{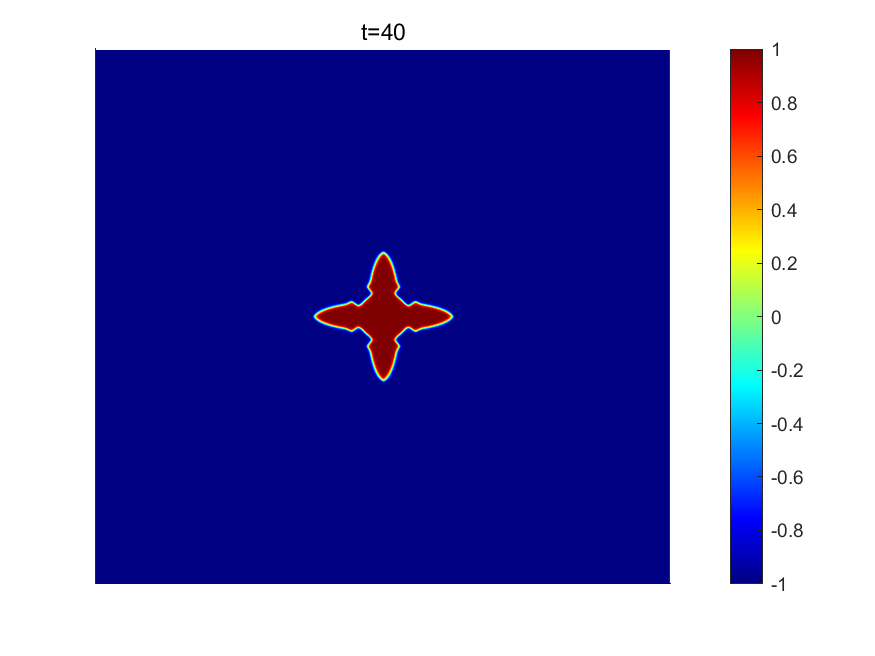}
		\end{minipage}
		\begin{minipage}[t]{0.185\linewidth}
			\centering
			\includegraphics[width=1\linewidth]{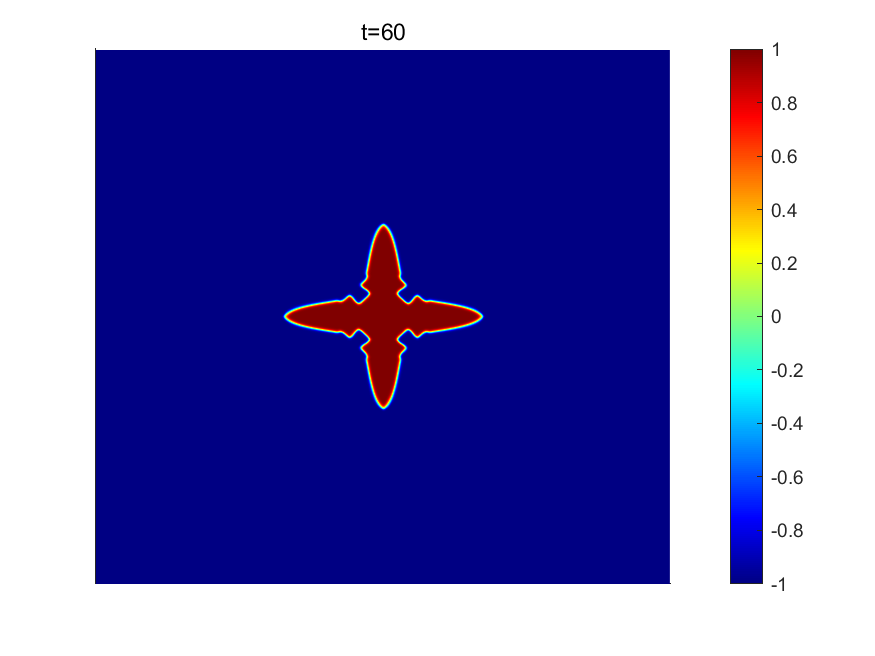}
		\end{minipage}
		\begin{minipage}[t]{0.185\linewidth}
			\centering
			\includegraphics[width=1\linewidth]{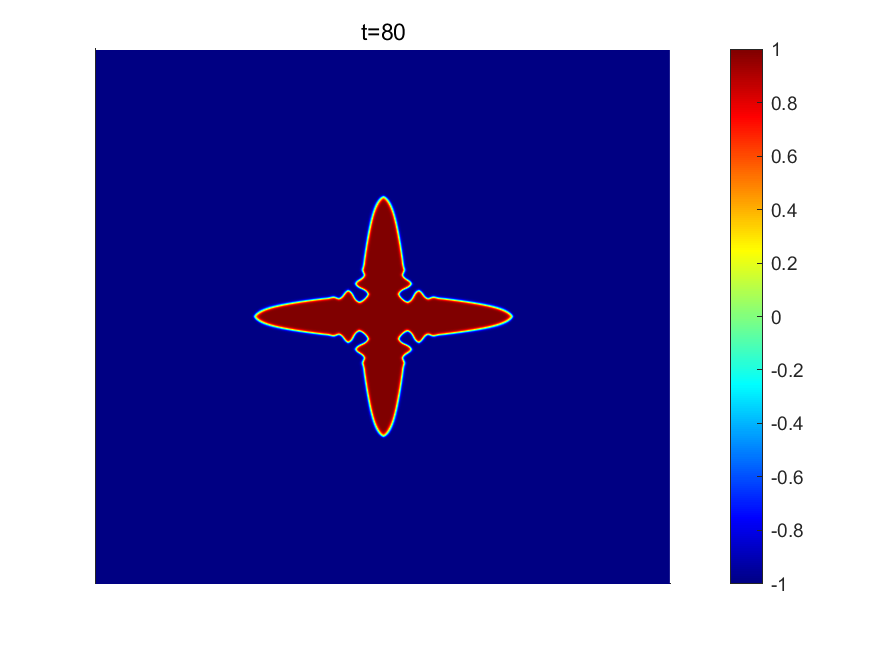}
		\end{minipage}
		\begin{minipage}[t]{0.185\linewidth}
			\centering
			\includegraphics[width=1\linewidth]{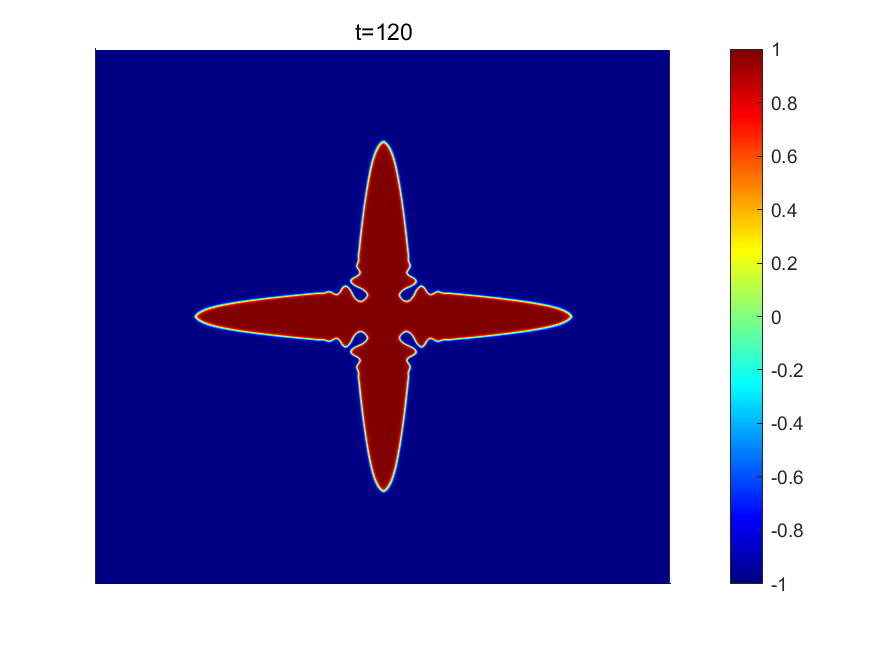}
		\end{minipage}
	}
    \subfigure[\scriptsize (b) Snapshots of the temperature $u$.]{
		\begin{minipage}[t]{0.185\linewidth}
			\centering
			\includegraphics[width=1\linewidth]{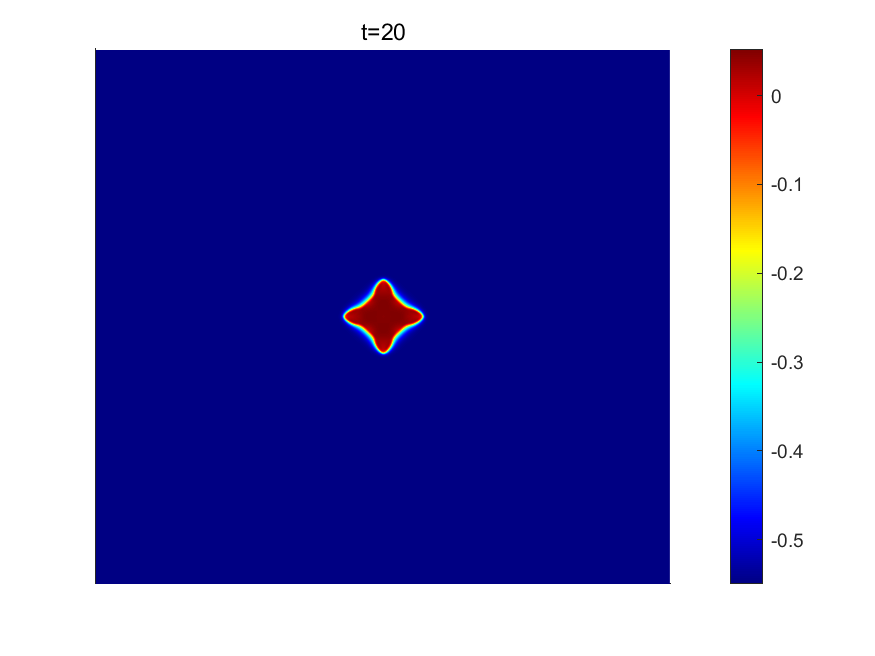}
		\end{minipage}
		\begin{minipage}[t]{0.185\linewidth}
			\centering
			\includegraphics[width=1\linewidth]{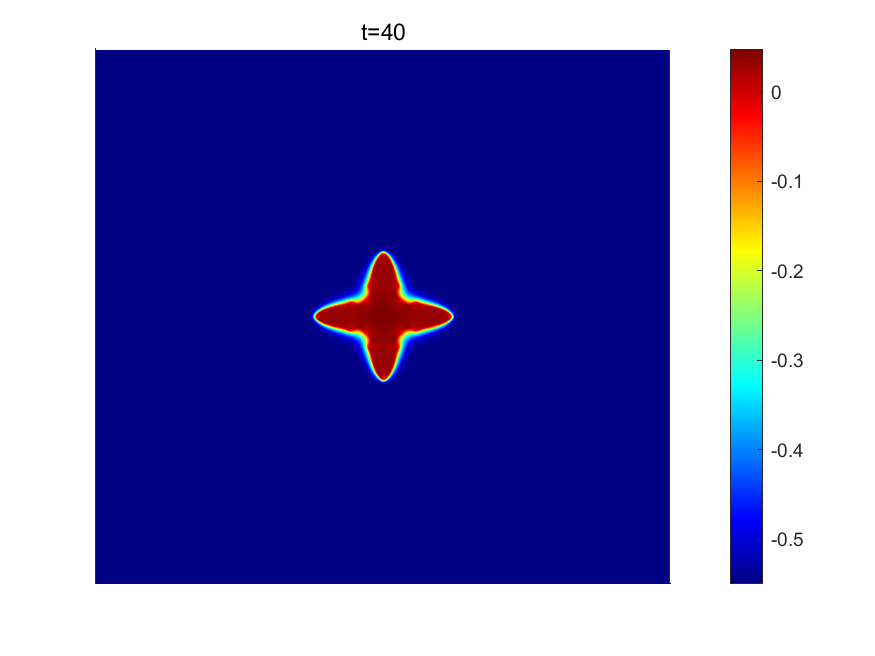}
		\end{minipage}
		\begin{minipage}[t]{0.185\linewidth}
			\centering
			\includegraphics[width=1\linewidth]{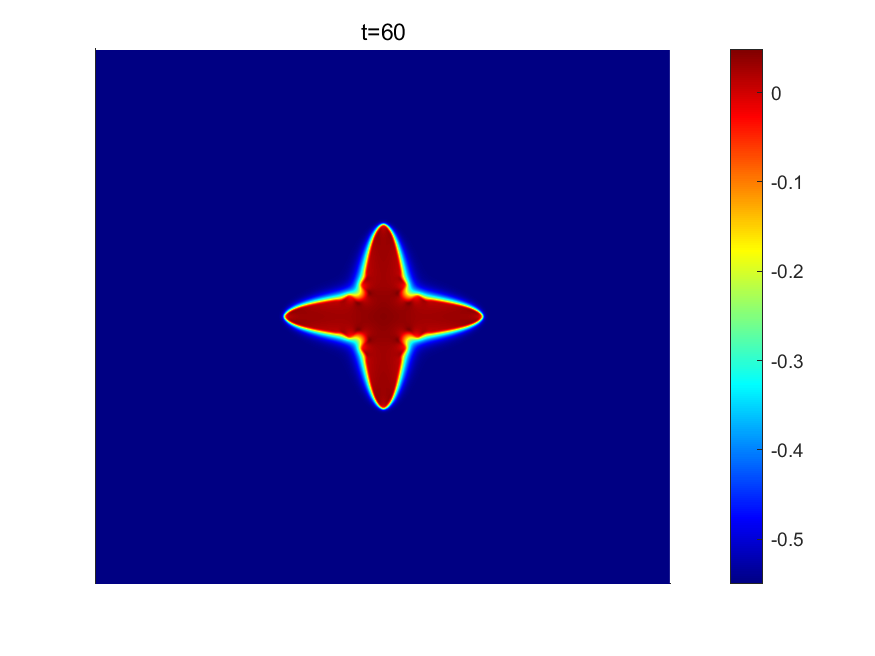}
		\end{minipage}
		\begin{minipage}[t]{0.185\linewidth}
			\centering
			\includegraphics[width=1\linewidth]{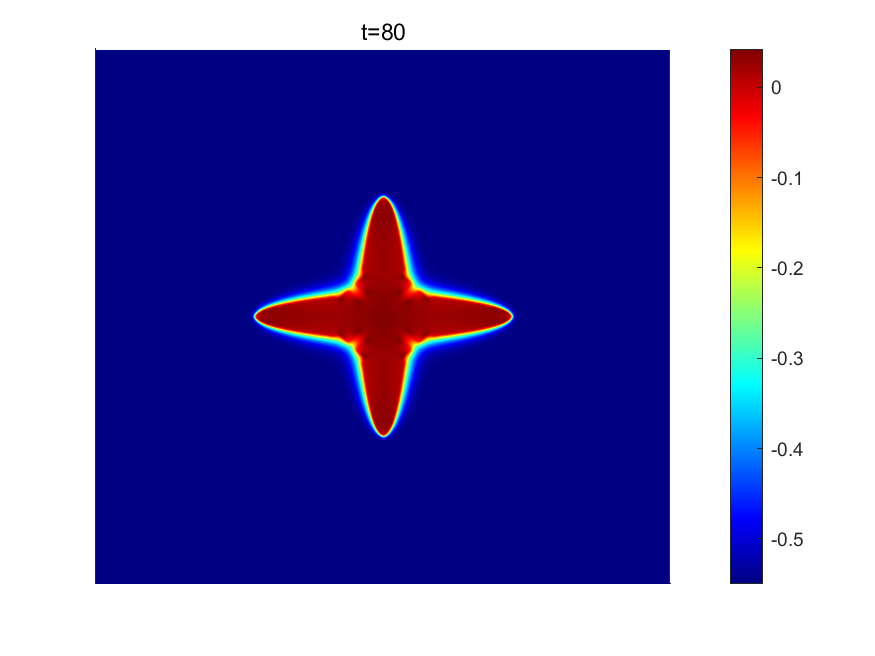}
		\end{minipage}
		\begin{minipage}[t]{0.185\linewidth}
			\centering
			\includegraphics[width=1\linewidth]{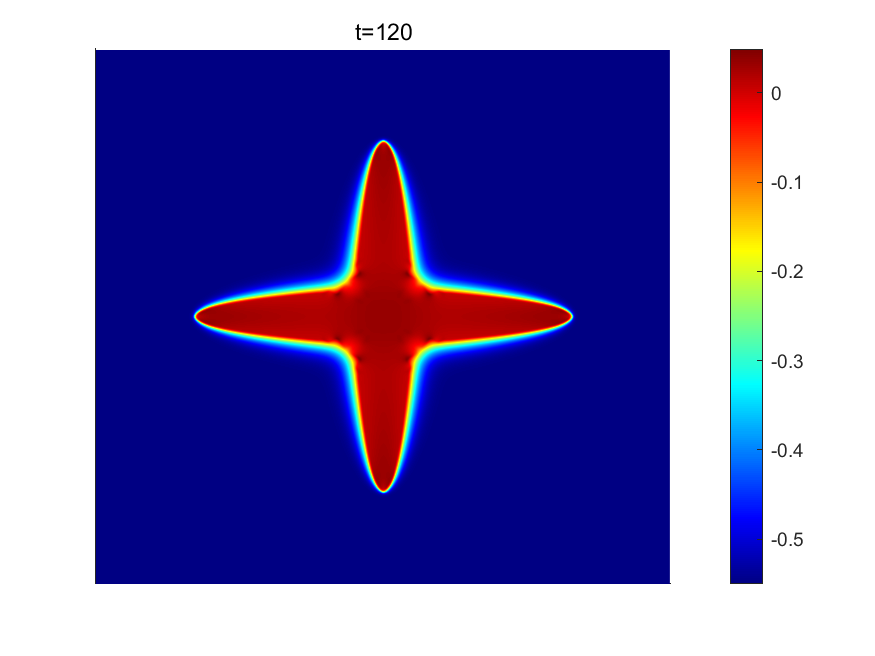}
		\end{minipage}
	}
	\caption{(Example \ref{exp.4}) Dendritic crystal growth in time with fourfold anisotropy $\beta=4$ and the latent heat parameter $K=0.7$.}
	\label{fig:fourfoldK07}
\end{figure}

\begin{figure}[htbp]
	\vspace{-0.35cm}
	\subfigtopskip=2pt
	\subfigbottomskip=2pt
	\subfigcapskip=-5pt
	\centering	
    \subfigure[\scriptsize (a) Snapshots of the phase field $\phi$.]{
		\begin{minipage}[t]{0.185\linewidth}
			\centering
			\includegraphics[width=1\linewidth]{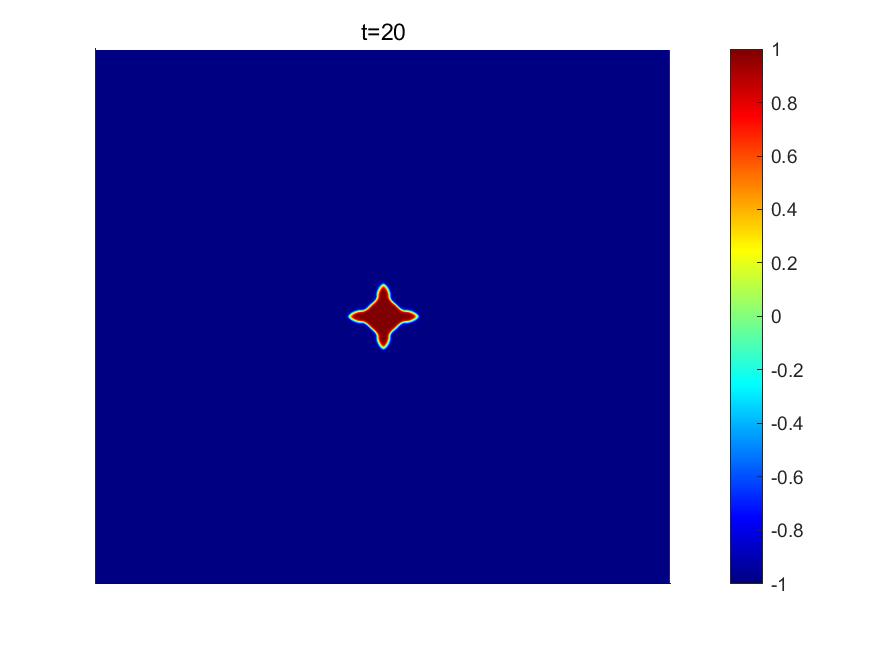}
		\end{minipage}
		\begin{minipage}[t]{0.185\linewidth}
			\centering
			\includegraphics[width=1\linewidth]{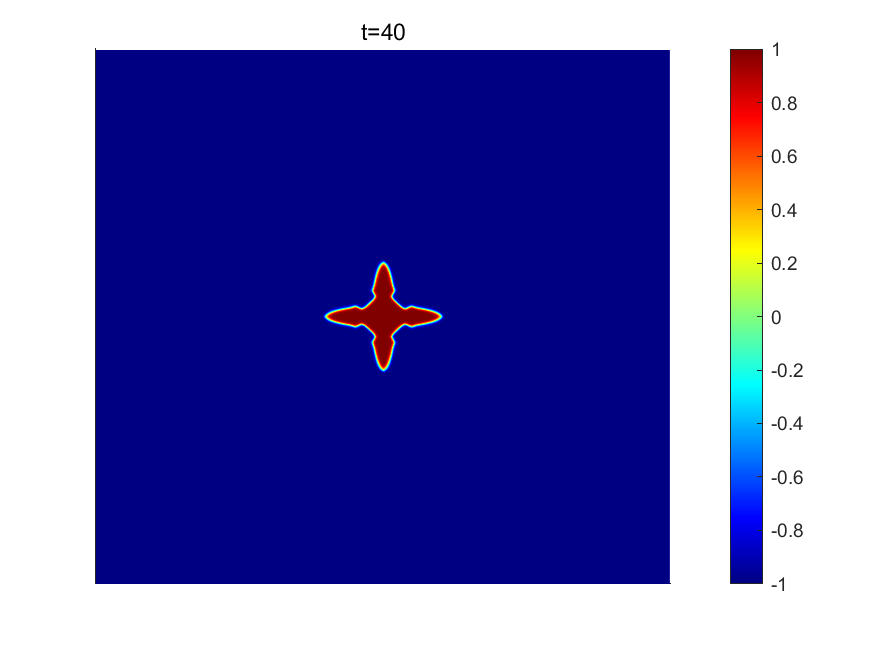}
		\end{minipage}
		\begin{minipage}[t]{0.185\linewidth}
			\centering
			\includegraphics[width=1\linewidth]{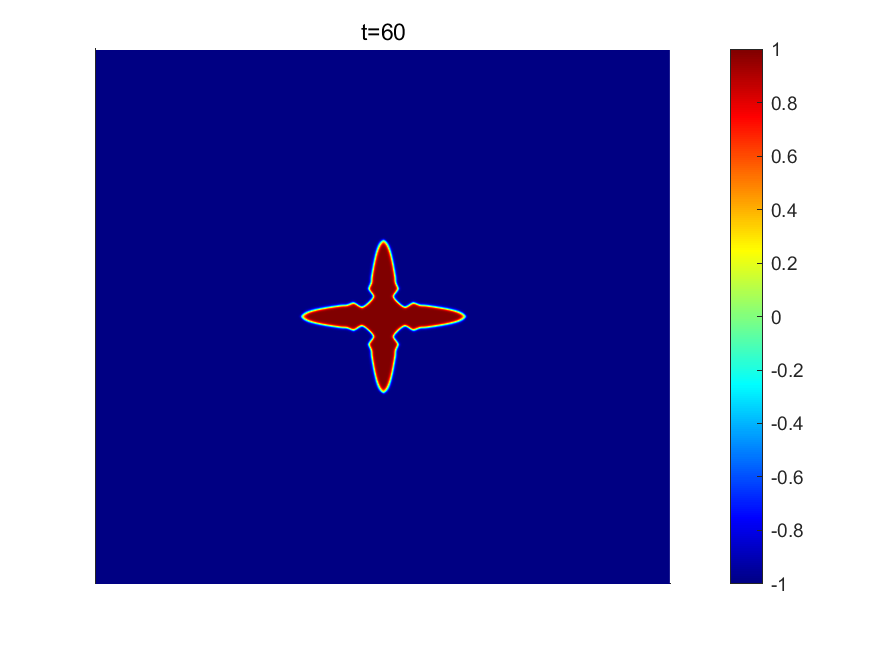}
		\end{minipage}
		\begin{minipage}[t]{0.185\linewidth}
			\centering
			\includegraphics[width=1\linewidth]{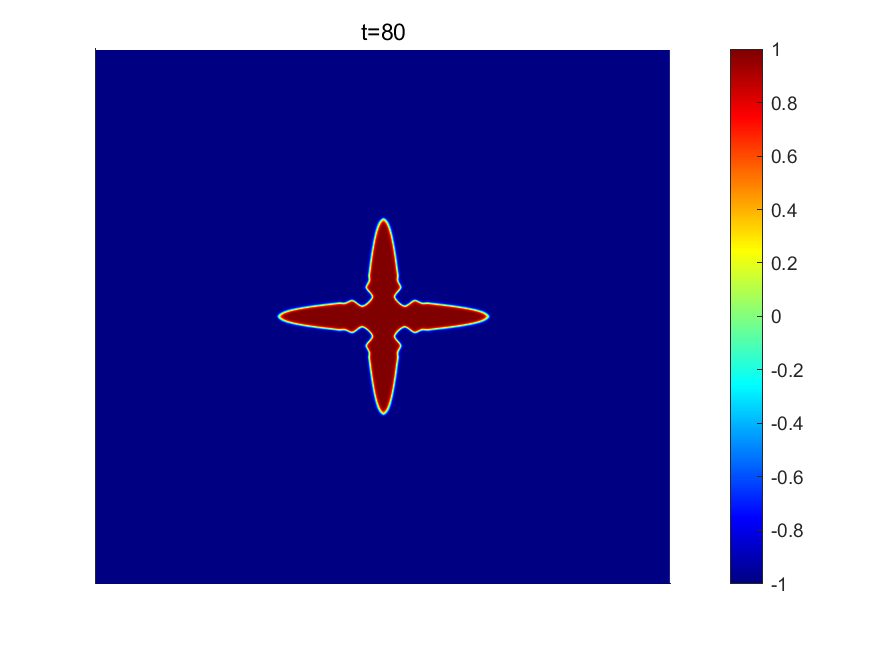}
		\end{minipage}
		\begin{minipage}[t]{0.185\linewidth}
			\centering
			\includegraphics[width=1\linewidth]{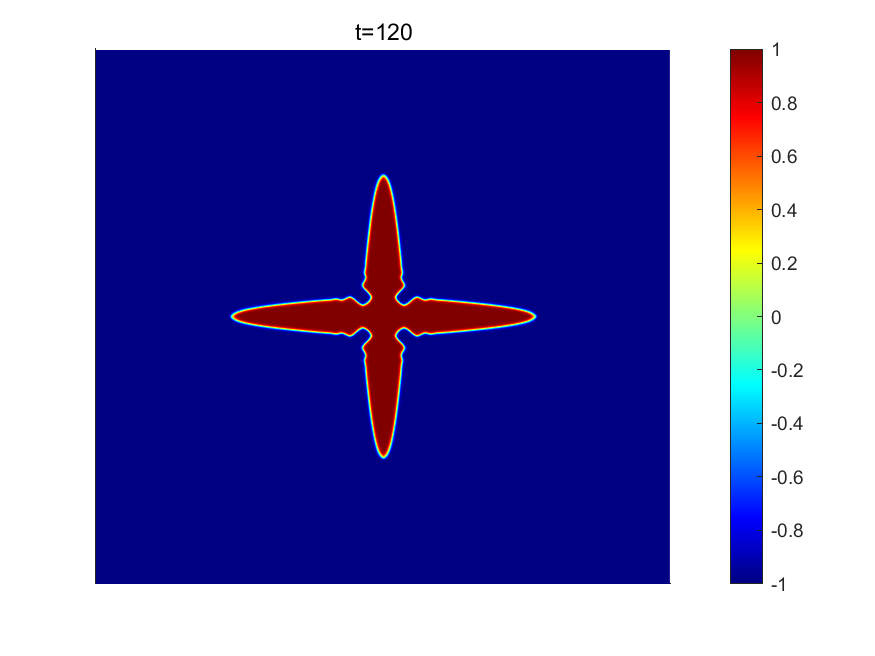}
		\end{minipage}
	}
    \subfigure[\scriptsize (b) Snapshots of the temperature field $u$.]{
		\begin{minipage}[t]{0.185\linewidth}
			\centering
			\includegraphics[width=1\linewidth]{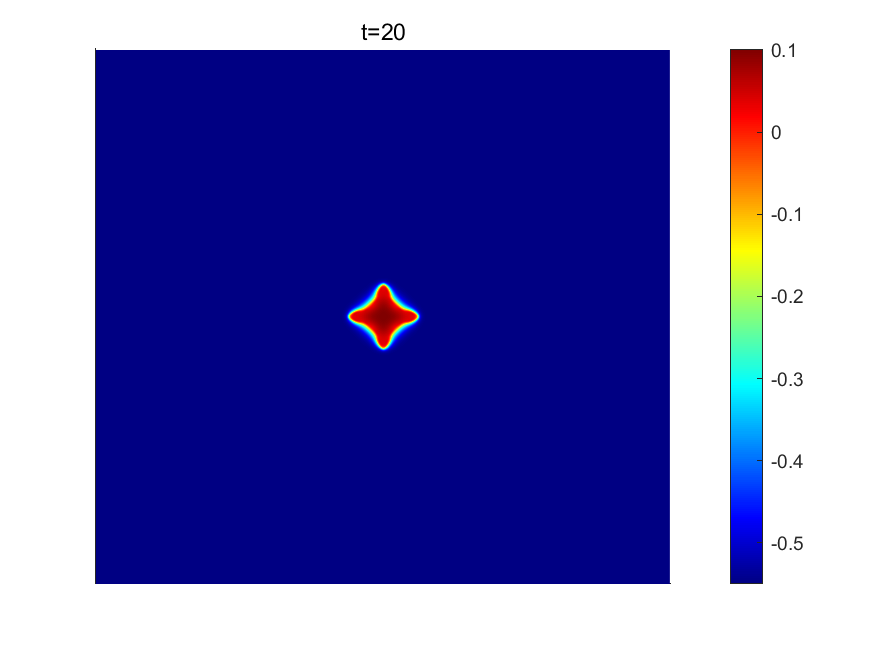}
		\end{minipage}
		\begin{minipage}[t]{0.185\linewidth}
			\centering
			\includegraphics[width=1\linewidth]{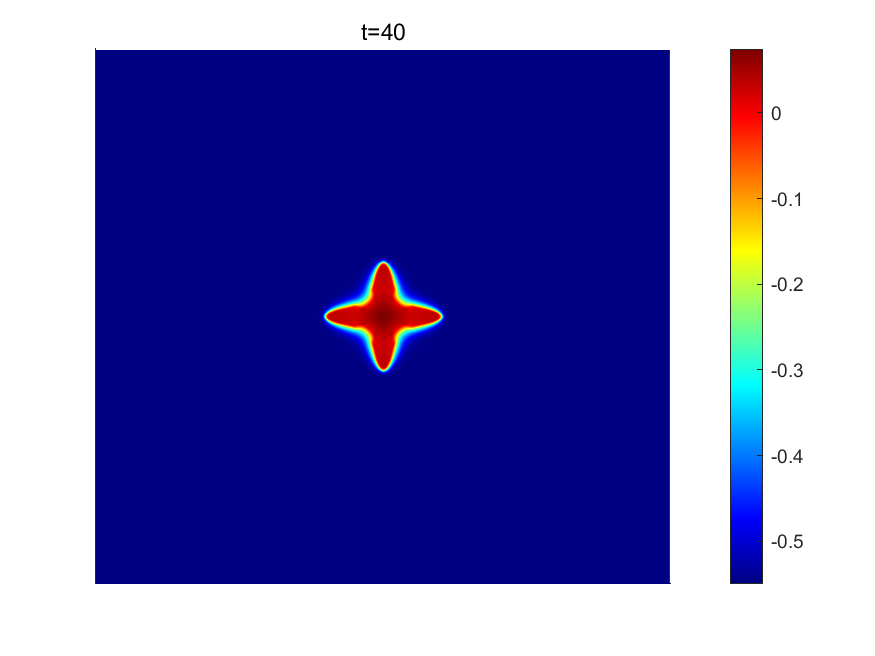}
		\end{minipage}
		\begin{minipage}[t]{0.185\linewidth}
			\centering
			\includegraphics[width=1\linewidth]{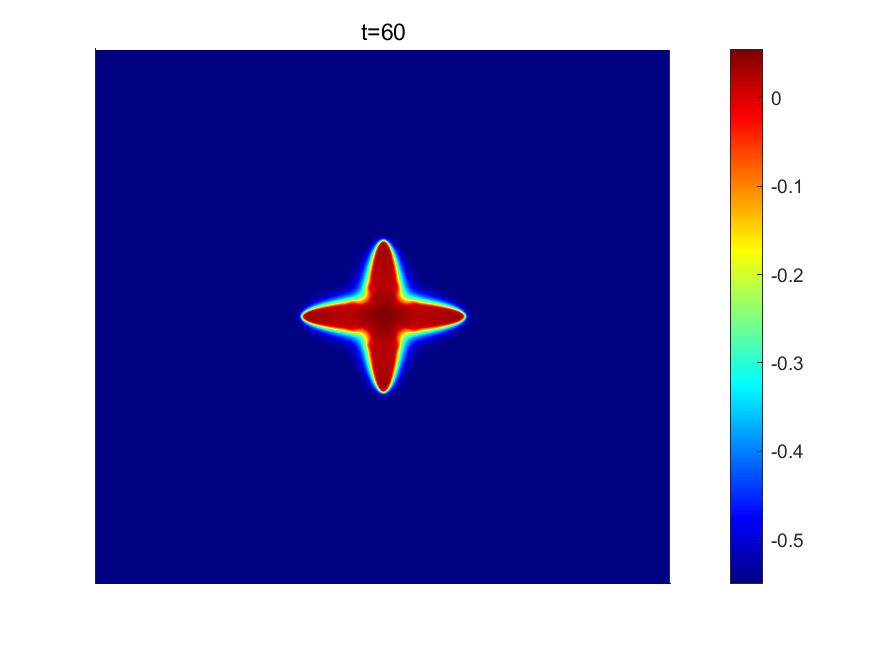}
		\end{minipage}
		\begin{minipage}[t]{0.185\linewidth}
			\centering
			\includegraphics[width=1\linewidth]{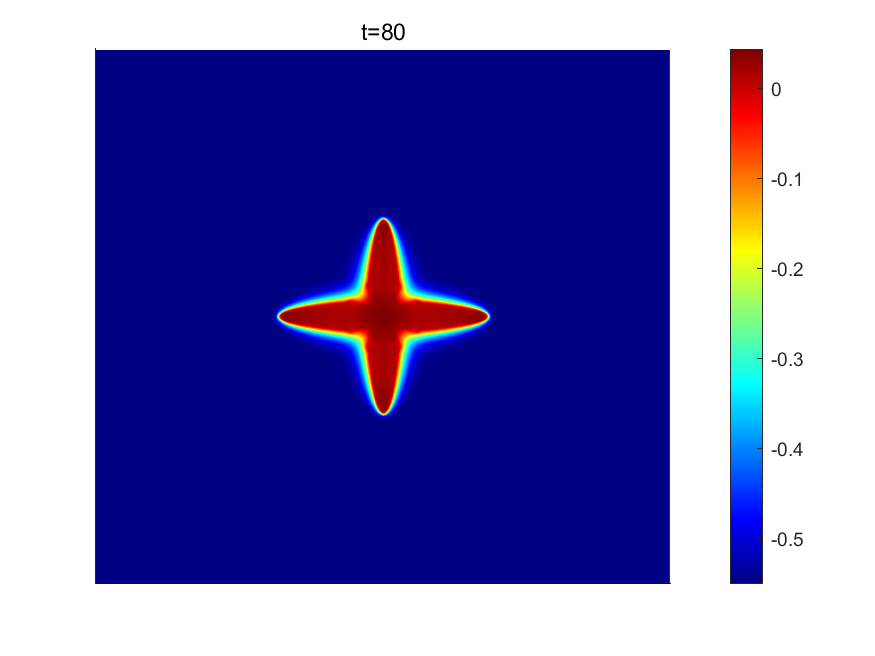}
		\end{minipage}
		\begin{minipage}[t]{0.185\linewidth}
			\centering
			\includegraphics[width=1\linewidth]{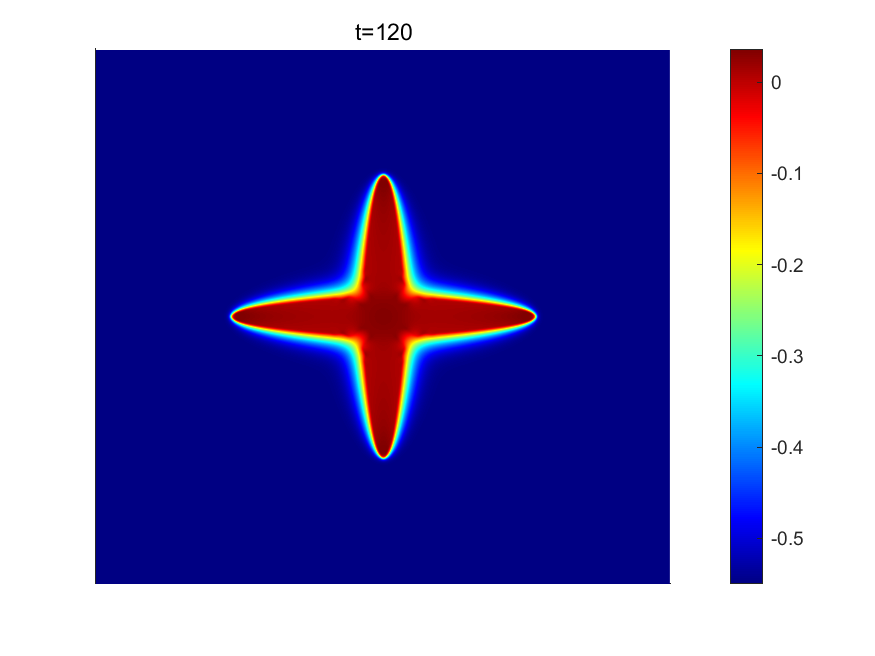}
		\end{minipage}
	}
	\caption{(Example \ref{exp.4}) Dendritic crystal growth in time with fourfold anisotropy $\beta=4$ and the latent heat parameter $K=0.8$.}
	\label{fig:fourfoldK08}
\end{figure}

\begin{figure}[h]
\subfigure[\scriptsize (a) Energy evolution]{
\begin{minipage}{0.45 \textwidth}
\centering
\includegraphics[width=\textwidth]{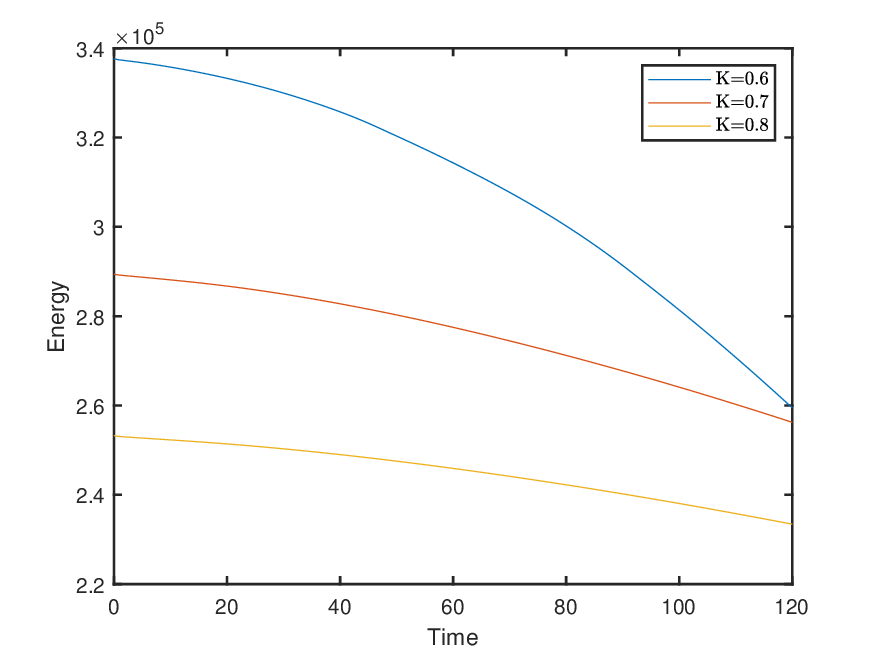}
\end{minipage}
}
\subfigure[\scriptsize (b) Crystal area]{
\begin{minipage}{0.45 \textwidth}
\centering
\includegraphics[width=\textwidth]{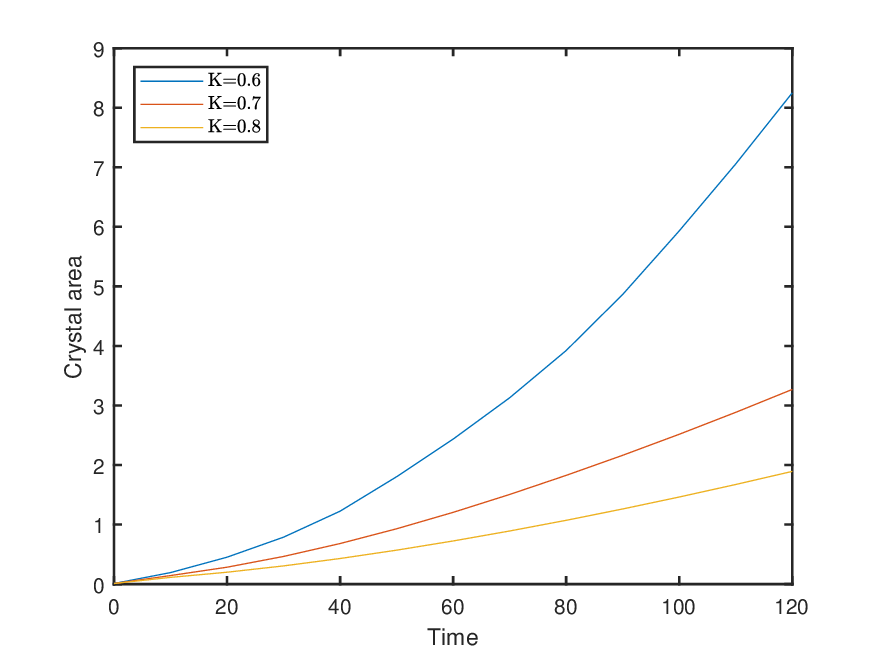}
\end{minipage}
}
%
\caption{(Example \ref{exp.4}) (a) Time evolution of the energy; (b) crystal
area $\int_{\Omega}\frac{1+\phi}{2}dx$.
}
\label{fourfold_energy}
\end{figure}

 \begin{example}\label{exp.6}
\rm{(Three nuclei fourfold anisotropy crystal growth) We consider the following initial conditions}
\begin{equation*}
\left\{
             \begin{array}{ll}
\phi(x,y,0)=\sum\limits_{i=1}^{3}\tanh\Big(\frac{0.02-\sqrt{(x-x_i)^2+(y-y_i)^2}}{0.072}\Big)+2,\\
u(x,y,0)=\left\{
             \begin{array}{ll}
             0,  &\phi>0,\\
             -0.55, &otherwise,
                          \end{array}
           \right.
             \end{array}
           \right.
\end{equation*}
where $(x_1,y_1)=(\frac{4}{5}\pi,\frac{13}{10}\pi), (x_2,y_2)=(\frac{3}{4}\pi,\frac{3}{4}\pi)$, and $(x_3,y_3)=(\frac{13}{10}\pi,\pi)$.
The parameters are set as before.
\end{example}
This configuration is used to simulate the growth of three tiny nuclei placed in the square domain.
Figure \ref{fig:three-fourfoldK06} shows the isocontours of the phase field and the
temperature for $K=0.6$.
It is observed that three main dendrites with a lot of extruded dendrites
are formed during the formation process.
This is consistent with the findings that already exist \cite{ohno2020quantitative}.

\begin{figure}[htbp]
	\vspace{-0.35cm}
	\subfigtopskip=2pt
	\subfigbottomskip=2pt
	\subfigcapskip=-5pt
	\centering	
    \subfigure[\scriptsize (a) Snapshots of the phase field $\phi$.]{
		\begin{minipage}[t]{0.185\linewidth}
			\centering
			\includegraphics[width=1\linewidth]{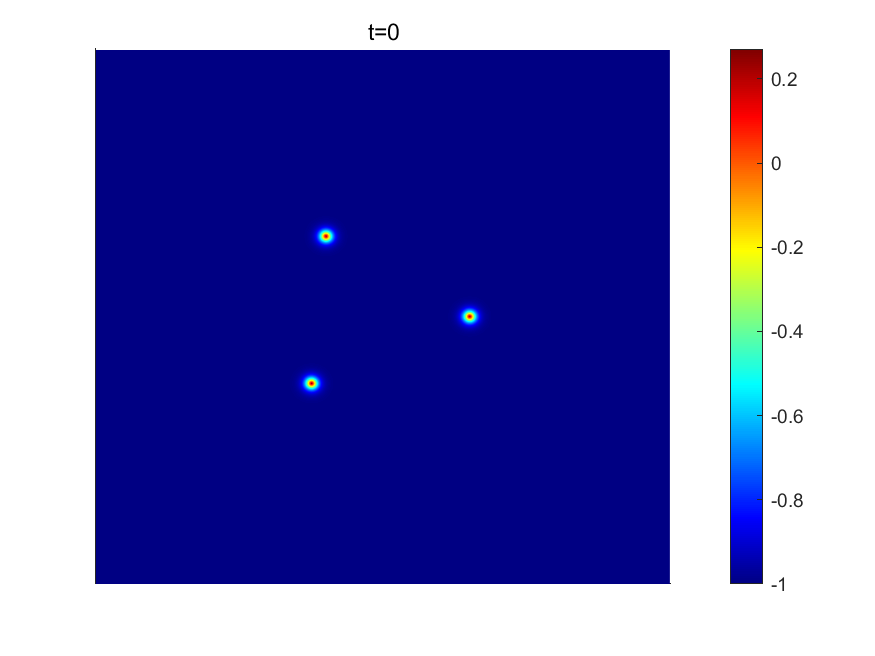}
		\end{minipage}
		\begin{minipage}[t]{0.185\linewidth}
			\centering
			\includegraphics[width=1\linewidth]{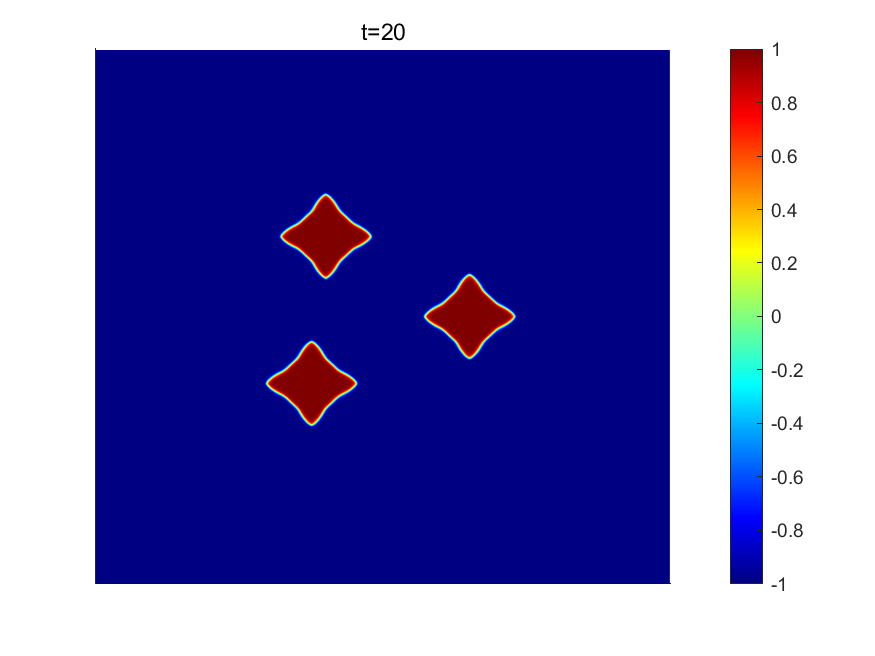}
		\end{minipage}
		\begin{minipage}[t]{0.185\linewidth}
			\centering
			\includegraphics[width=1\linewidth]{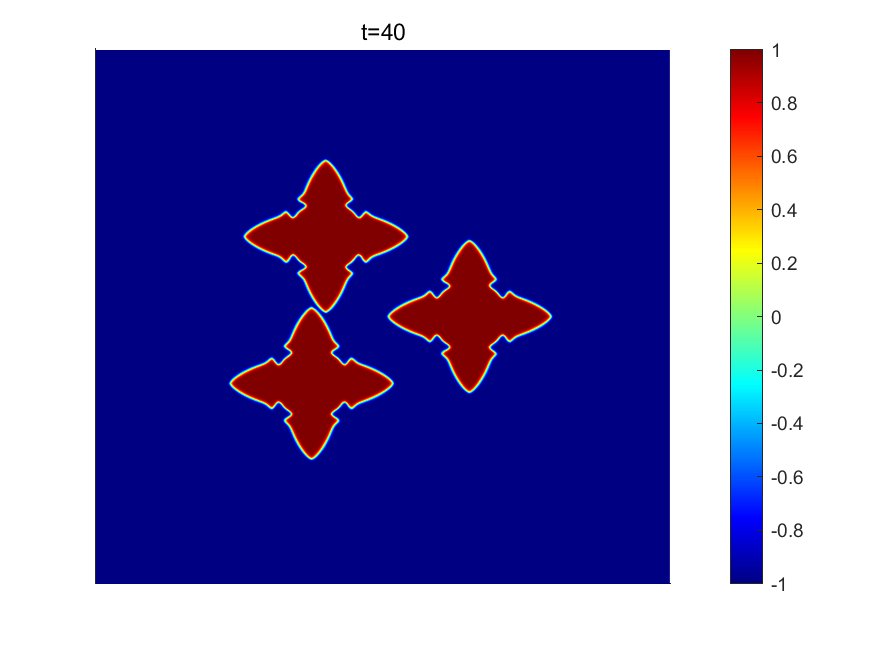}
		\end{minipage}
		\begin{minipage}[t]{0.185\linewidth}
			\centering
			\includegraphics[width=1\linewidth]{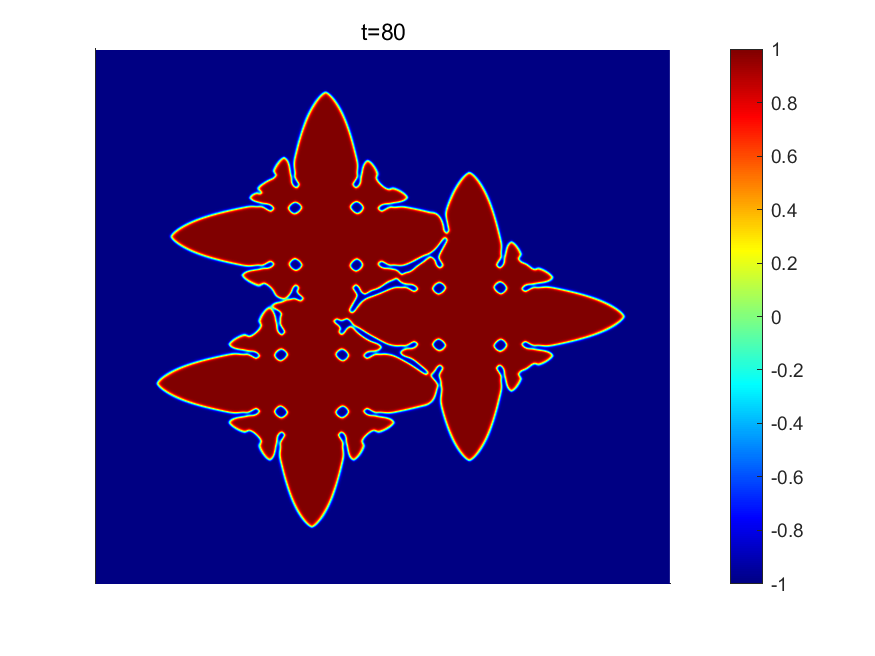}
		\end{minipage}
		\begin{minipage}[t]{0.185\linewidth}
			\centering
			\includegraphics[width=1\linewidth]{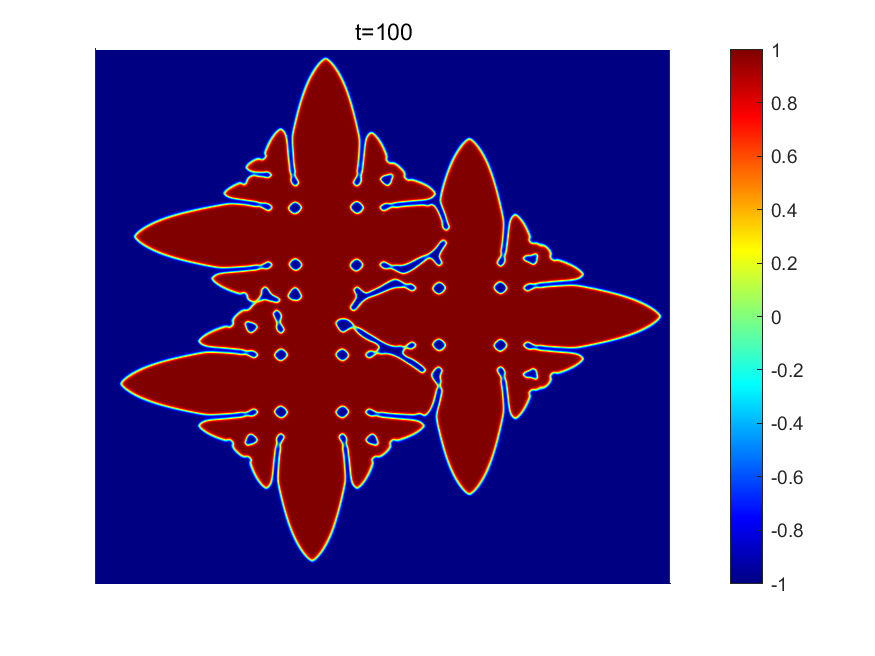}
		\end{minipage}
	}
    \subfigure[\scriptsize (b) Snapshots of the temperature field $u$.]{

		\begin{minipage}[t]{0.185\linewidth}
			\centering
			\includegraphics[width=1\linewidth]{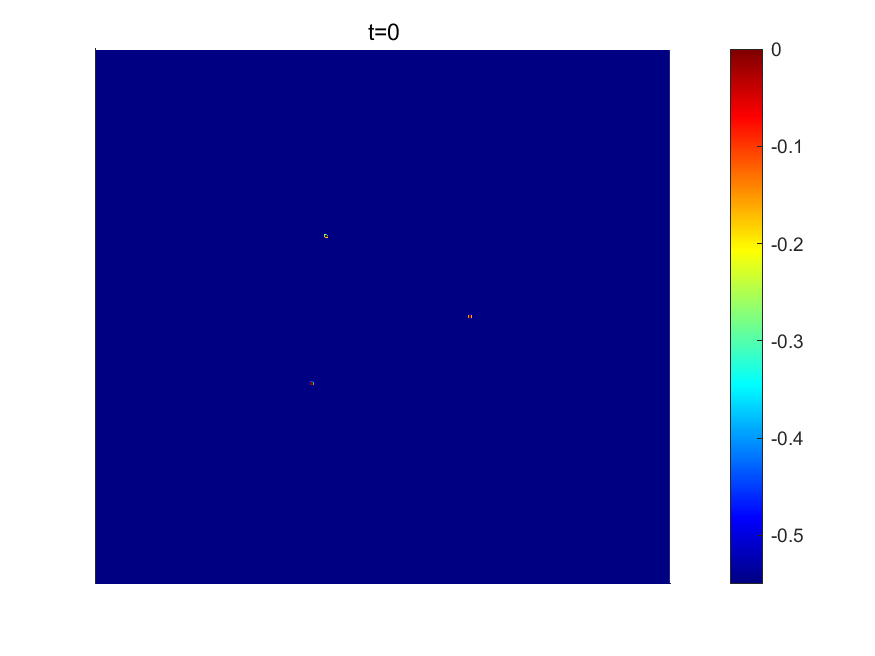}
		\end{minipage}
		\begin{minipage}[t]{0.185\linewidth}
			\centering
			\includegraphics[width=1\linewidth]{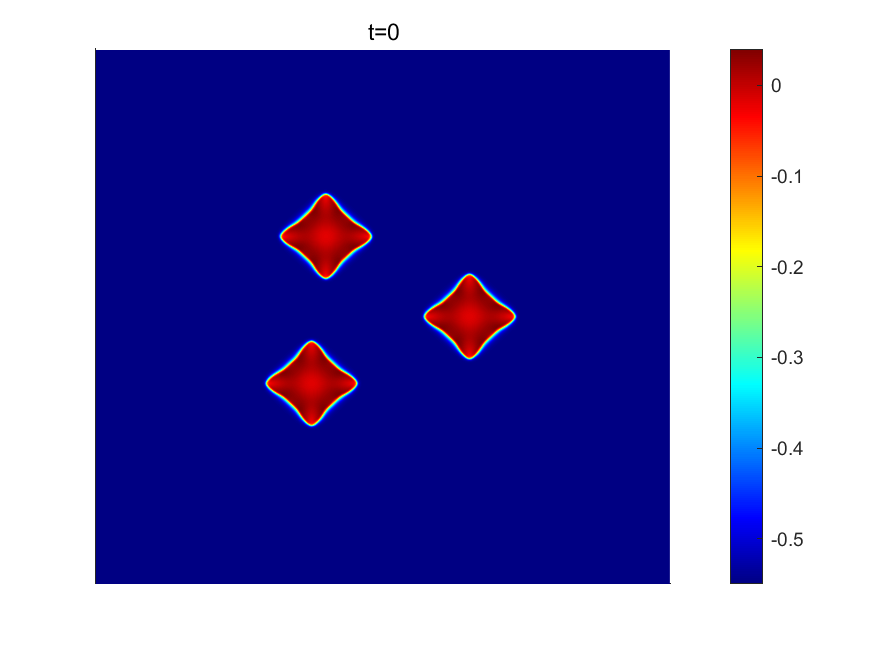}
		\end{minipage}
		\begin{minipage}[t]{0.185\linewidth}
			\centering
			\includegraphics[width=1\linewidth]{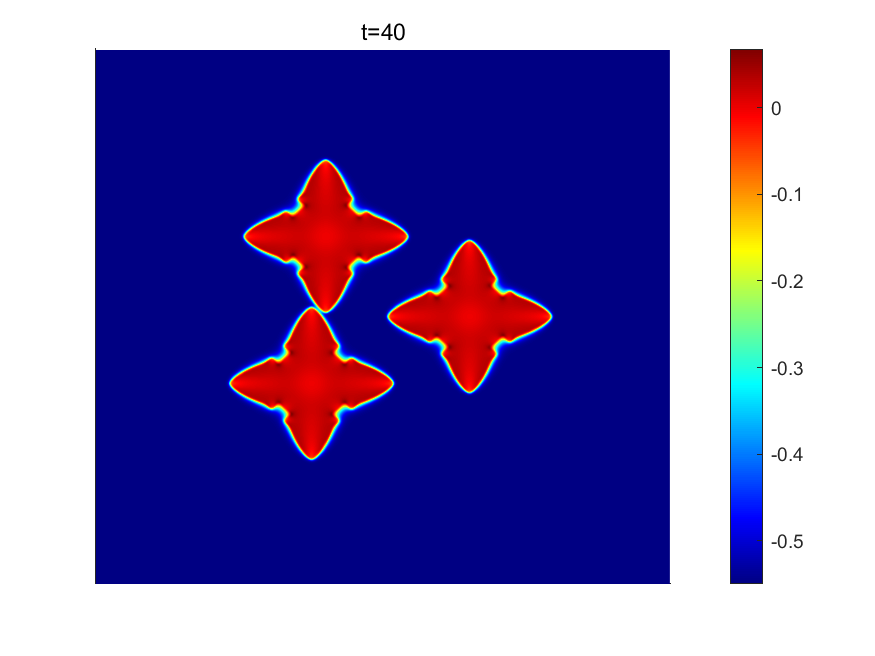}
		\end{minipage}
		\begin{minipage}[t]{0.185\linewidth}
			\centering
			\includegraphics[width=1\linewidth]{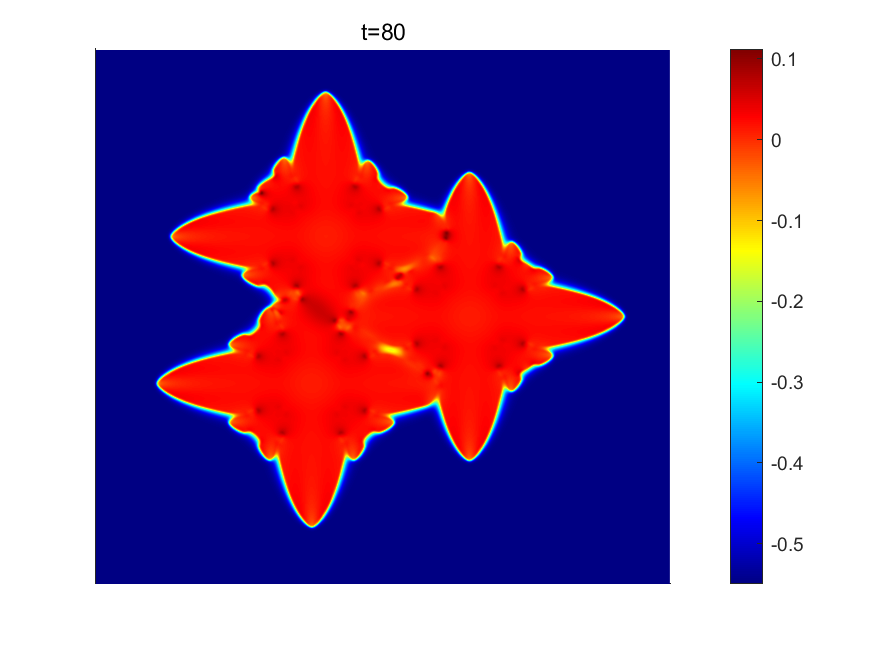}
		\end{minipage}
		\begin{minipage}[t]{0.185\linewidth}
			\centering
			\includegraphics[width=1\linewidth]{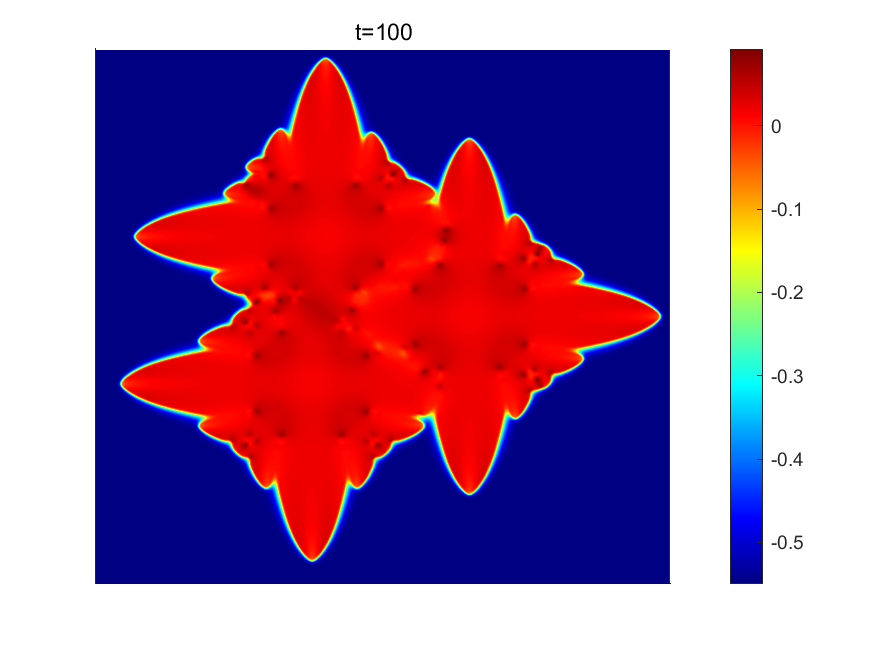}
		\end{minipage}
	}
	\caption{(Example \ref{exp.6}) Fourfold anisotropy dendritic crystal growth starting with three deposited nuclei.}
	\label{fig:three-fourfoldK06}
\end{figure}


\subsection{Sixfold dendrite crystal growth in 2D}
Sixfold dendrite crystal growth is simulated by setting $\beta=6$ with
other parameters taken same as Example \ref{exp.4}.

We present in Figures \ref{fig:sixfoldK06}-\ref{fig:sixfoldK07} the isocontours of
$\{\phi = 0\}$ and the temperature $u$ for $K=0.6,0.65,0.7$.
We observe that the small circle at the initial time first evolves
into a hexagon shape, then into a snowflake with six branches.
There are microscopic structures that emerge in the growth process
due to the anisotropy in the heat transfer.
Comparing the crystal shapes in Figure \ref{fig:sixfoldK06}, Figure \ref{fig:sixfoldK065},  and Figure \ref{fig:sixfoldK07} for different $K$, we see
that the bigger is $K$, the sharper are snow-shaped tips, the thinner are branches,
and the less subtle are micro structures.
The evolution of the energy and the crystal area are
plotted in Figure \ref{sixfold_energy}, showing the energy dissipation and area increasing in time and in decreasing $K$.
Again the simulation results are consistent with the references \cite{kobayashi1993modeling,karma1998quantitative,yang2020efficient}.

Sixfold dendrite crystal growth with three deposited nuclei is also simulated by
using the same initial conditions as Example \ref{exp.6}.
The result is given in Figure \ref{fig:three-sixfoldK06}.
We observe that three dendrites are eventually created during the formation process, each containing a plentiful amount of compressed branches. We notice that some similar simulation result has been reported in \cite{yang2020efficient}.

\begin{figure}[htbp]
	\vspace{-0.35cm}
	\subfigtopskip=2pt
	\subfigbottomskip=2pt
	\subfigcapskip=-5pt
	\centering	
    \subfigure[\scriptsize (a) Snapshots of the phase field $\phi$.]{

		\begin{minipage}[t]{0.185\linewidth}
			\centering
			\includegraphics[width=1\linewidth]{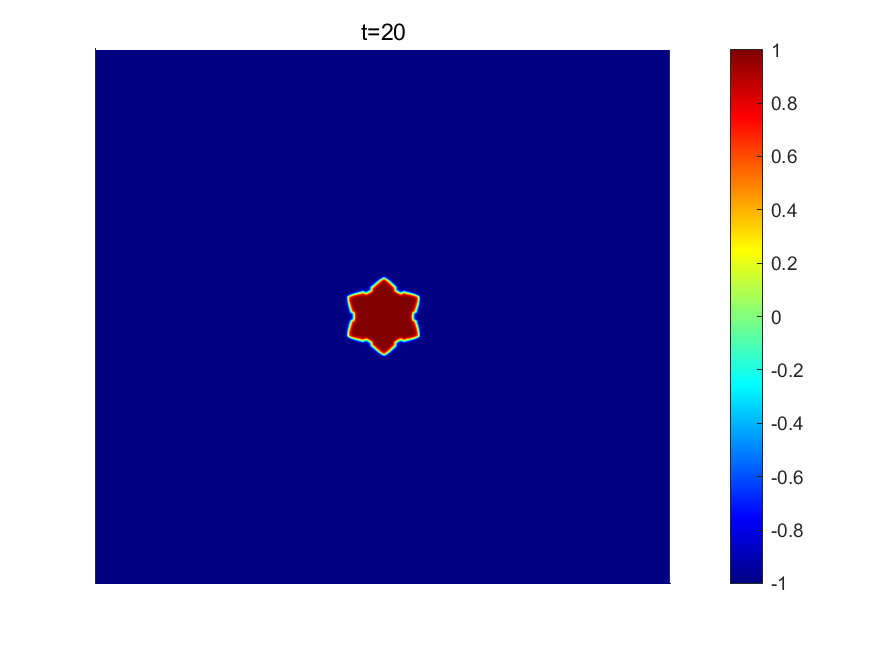}
		\end{minipage}
		\begin{minipage}[t]{0.185\linewidth}
			\centering
			\includegraphics[width=1\linewidth]{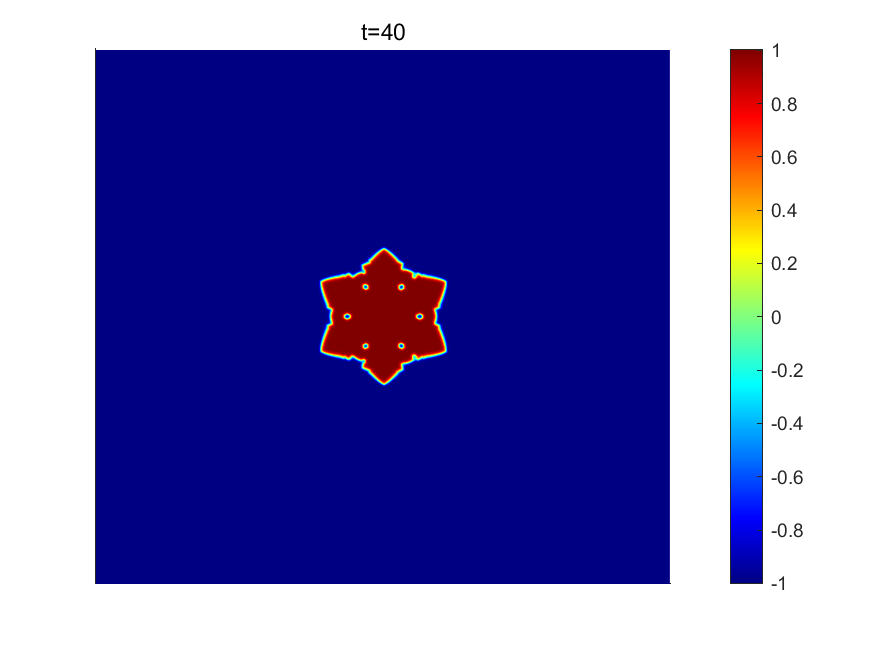}
		\end{minipage}
		\begin{minipage}[t]{0.185\linewidth}
			\centering
			\includegraphics[width=1\linewidth]{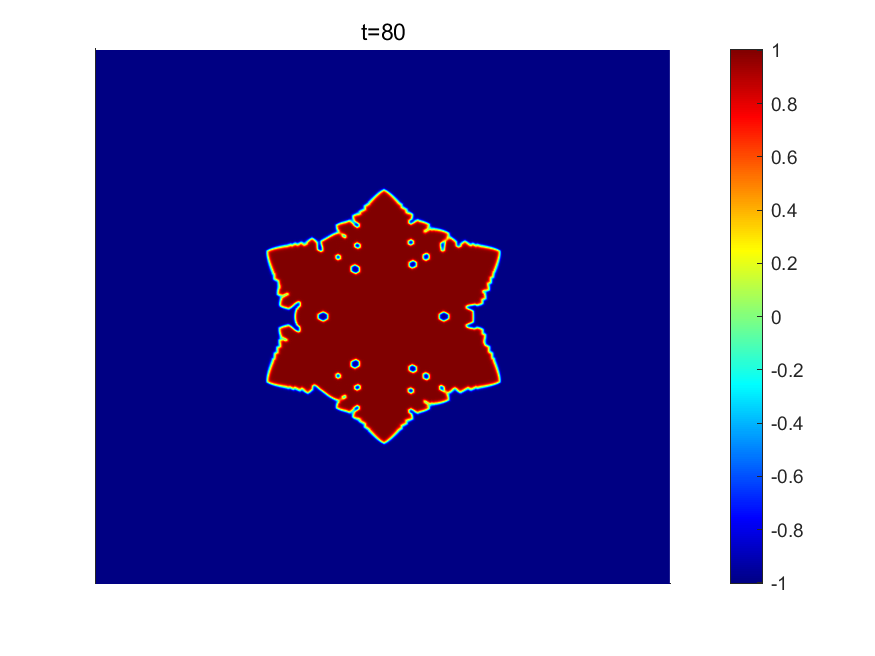}
		\end{minipage}
		\begin{minipage}[t]{0.185\linewidth}
			\centering
			\includegraphics[width=1\linewidth]{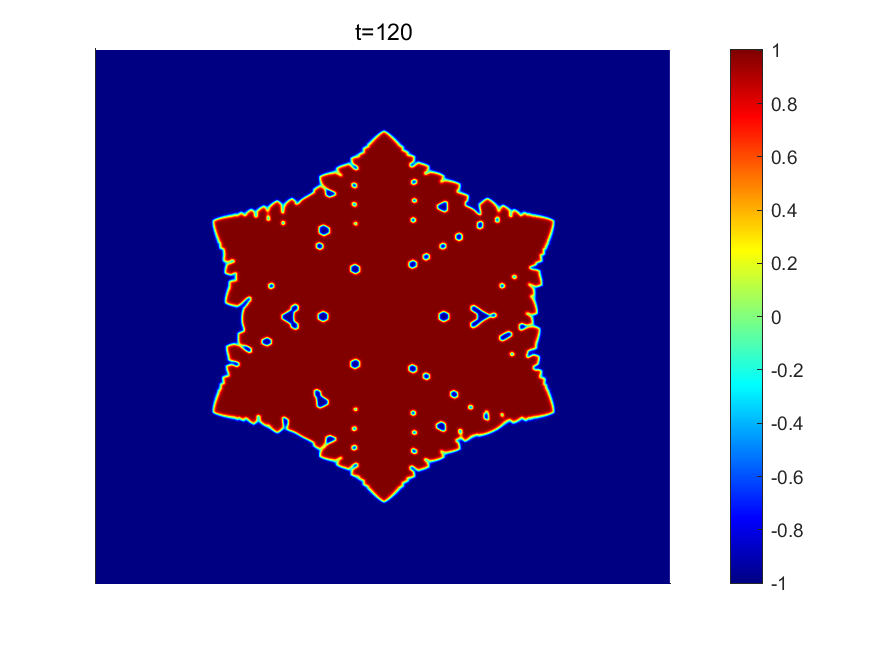}
		\end{minipage}
		\begin{minipage}[t]{0.185\linewidth}
			\centering
			\includegraphics[width=1\linewidth]{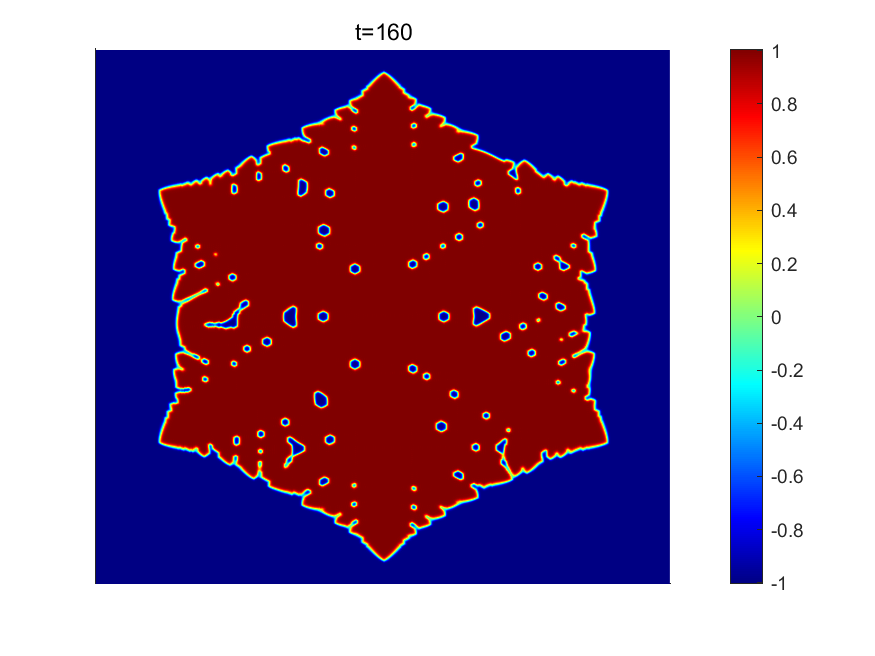}
		\end{minipage}
	}
    \subfigure[\scriptsize (b) Snapshots of the temperature field $u$.]{
		\begin{minipage}[t]{0.185\linewidth}
			\centering
			\includegraphics[width=1\linewidth]{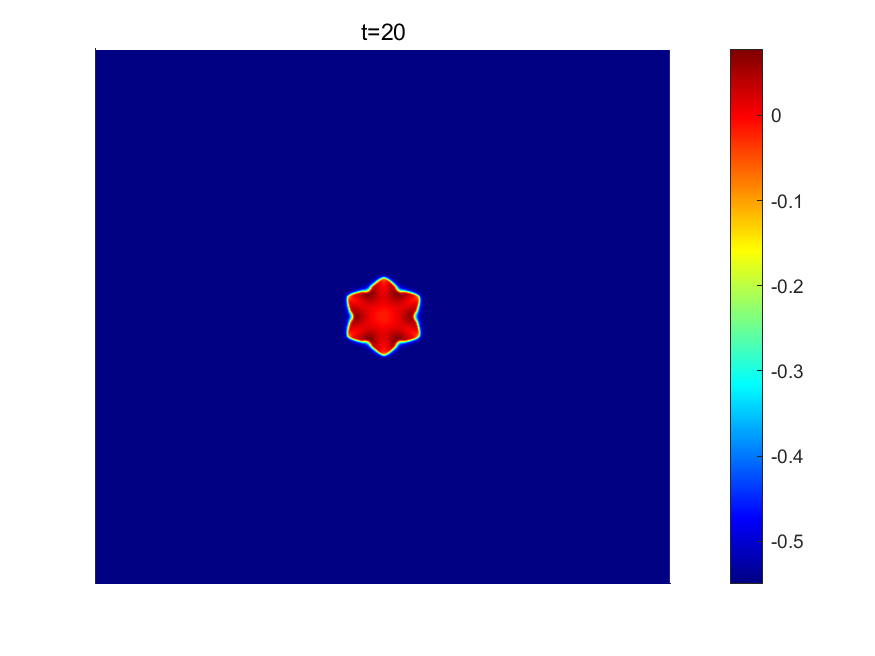}
		\end{minipage}
		\begin{minipage}[t]{0.185\linewidth}
			\centering
			\includegraphics[width=1\linewidth]{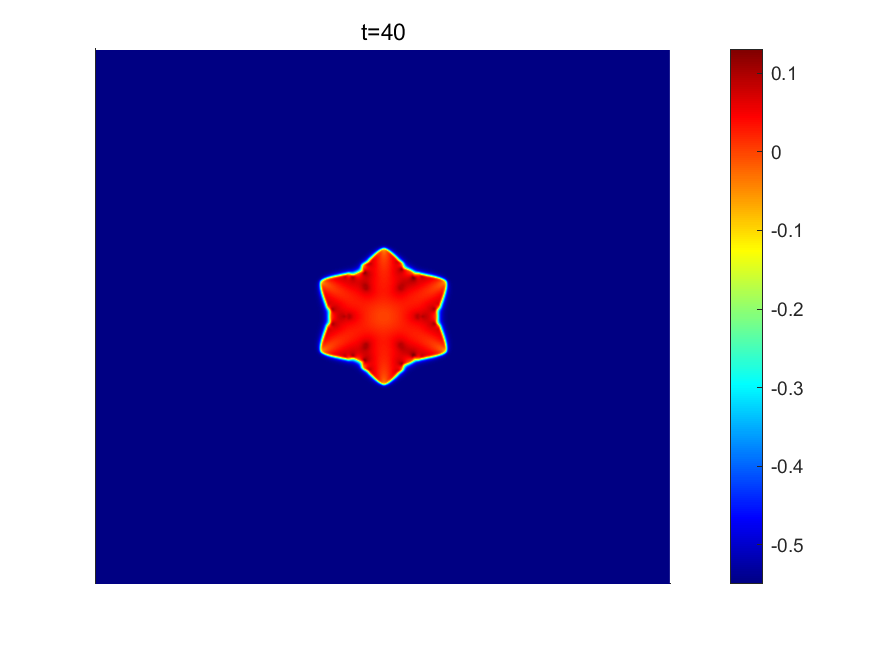}
		\end{minipage}
		\begin{minipage}[t]{0.185\linewidth}
			\centering
			\includegraphics[width=1\linewidth]{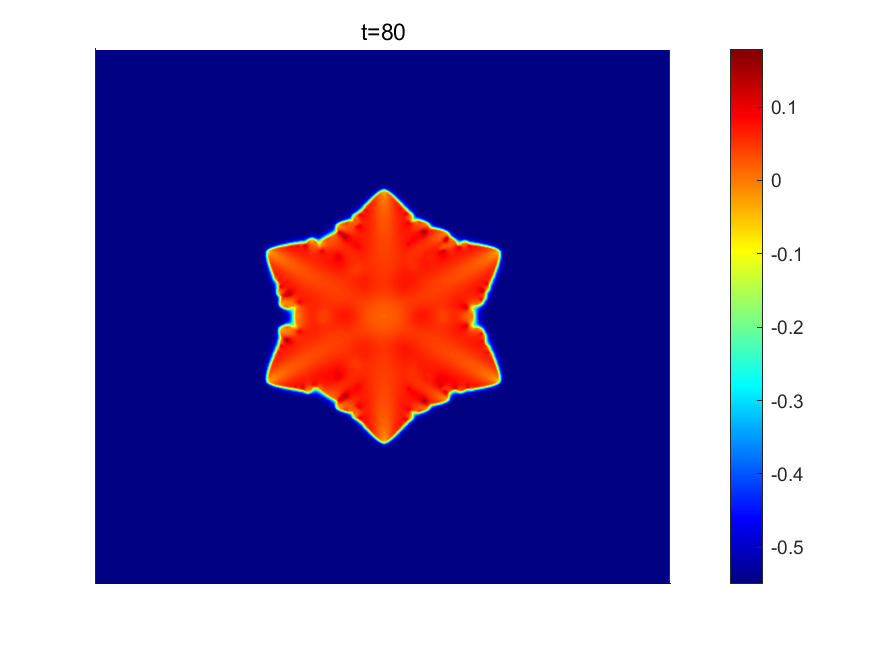}
		\end{minipage}
		\begin{minipage}[t]{0.185\linewidth}
			\centering
			\includegraphics[width=1\linewidth]{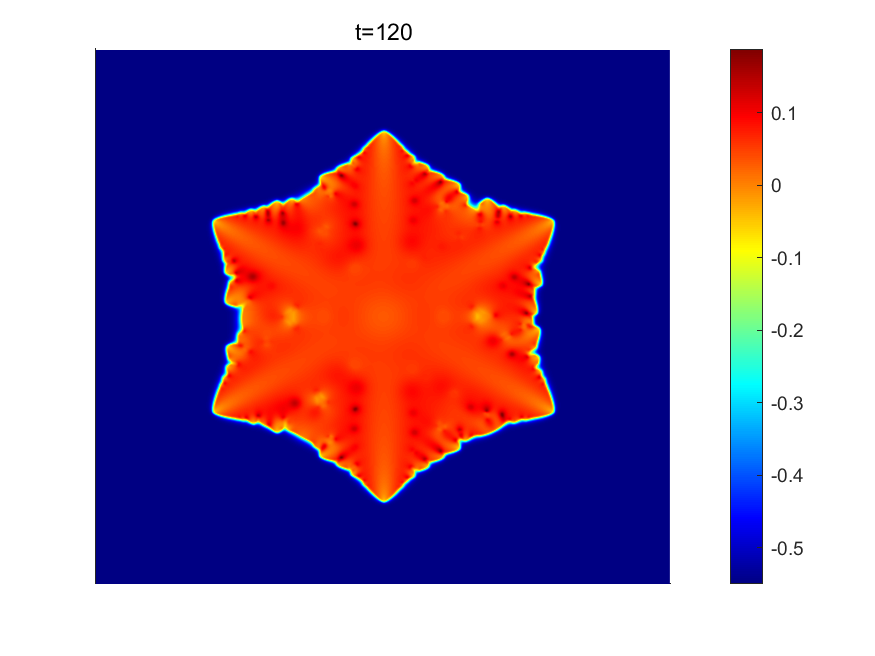}
		\end{minipage}
		\begin{minipage}[t]{0.185\linewidth}
			\centering
			\includegraphics[width=1\linewidth]{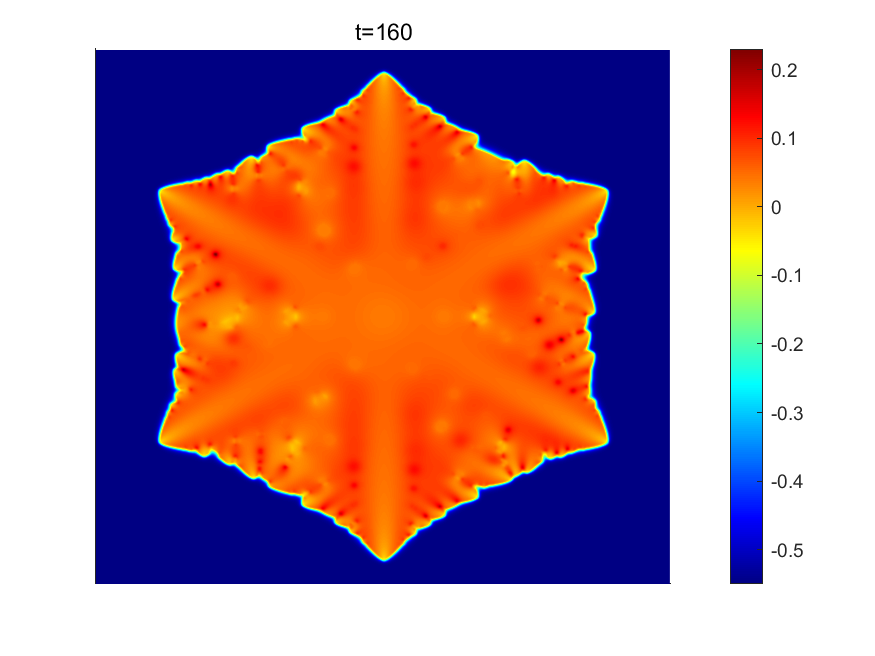}
		\end{minipage}
	}
	\caption{Sixfold dendritic crystal growth for $K=0.6$.}
	\label{fig:sixfoldK06}
\end{figure}

\begin{figure}[htbp]
	\vspace{-0.35cm}
	\subfigtopskip=2pt
	\subfigbottomskip=2pt
	\subfigcapskip=-5pt
	\centering	
    \subfigure[\scriptsize (a) Snapshots of the phase field $\phi$.]{
		\begin{minipage}[t]{0.185\linewidth}
			\centering
			\includegraphics[width=1\linewidth]{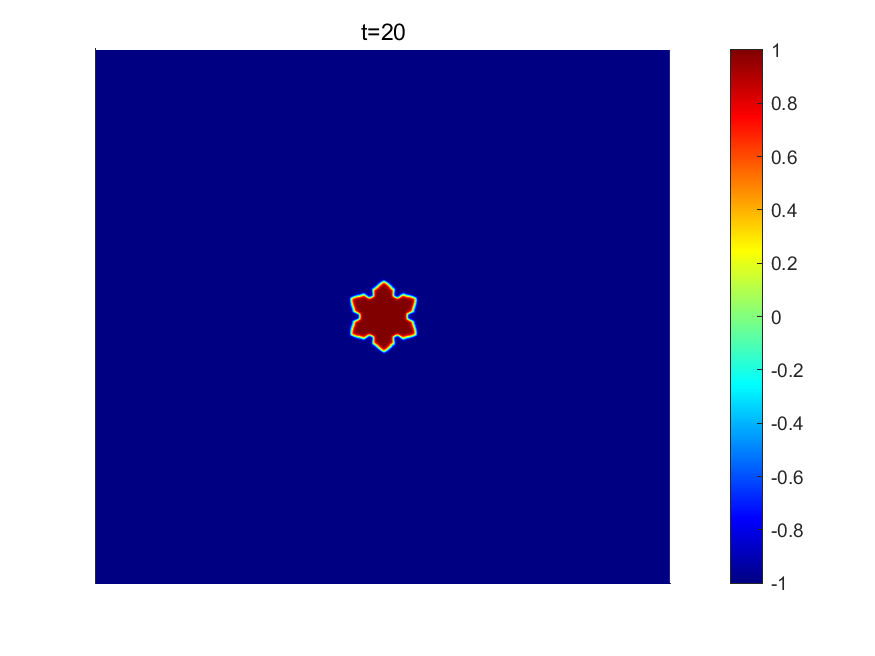}
		\end{minipage}
		\begin{minipage}[t]{0.185\linewidth}
			\centering
			\includegraphics[width=1\linewidth]{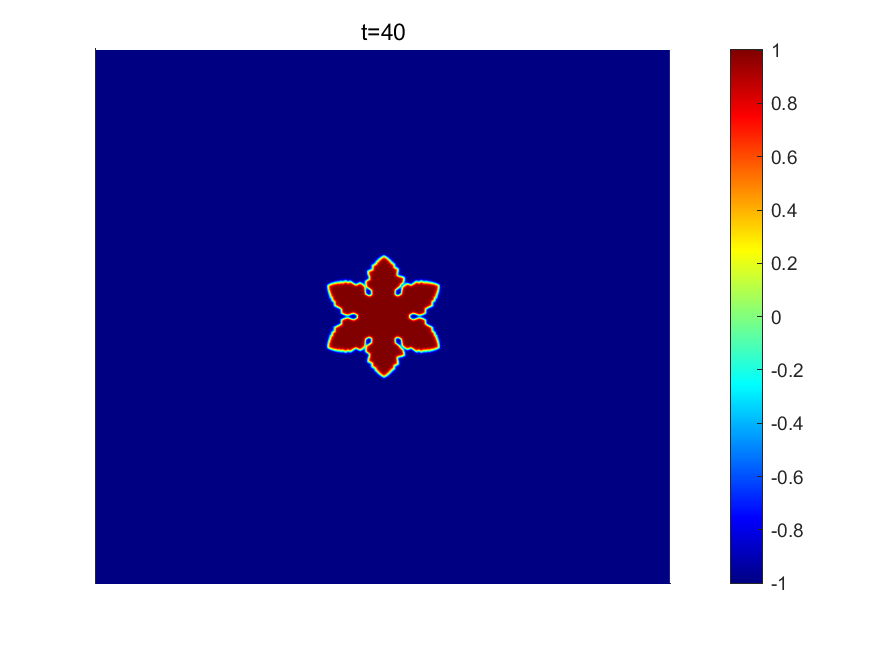}
		\end{minipage}
		\begin{minipage}[t]{0.185\linewidth}
			\centering
			\includegraphics[width=1\linewidth]{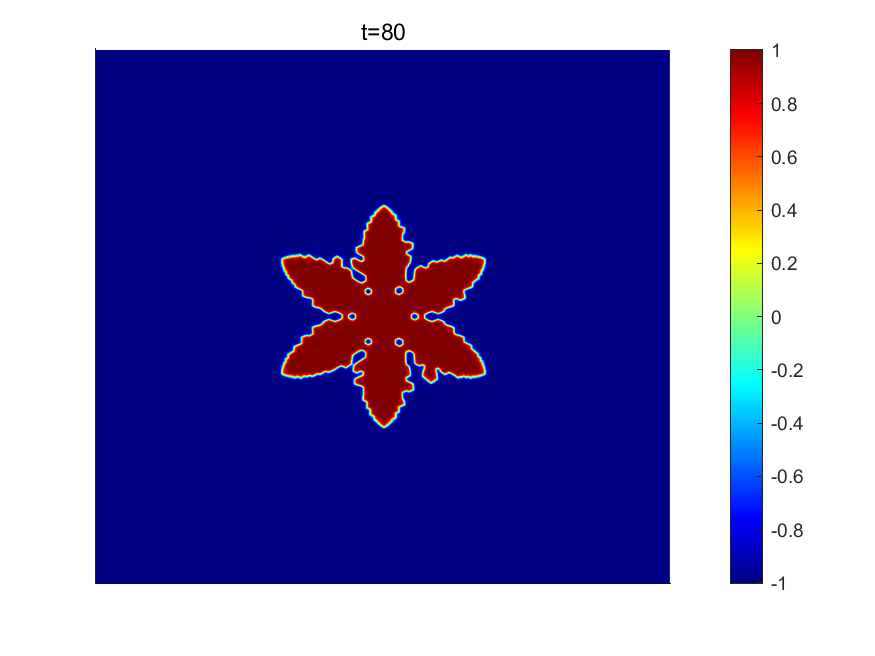}
		\end{minipage}
		\begin{minipage}[t]{0.185\linewidth}
			\centering
			\includegraphics[width=1\linewidth]{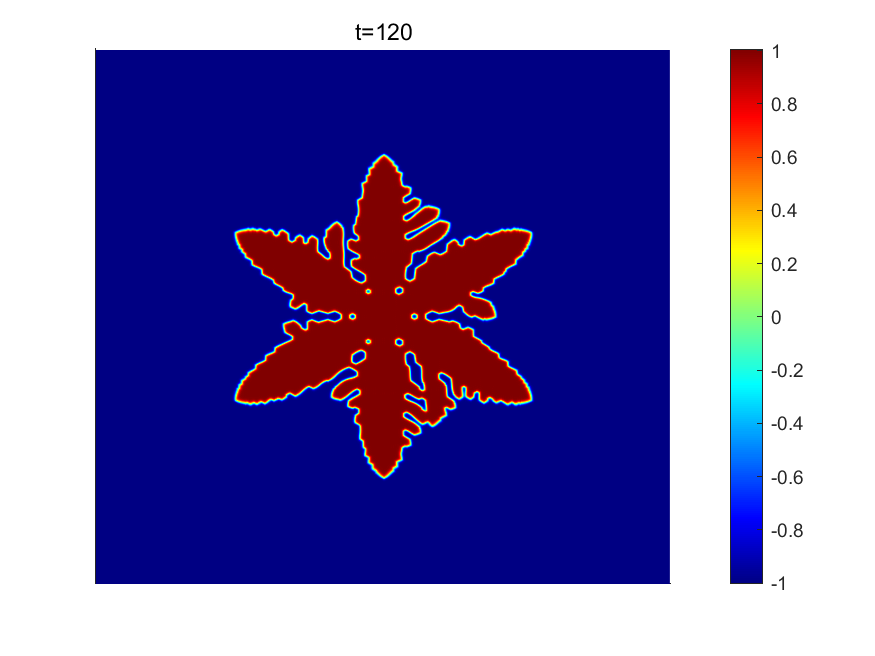}
		\end{minipage}
		\begin{minipage}[t]{0.185\linewidth}
			\centering
			\includegraphics[width=1\linewidth]{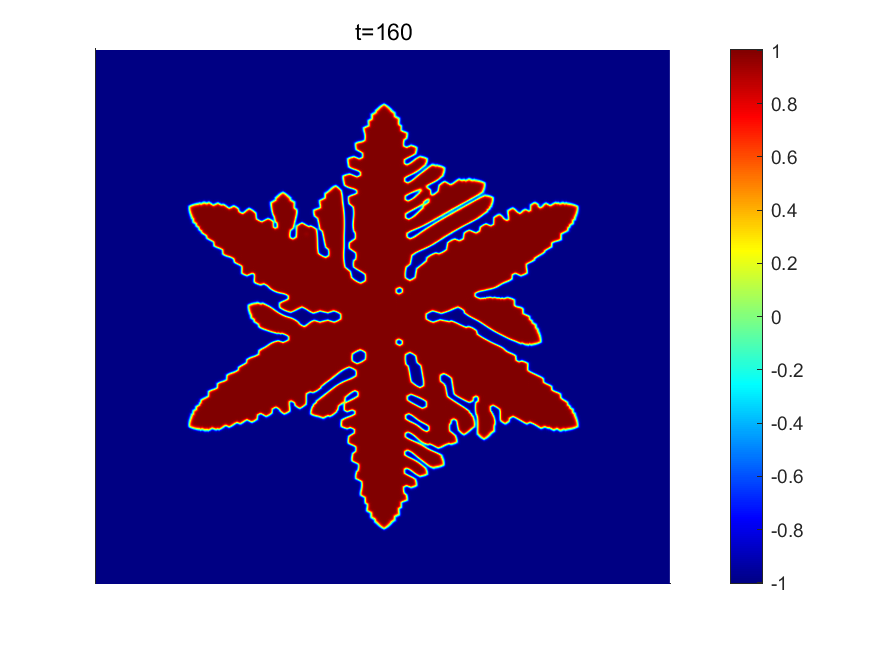}
		\end{minipage}
	}
    \subfigure[\scriptsize (b) Snapshots of the temperature field $u$.]{
		\begin{minipage}[t]{0.185\linewidth}
			\centering
			\includegraphics[width=1\linewidth]{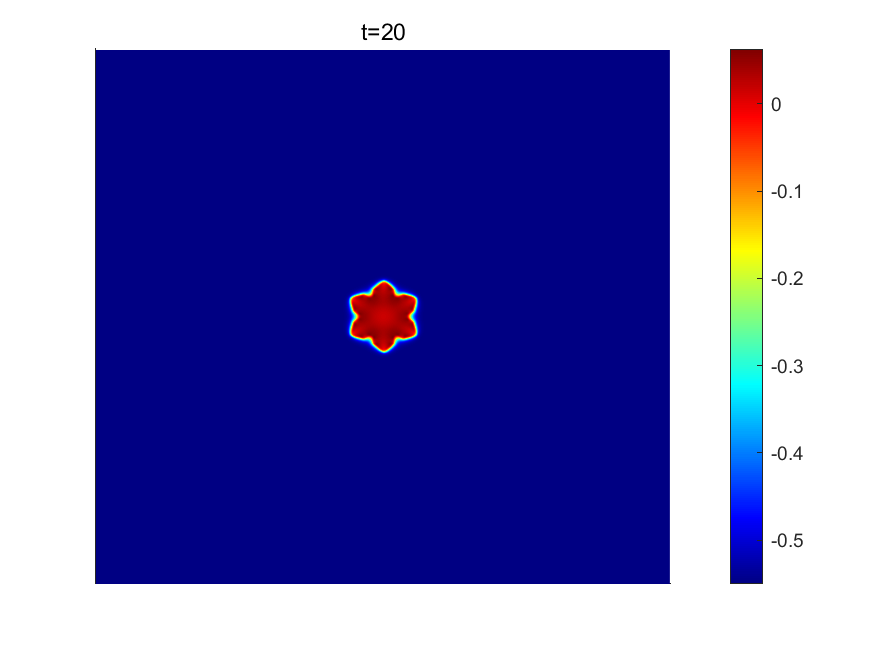}
		\end{minipage}
		\begin{minipage}[t]{0.185\linewidth}
			\centering
			\includegraphics[width=1\linewidth]{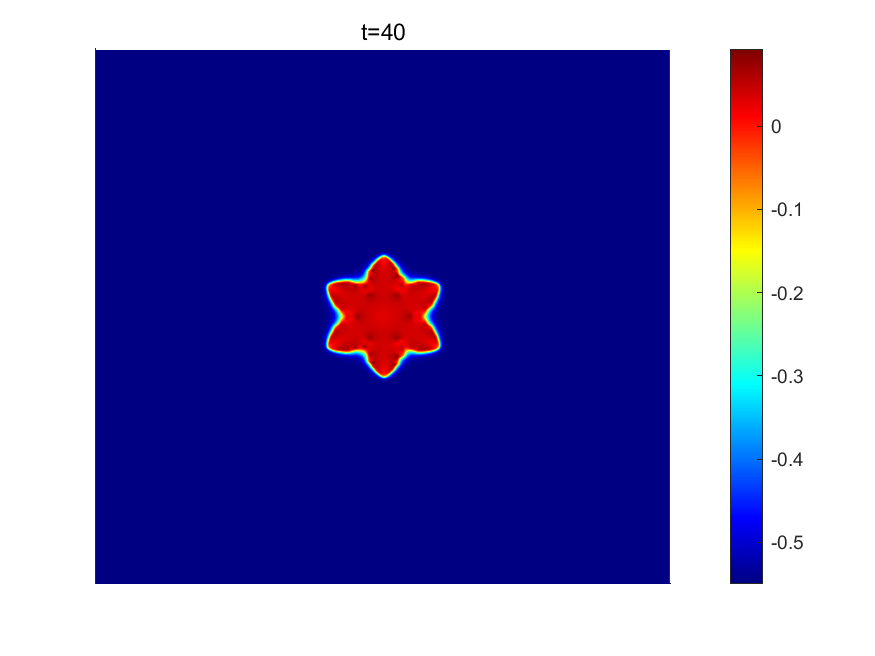}
		\end{minipage}
		\begin{minipage}[t]{0.185\linewidth}
			\centering
			\includegraphics[width=1\linewidth]{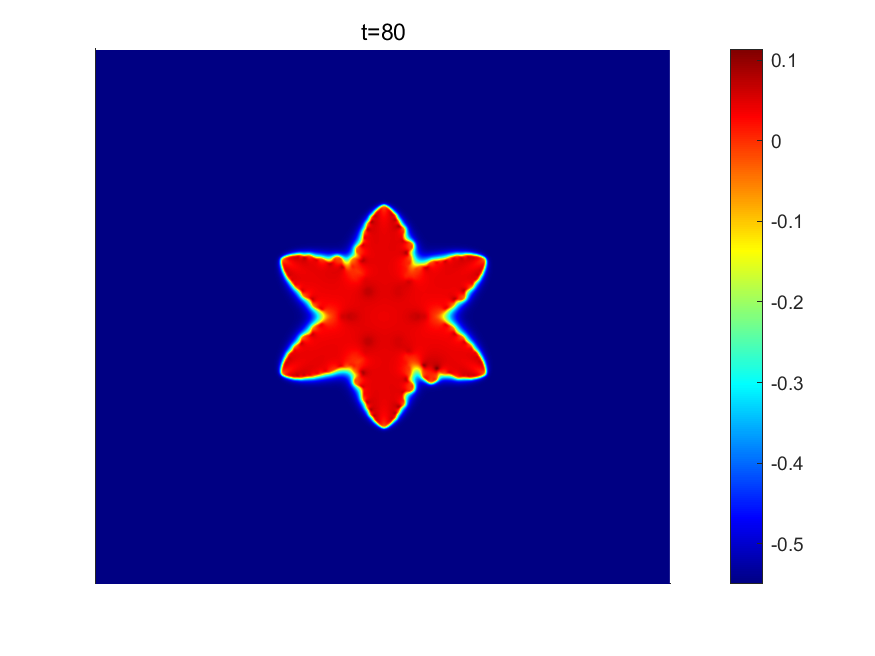}
		\end{minipage}
		\begin{minipage}[t]{0.185\linewidth}
			\centering
			\includegraphics[width=1\linewidth]{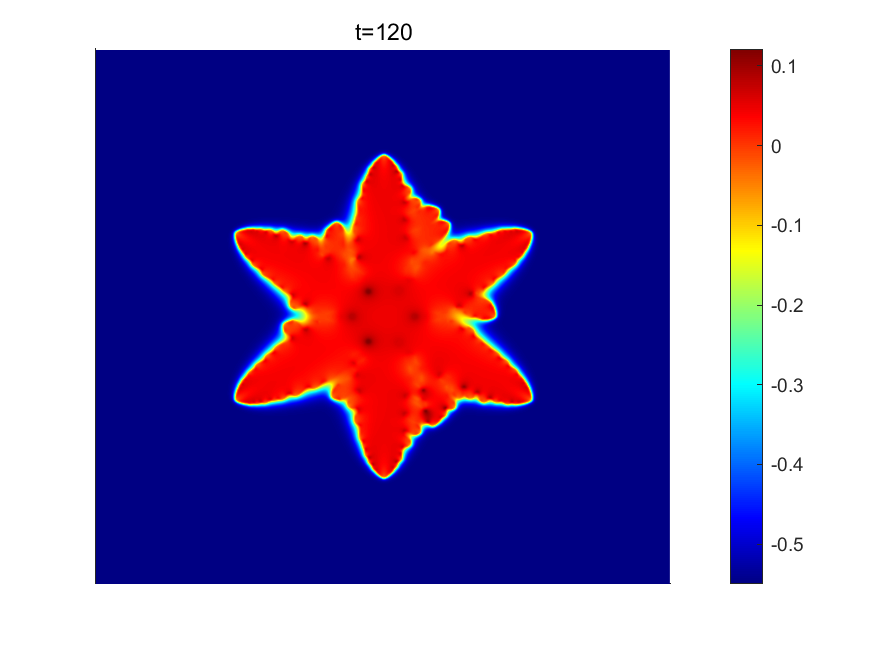}
		\end{minipage}
		\begin{minipage}[t]{0.185\linewidth}
			\centering
			\includegraphics[width=1\linewidth]{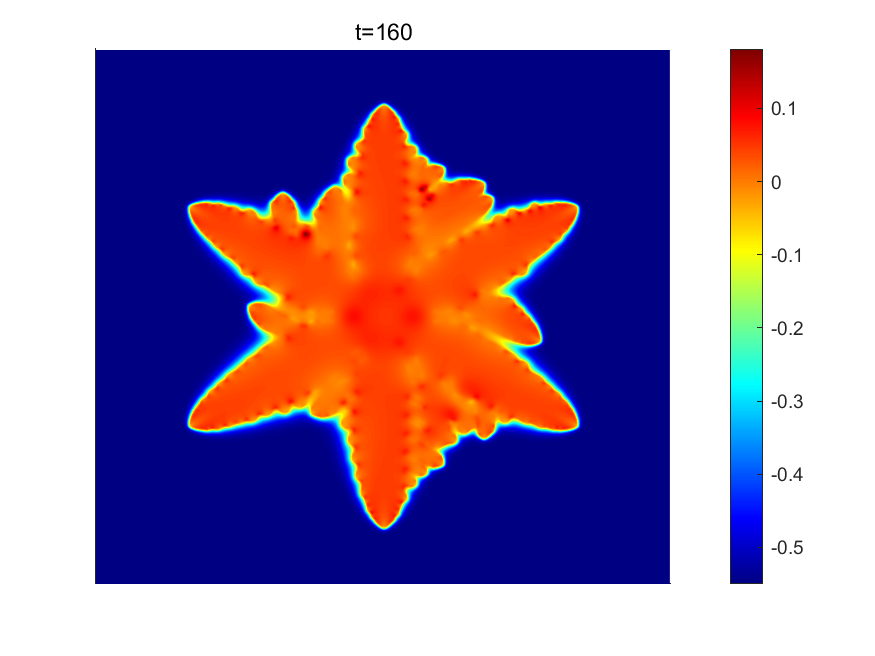}
		\end{minipage}
	}
	\caption{Sixfold dendritic crystal growth for $K=0.65$.}
	\label{fig:sixfoldK065}
\end{figure}

\begin{figure}[htbp]
	\vspace{-0.35cm}
	\subfigtopskip=2pt
	\subfigbottomskip=2pt
	\subfigcapskip=-5pt
	\centering	
    \subfigure[\scriptsize (a) Snapshots of the phase field $\phi$.]{
		\begin{minipage}[t]{0.185\linewidth}
			\centering
			\includegraphics[width=1\linewidth]{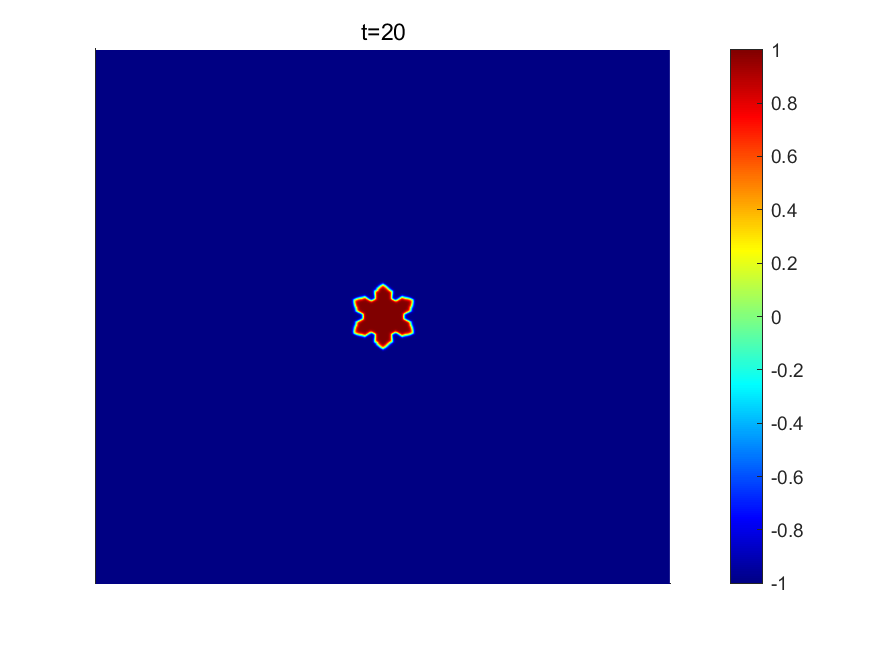}
		\end{minipage}
		\begin{minipage}[t]{0.185\linewidth}
			\centering
			\includegraphics[width=1\linewidth]{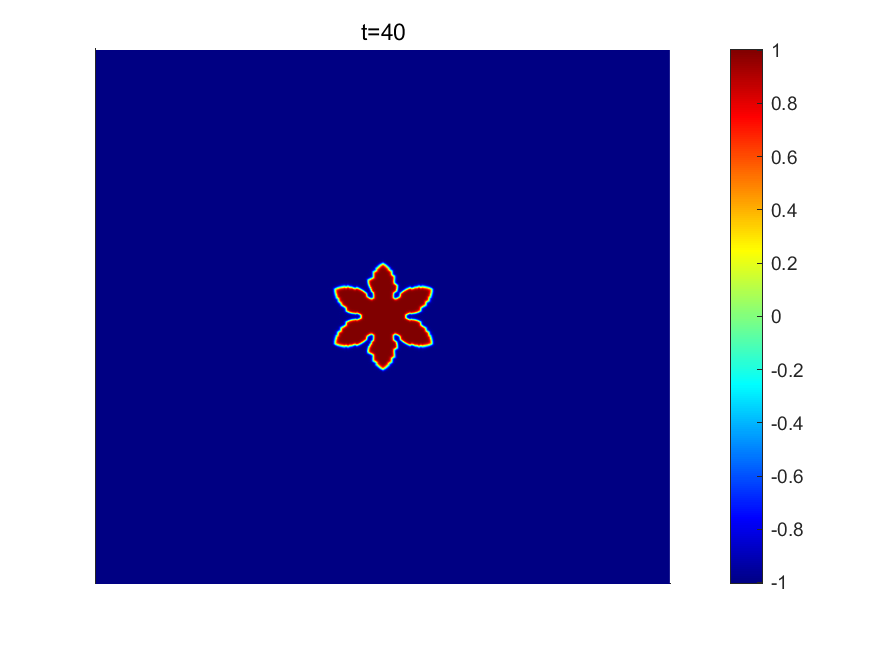}
		\end{minipage}
		\begin{minipage}[t]{0.185\linewidth}
			\centering
			\includegraphics[width=1\linewidth]{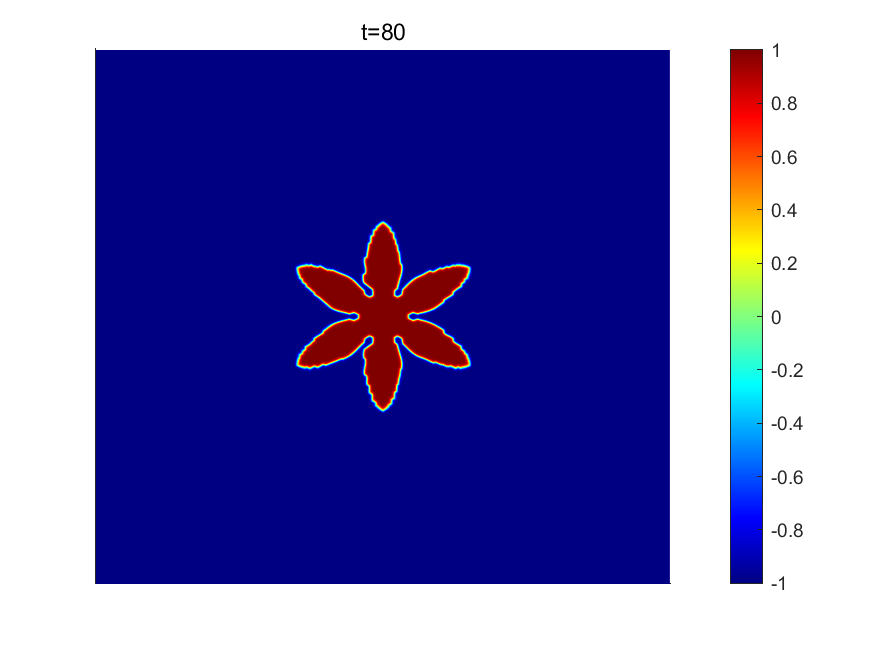}
		\end{minipage}
		\begin{minipage}[t]{0.185\linewidth}
			\centering
			\includegraphics[width=1\linewidth]{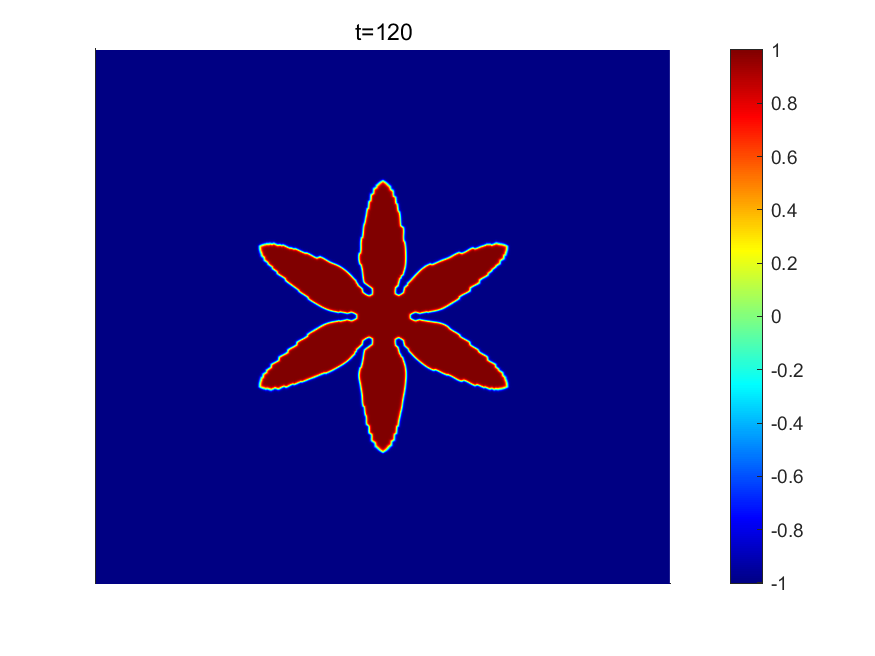}
		\end{minipage}
		\begin{minipage}[t]{0.185\linewidth}
			\centering
			\includegraphics[width=1\linewidth]{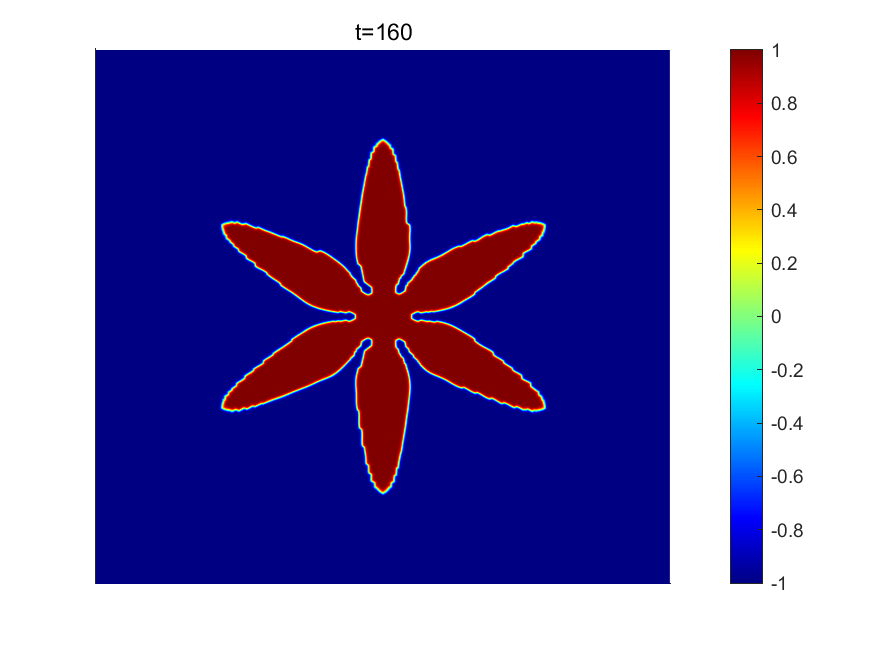}
		\end{minipage}
	}
    \subfigure[\scriptsize (b) Snapshots of the temperature field $u$.]{
		\begin{minipage}[t]{0.185\linewidth}
			\centering
			\includegraphics[width=1\linewidth]{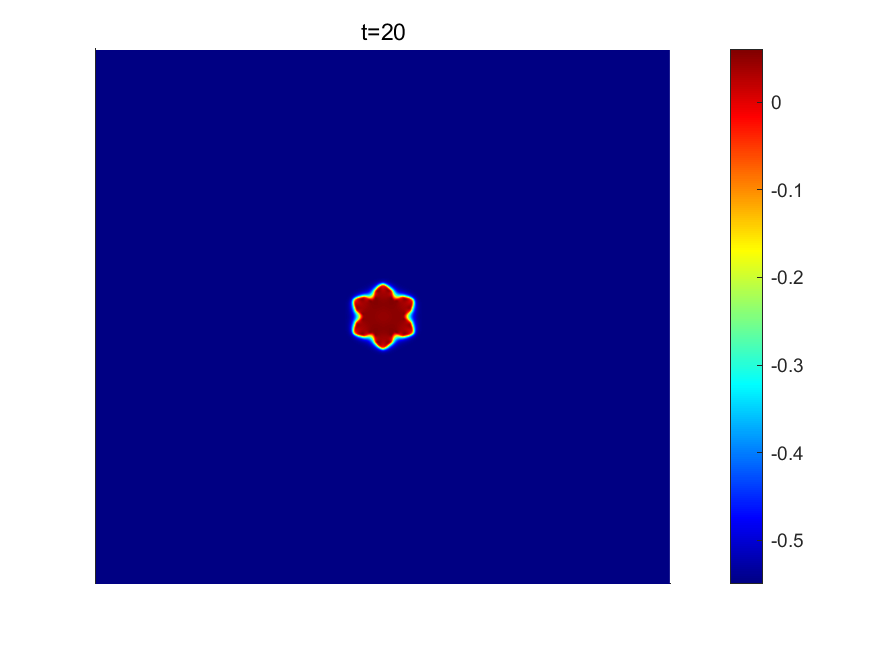}
		\end{minipage}
		\begin{minipage}[t]{0.185\linewidth}
			\centering
			\includegraphics[width=1\linewidth]{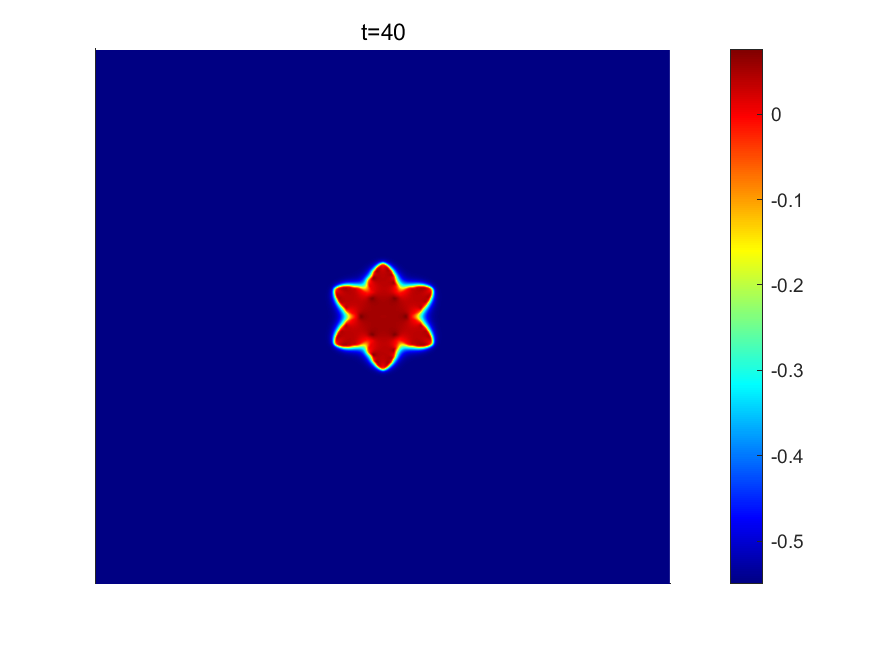}
		\end{minipage}
		\begin{minipage}[t]{0.185\linewidth}
			\centering
			\includegraphics[width=1\linewidth]{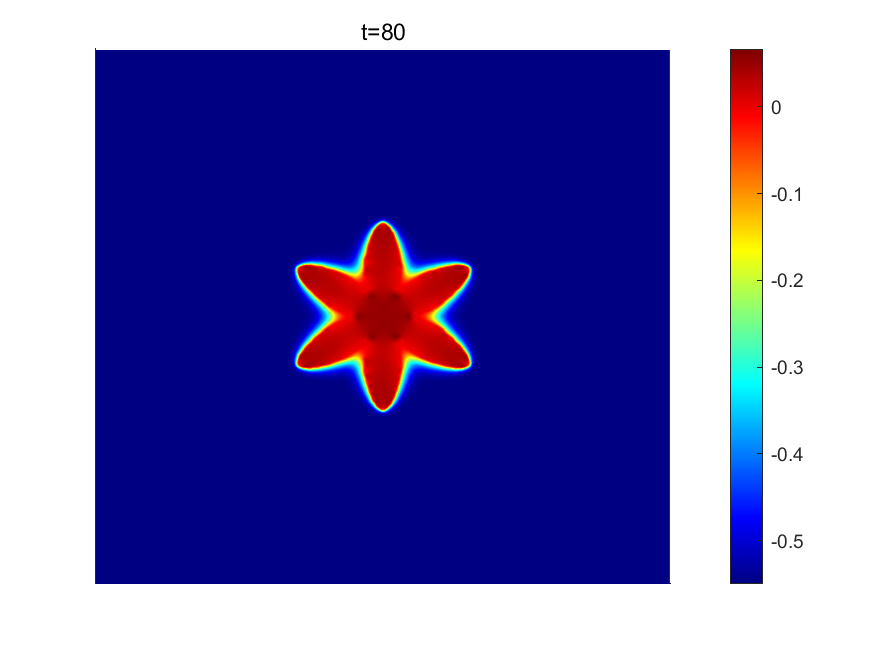}
		\end{minipage}
		\begin{minipage}[t]{0.185\linewidth}
			\centering
			\includegraphics[width=1\linewidth]{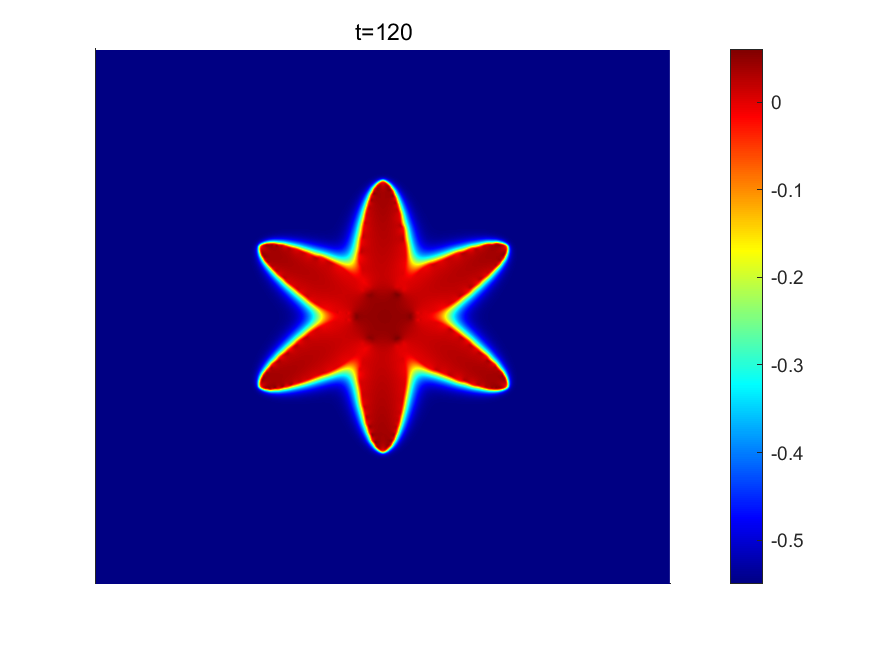}
		\end{minipage}
		\begin{minipage}[t]{0.185\linewidth}
			\centering
			\includegraphics[width=1\linewidth]{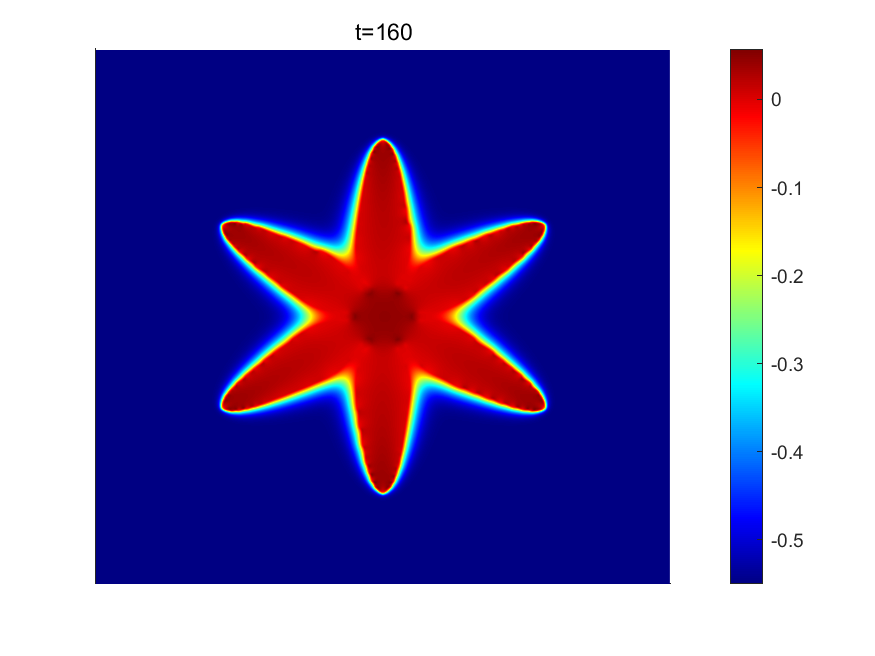}
		\end{minipage}
	}
	\caption{Sixfold dendritic crystal growth for $K=0.7$.}
	\label{fig:sixfoldK07}
\end{figure}

\begin{figure}[h]
\subfigure[\scriptsize (a) Energy evolution]{
\begin{minipage}{0.45 \textwidth}
\centering
\includegraphics[width=\textwidth]{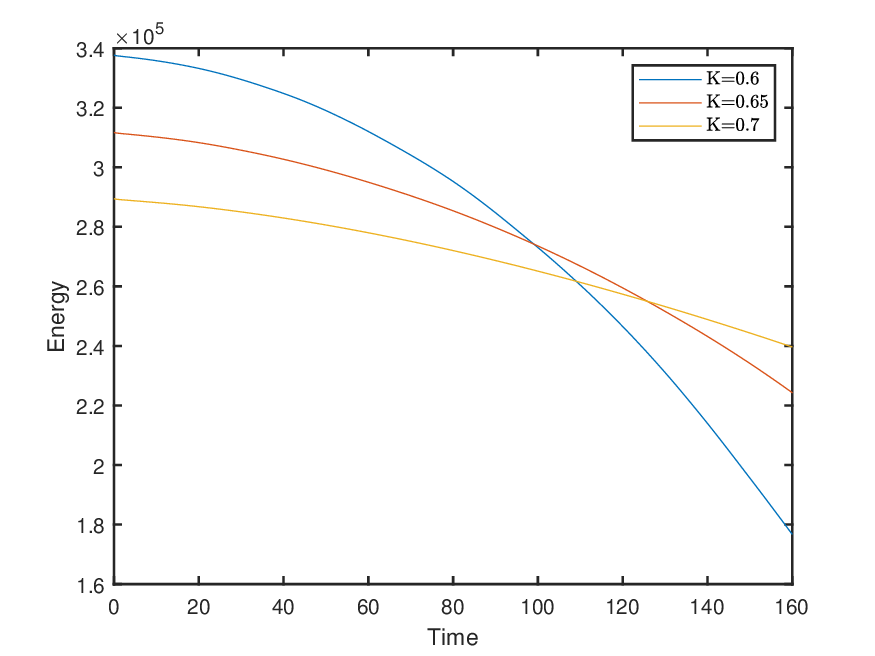}
\end{minipage}
}
\subfigure[\scriptsize (b) Crystal area]{
\begin{minipage}{0.45 \textwidth}
\centering
\includegraphics[width=\textwidth]{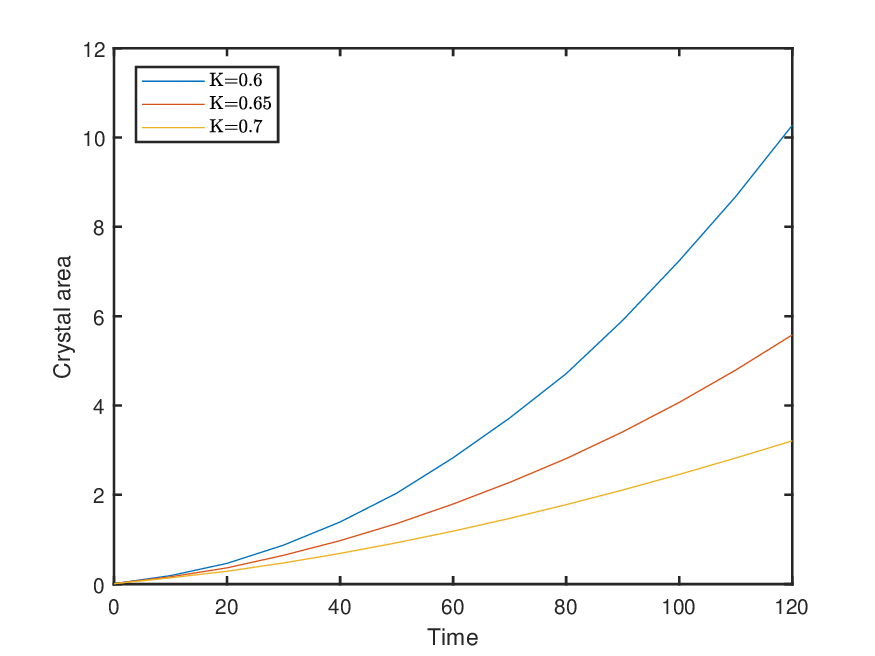}
\end{minipage}
}

\caption{ (a) Time evolution of the energy; (b) crystal
area $\int_{\Omega}\frac{1+\phi}{2}dx$.
}
\label{sixfold_energy}
\end{figure}

\begin{figure}[htbp]
	\vspace{-0.35cm}
	\subfigtopskip=2pt
	\subfigbottomskip=2pt
	\subfigcapskip=-5pt
	\centering	
    \subfigure[\scriptsize (a) Snapshots of the phase field $\phi$.]{
		\begin{minipage}[t]{0.185\linewidth}
			\centering
			\includegraphics[width=1\linewidth]{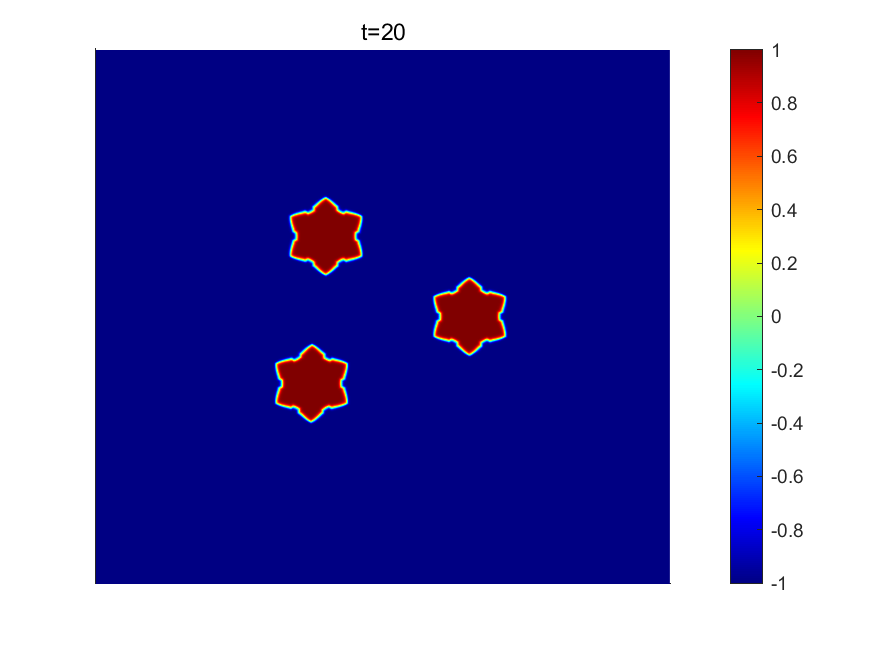}
		\end{minipage}
		\begin{minipage}[t]{0.185\linewidth}
			\centering
			\includegraphics[width=1\linewidth]{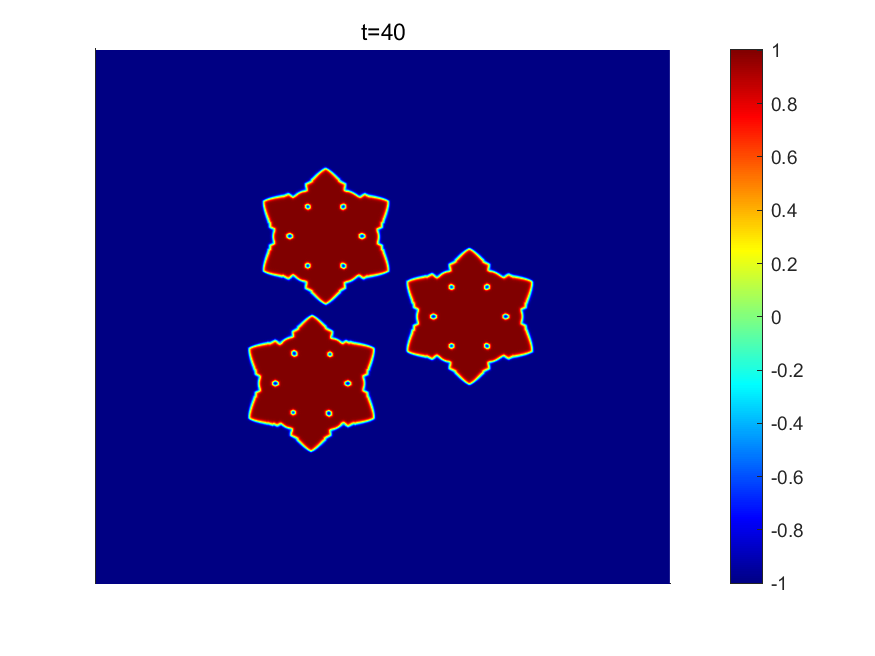}
		\end{minipage}
		\begin{minipage}[t]{0.185\linewidth}
			\centering
			\includegraphics[width=1\linewidth]{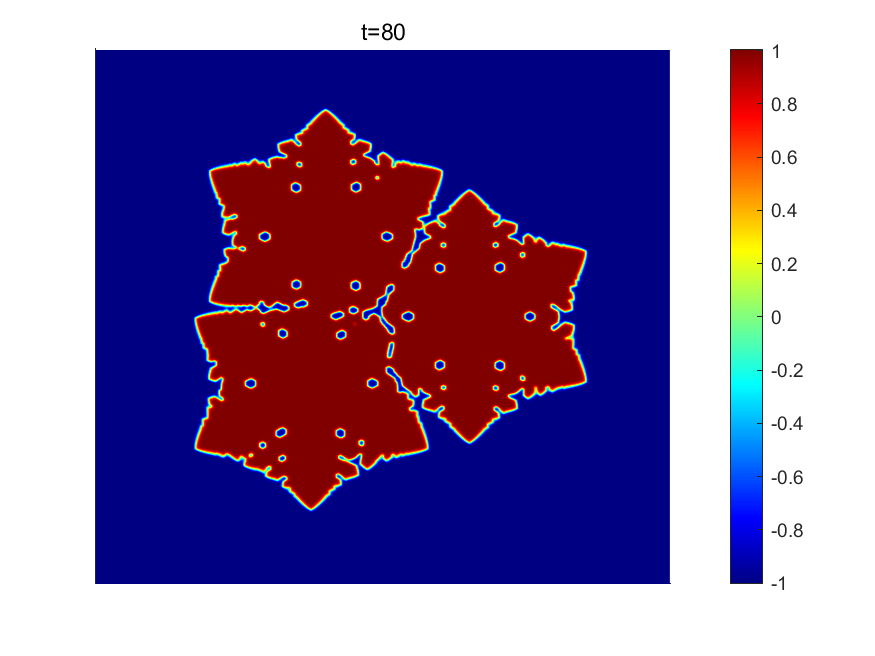}
		\end{minipage}
		\begin{minipage}[t]{0.185\linewidth}
			\centering
			\includegraphics[width=1\linewidth]{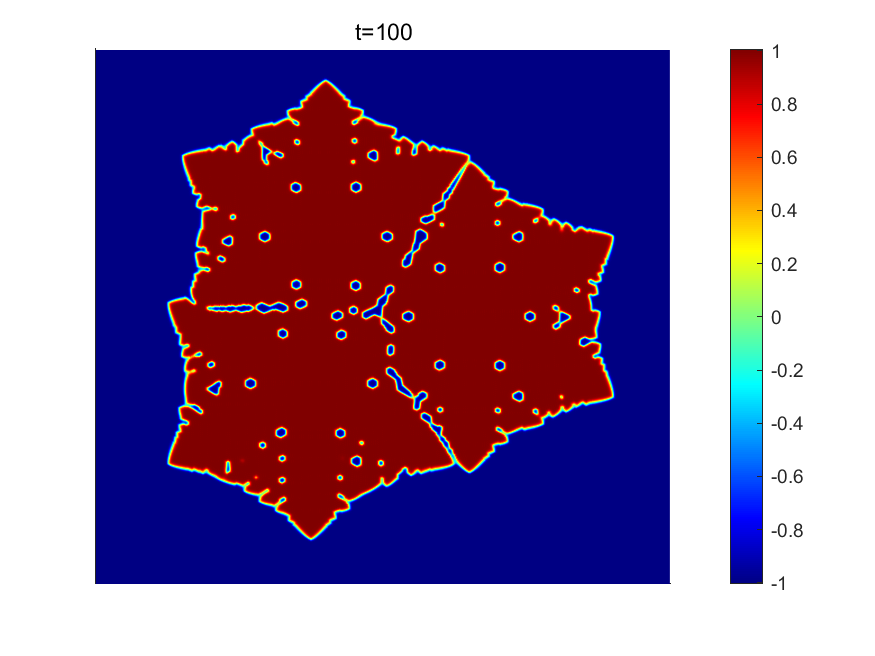}
		\end{minipage}
		\begin{minipage}[t]{0.185\linewidth}
			\centering
			\includegraphics[width=1\linewidth]{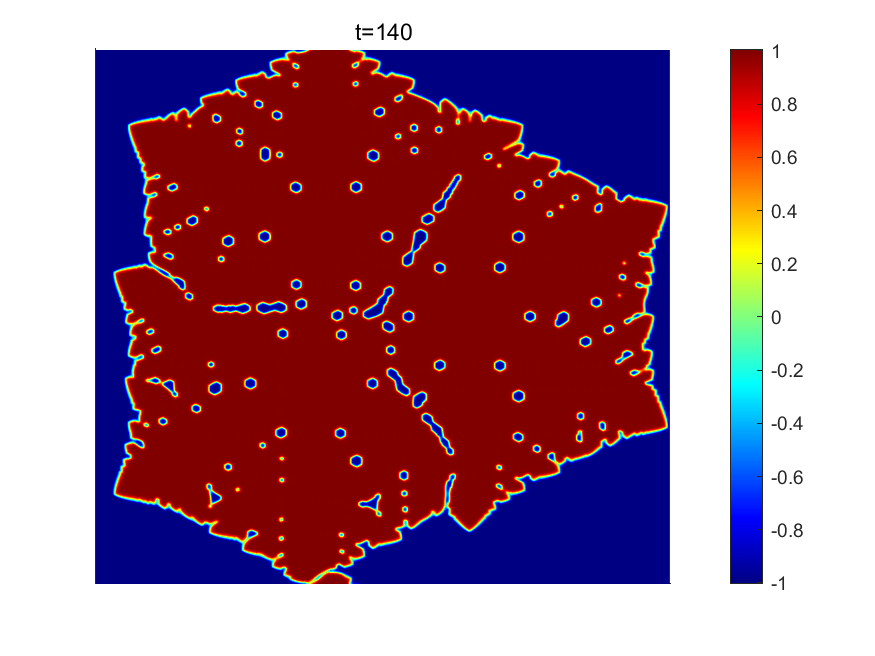}
		\end{minipage}
	}
    \subfigure[\scriptsize (b) Snapshots of the temperature field $u$.]{
		\begin{minipage}[t]{0.185\linewidth}
			\centering
			\includegraphics[width=1\linewidth]{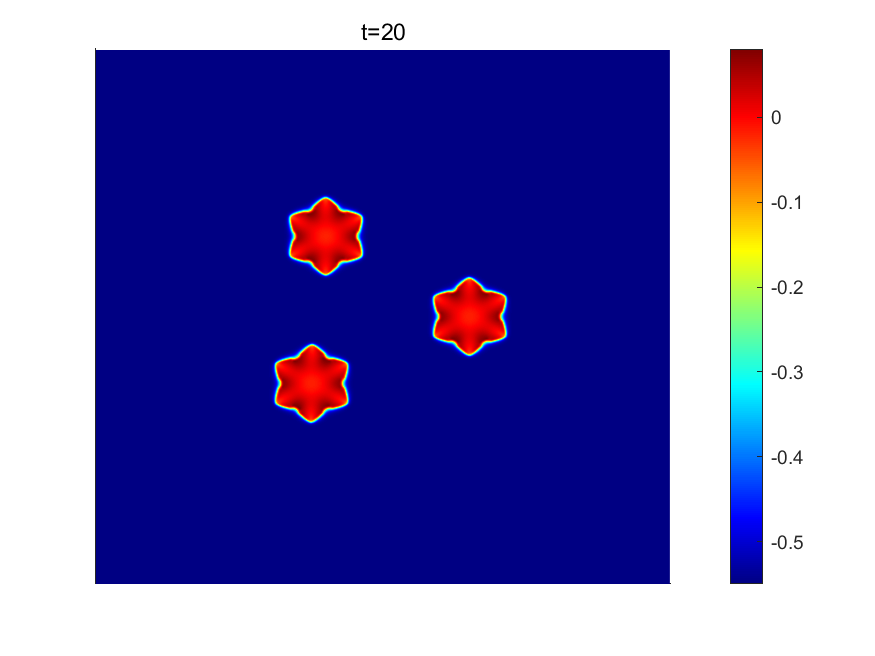}
		\end{minipage}
		\begin{minipage}[t]{0.185\linewidth}
			\centering
			\includegraphics[width=1\linewidth]{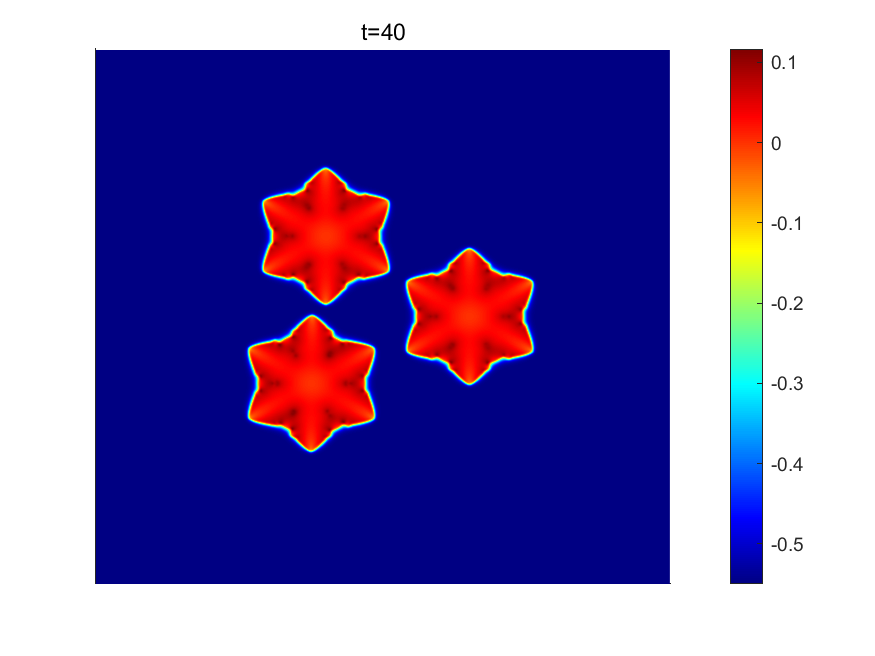}
		\end{minipage}
		\begin{minipage}[t]{0.185\linewidth}
			\centering
			\includegraphics[width=1\linewidth]{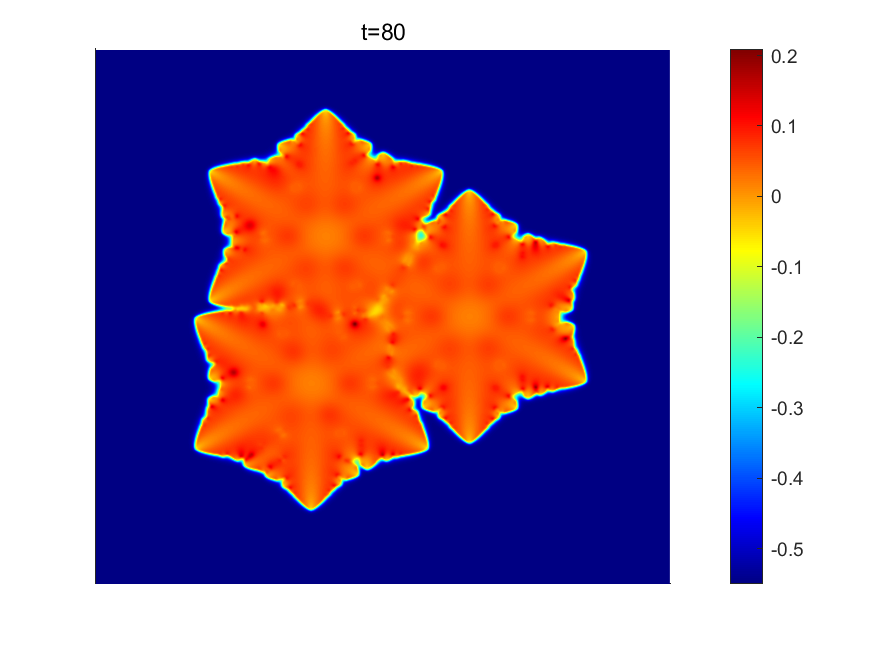}
		\end{minipage}
		\begin{minipage}[t]{0.185\linewidth}
			\centering
			\includegraphics[width=1\linewidth]{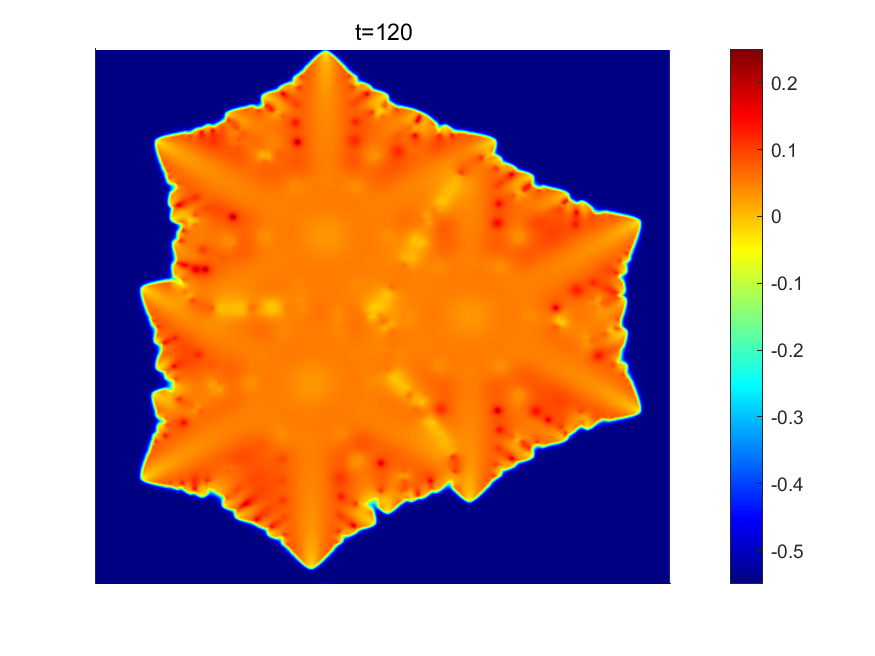}
		\end{minipage}
		\begin{minipage}[t]{0.185\linewidth}
			\centering
			\includegraphics[width=1\linewidth]{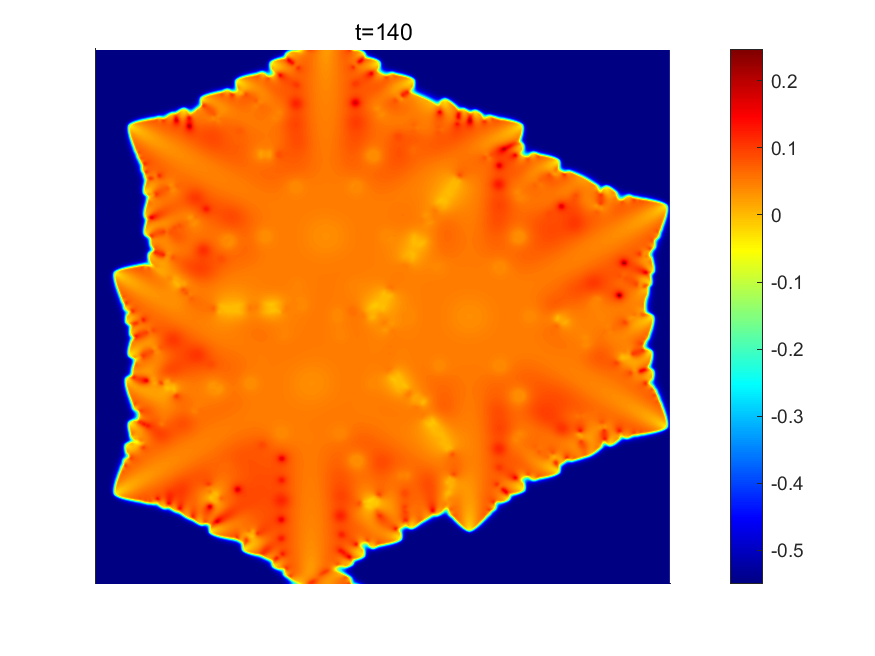}
		\end{minipage}
	}
	\caption{Sixfold anisotropy dendritic crystal growth staring with three deposited nuclei for $K=0.6$.}
	\label{fig:three-sixfoldK06}
\end{figure}

\subsection{3D dendrite crystal growth}

Finally we simulate fourfold dendrite crystal growth in 3D and investigate
the effect of the latent heat parameter $K$ on the crystal shape.
The time step size $\delta t=0.1$ is used in the simulation.
The Fourier spectral method for the spatial discretization uses 128 modes.

 \begin{example}\label{exp.5}
\rm{(Fourfold anisotropy crystal growth in 3D) Consider the following initial conditions}
\begin{equation*}
\left\{
             \begin{array}{ll}
\phi(x,y,z,0)=\tanh\Big(\frac{0.2-\sqrt{(x-\pi)^2+(y-\pi)^2+(z-\pi)^2}}{0.072}\Big),\\
u(x,y,z,0)=\left\{
             \begin{array}{ll}
             0,  &\phi>0,\\
             -0.55, &otherwise.
                          \end{array}
           \right.
             \end{array}
           \right.
\end{equation*}
Set
$$
\begin{aligned}
\tau=2.5e4, \varepsilon=3e-2, \lambda=260, D=2e-4, \sigma=0.05,\beta=4.
\end{aligned}
$$
\end{example}
In Figure \ref{fourfoldK13D} we present the isosurfaces of
$\{\phi = 0\}$ at some different time instants for $K=1$.
A diamond-like crystal structure is formed and grows from a tiny nuclei. A closer look
at the figures show that the crystal structure has four branches at each of three directions. The case $K=1.5$ is given in Figure \ref{fourfoldK153D}, showing
sharper tips and the thinner branches, comparing with the case $K=1$.
 Notice that the same 3D simulations have been performed in \cite{karma1998quantitative,li2012phase,yang2019efficient,yang2021novel,wang2022accurate}, and similar growth behavior has been observed therein.

\begin{figure}[h]
\begin{minipage}{0.3 \textwidth}
\centering
\includegraphics[width=\textwidth]{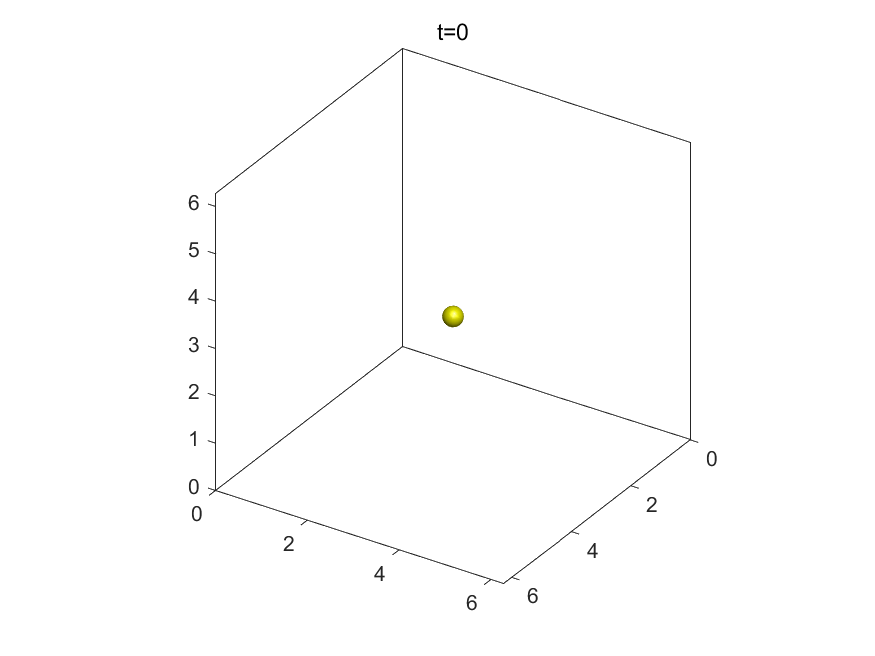}
\end{minipage}
\begin{minipage}{0.3 \textwidth}
\centering
\includegraphics[width=\textwidth]{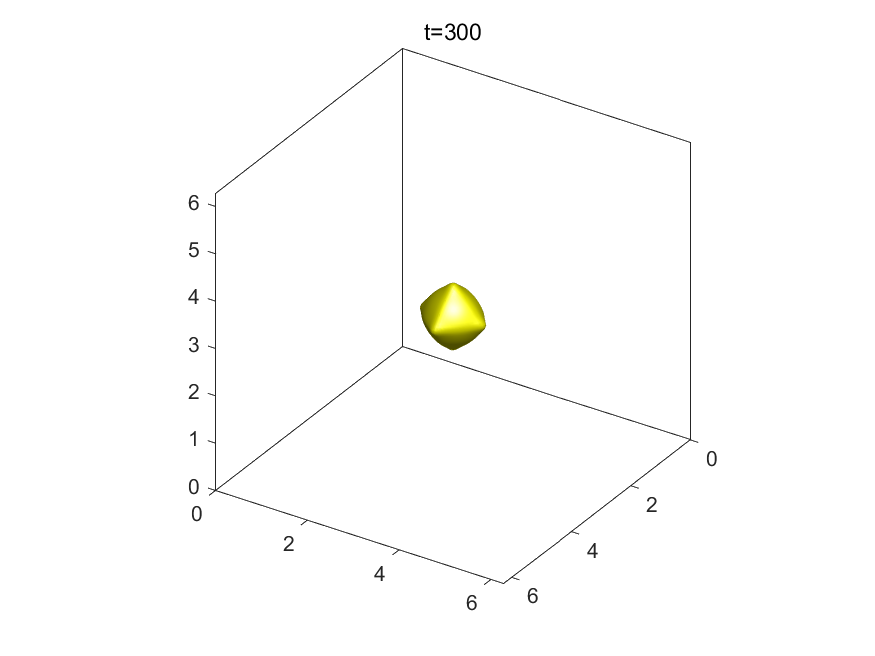}
\end{minipage}
\begin{minipage}{0.3 \textwidth}
\centering
\includegraphics[width=\textwidth]{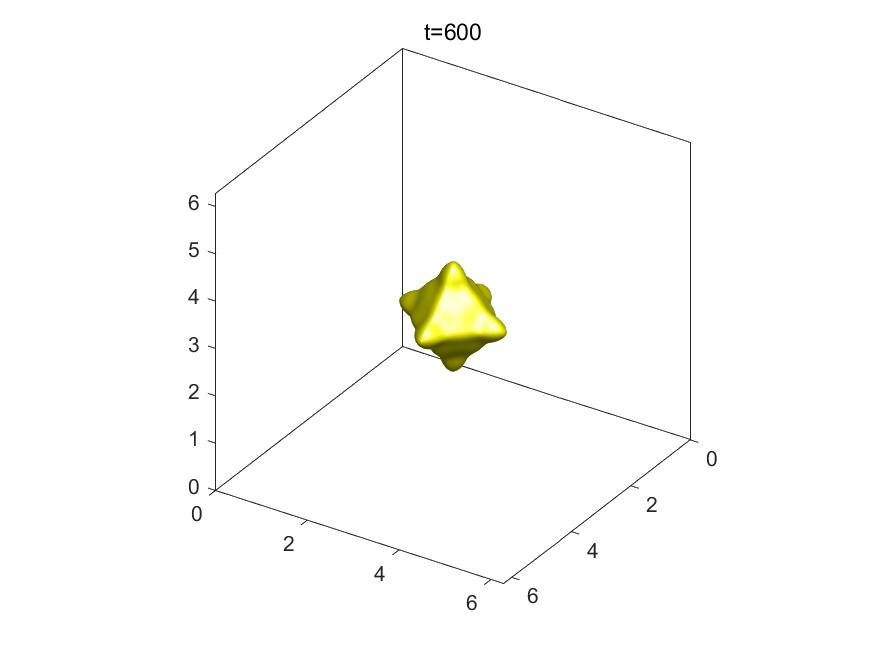}
\end{minipage}

\begin{minipage}{0.3 \textwidth}
\centering
\includegraphics[width=\textwidth]{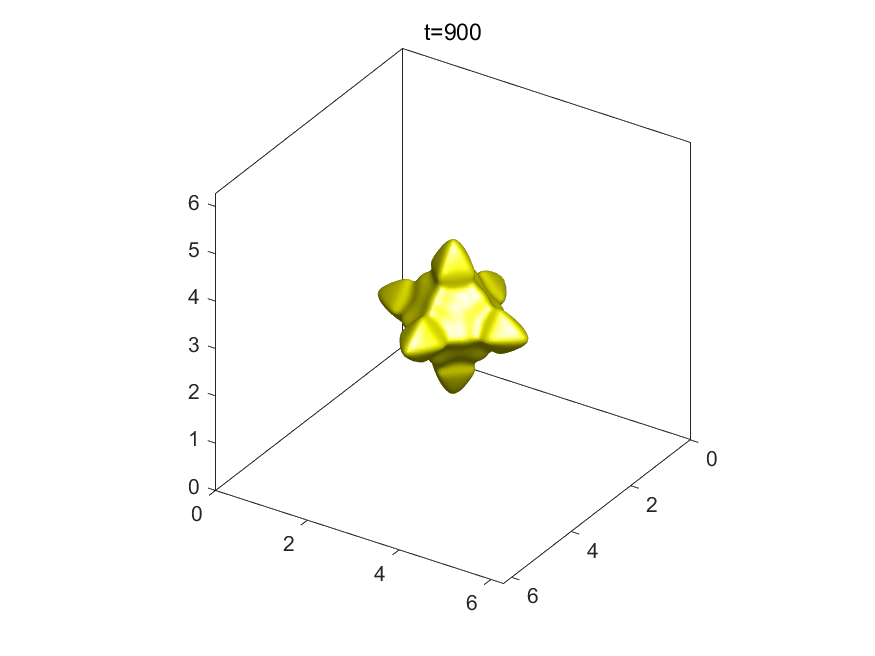}
\end{minipage}
\begin{minipage}{0.3 \textwidth}
\centering
\includegraphics[width=\textwidth]{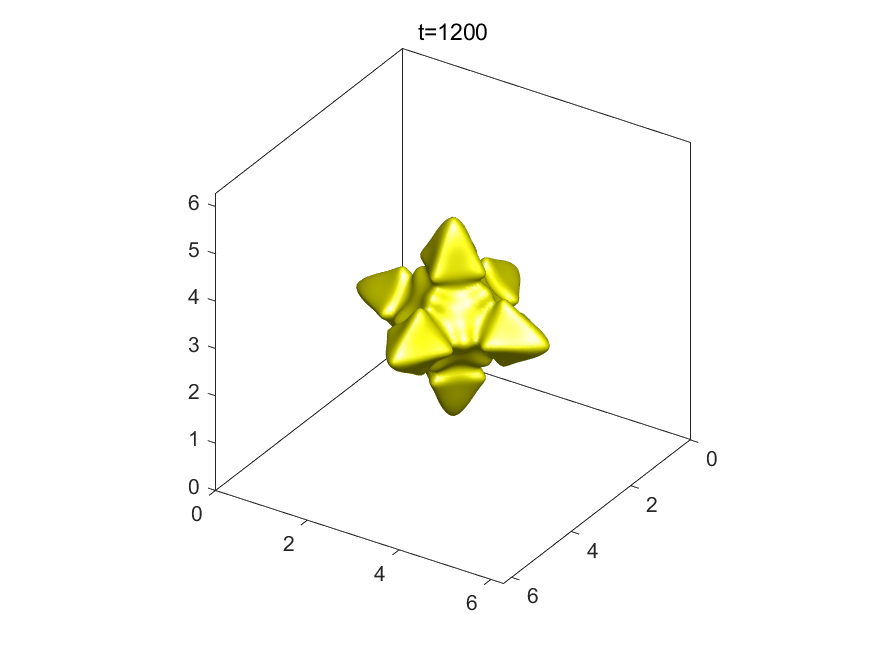}
\end{minipage}
\begin{minipage}{0.3 \textwidth}
\centering
\includegraphics[width=\textwidth]{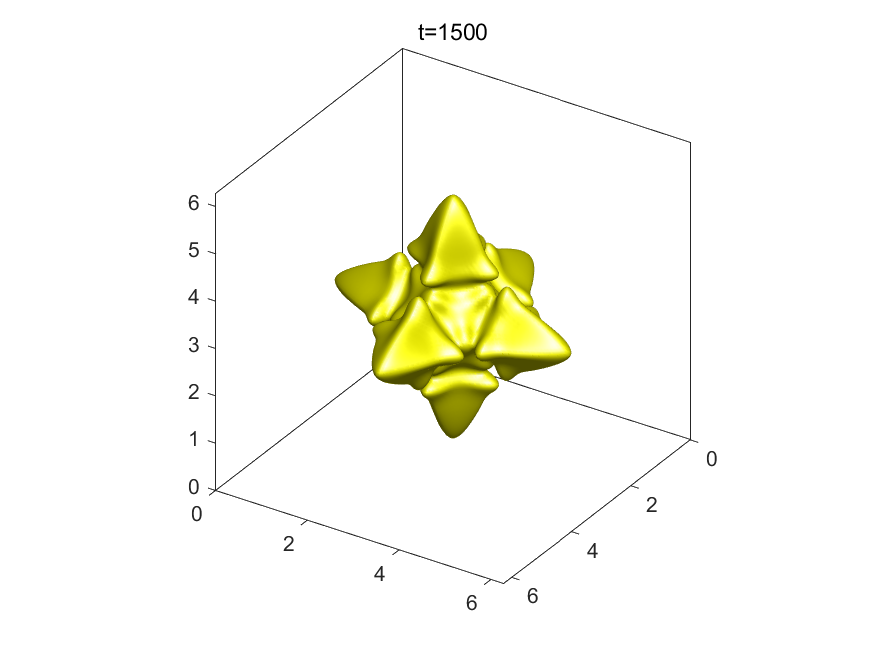}
\end{minipage}
	\caption{(Example \ref{exp.5}) Fourfold anisotropy dendritic crystal growth for $K=1$.}
\label{fourfoldK13D}
\end{figure}

\begin{figure}[h]
\begin{minipage}{0.3 \textwidth}
\centering
\includegraphics[width=\textwidth]{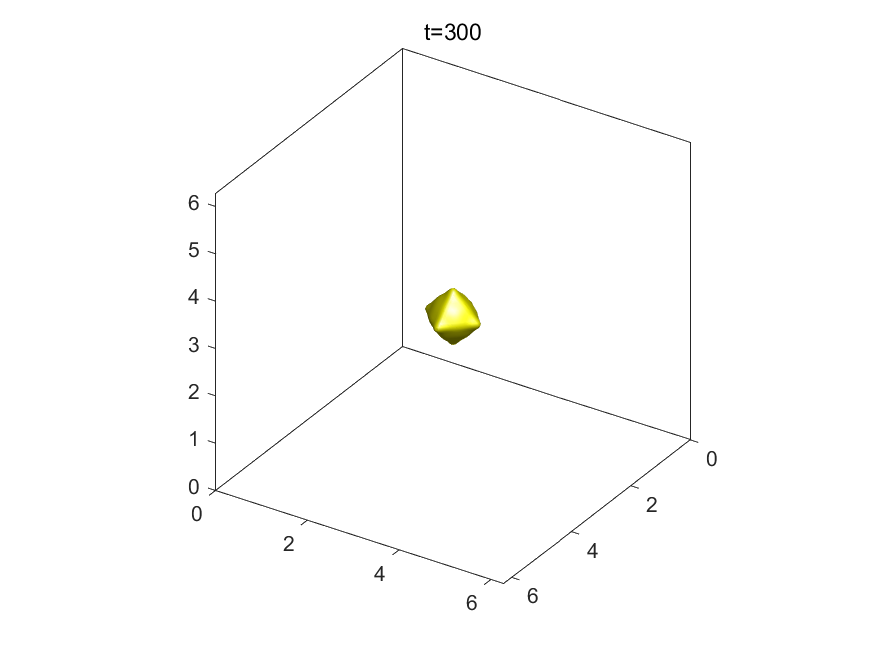}
\end{minipage}
\begin{minipage}{0.3 \textwidth}
\centering
\includegraphics[width=\textwidth]{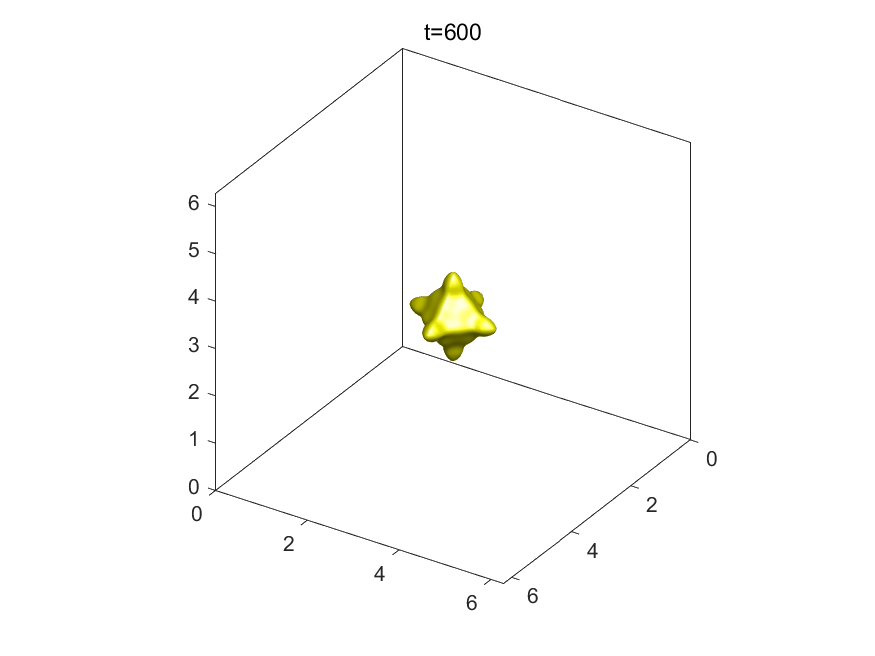}
\end{minipage}
\begin{minipage}{0.3 \textwidth}
\centering
\includegraphics[width=\textwidth]{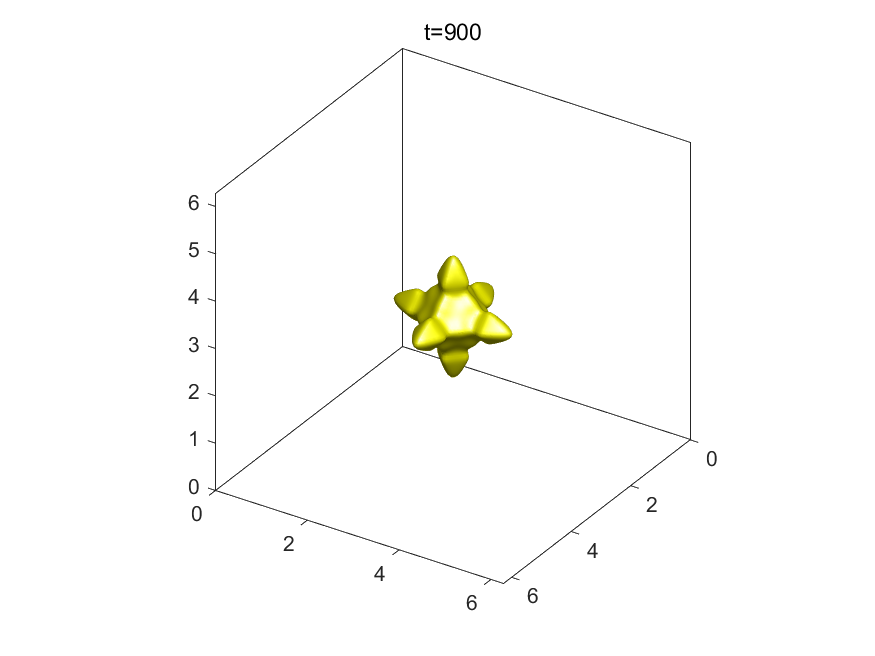}
\end{minipage}

\begin{minipage}{0.3 \textwidth}
\centering
\includegraphics[width=\textwidth]{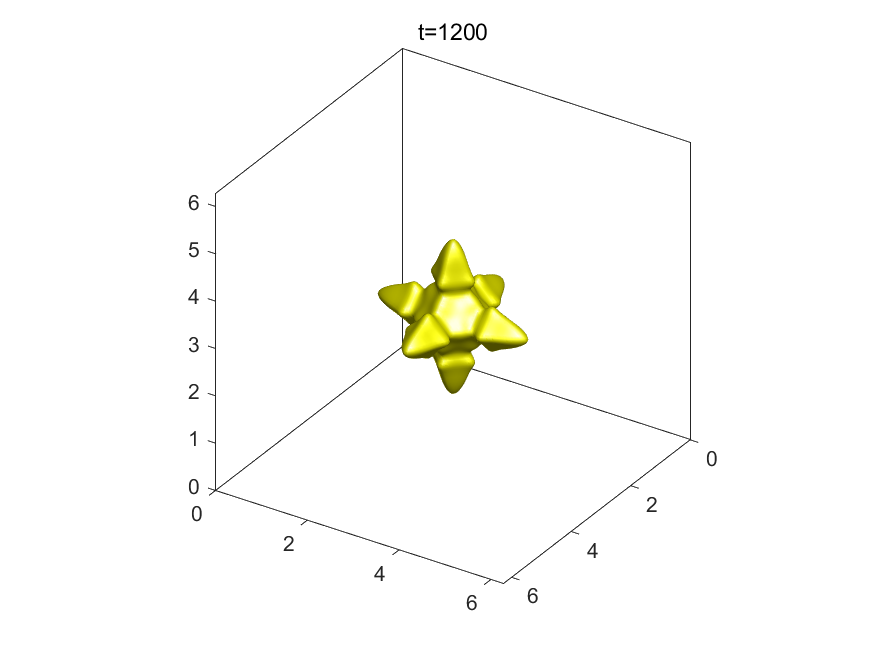}
\end{minipage}
\begin{minipage}{0.3 \textwidth}
\centering
\includegraphics[width=\textwidth]{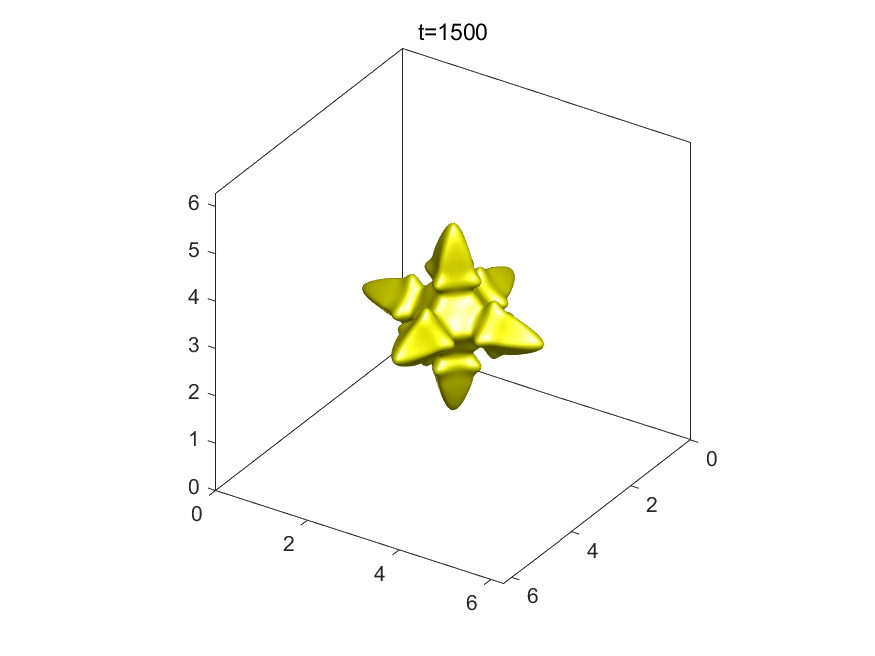}
\end{minipage}
\begin{minipage}{0.3 \textwidth}
\centering
\includegraphics[width=\textwidth]{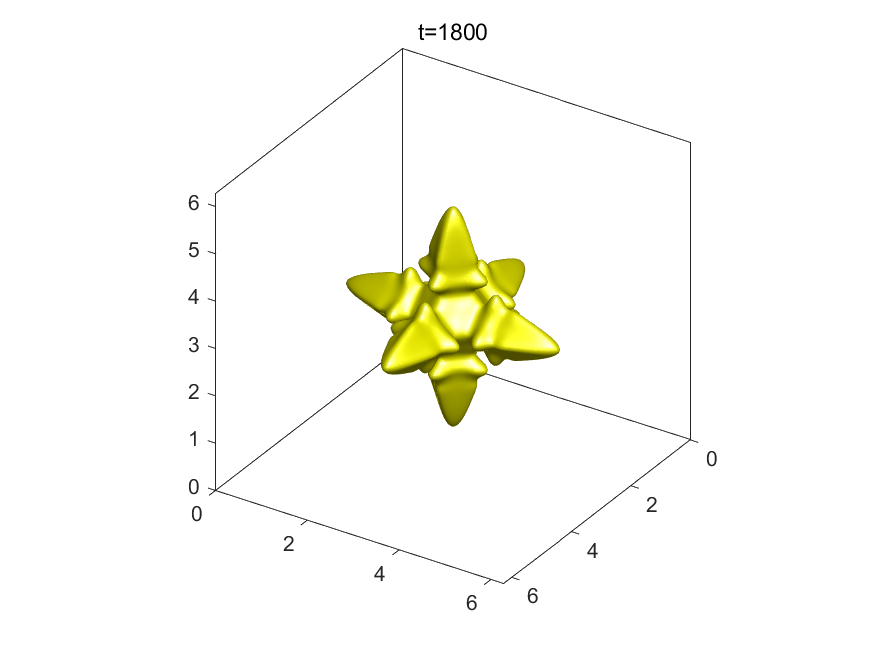}
\end{minipage}
	\caption{(Example \ref{exp.5}) Fourfold anisotropy dendritic crystal growth for $K=1.5$.}
\label{fourfoldK153D}
\end{figure}

\section{Conclusion}
We have designed a class of new time stepping schemes for the anisotropic dendritic crystal growth phase-field model, which is basically the coupling of
a phase-field equation and heat equation.
The underlying model involves
strong nonlinearity, anisotropic coefficient, and coupling
terms of different variables.
This makes construction of efficient numerical methods challenging.
The new proposed schemes are two-step method: in the first step, an intermediate solution is computed by using stabilized
BDF schemes of order up to 3 for both the
phase-field and heat equations.
In this stabilized BDF schemes, all nonlinear terms and coupling terms are
treated so that the intermediate solution can be obtained by solving two linear
elliptic equations.
The key for the success lies in the
stabilization step, which consists in correcting the intermediate phase-field and temperature solutions by using an auxiliary variable. This correction step
played a key role in stabilizing the schemes while keeping the expected
convergence orders.
In particular, in the correction step
a so-called generalized auxiliary variable with relaxation was introduced.
Compared to the existing techniques, our new correction algorithm needs
one less parameter, therefore easier to implement.
The stability property of the proposed schemes was established, while the convergence rate was examined through a series of numerical tests.
The computed numerical results demonstrated the efficiency and reliability of the proposed method.

Compared with the existing schemes for the dendritic crystal growth model,
the method proposed in the current paper are of higher order accuracy, fully decoupled, and
unconditionally energy stable, and cheaper in computation cost.
To be more specific,
compared with the most recent method in \cite{li2022new} which requires solving
four linear elliptic equations, our method only needs to solve two linear equations at each time step.

\bibliographystyle{plain}
\bibliography{v11}

\end{document}